\documentclass{amsart}%
\usepackage{amssymb}
\usepackage{amsfonts}
\usepackage{amsmath}
\usepackage[all,cmtip]{xy}
\usepackage{diagxy}
\usepackage{calrsfs}
\usepackage[margin=1.07in]{geometry}
\usepackage[backref]{hyperref}%
\usepackage{graphicx}

\newtheorem{theorem}{Theorem}
\newtheorem*{nonumtheorem}{Theorem}
\numberwithin{theorem}{subsection}
\theoremstyle{plain}
\newtheorem*{acknowledgement}{Acknowledgement}
\newtheorem{proposition}[theorem]{Proposition}

\newtheorem{corollary}[theorem]{Corollary}

\newtheorem{definition}[theorem]{Definition}
\newtheorem{example}[theorem]{Example}

\newtheorem*{nonnumdef}{Definition}
\newtheorem{lemma}[theorem]{Lemma}

\theoremstyle{remark}
\newtheorem{remark}[theorem]{Remark}
\newtheorem*{conventions}{Conventions}
\newtheorem*{outline}{Outline}
\numberwithin{equation}{section}
\begin{document}
\title[Hochschild coniveau]{Hochschild coniveau spectral sequence\linebreak and the Beilinson residue}
\author{Oliver Braunling\qquad Jesse Wolfson}
\address{Department of Mathematics, Universit\"{a}t Freiburg, Germany}
\email{oliver.braeunling@math.uni-freiburg.de}
\address{Department of Mathematics, University of Chicago, USA}
\email{wolfson@math.uchicago.edu}
\thanks{O.B.\ was supported by DFG GK 1821 \textquotedblleft Cohomological
Methods in Geometry\textquotedblright . J.W.\ was partially supported by an
NSF Post-doctoral Research Fellowship under Grant No.\ DMS-1400349. Both
authors thank the University of Chicago RTG for funding a research visit
under NSF Grant No. DMS-1344997.}

\begin{abstract}
We develop the Hochschild analogue of the coniveau spectral sequence and the
Gersten complex. Since Hochschild homology does not have d\'{e}vissage or
$\mathbf{A}^{1}$-invariance, this is a little different from the $K$-theory
story. In fact, the rows of our spectral sequence look a lot like the Cousin
complexes in \emph{Residues \&\ Duality}. Note that these are for coherent
cohomology. We prove that they agree by an `HKR isomorphism with
supports'.\smallskip\newline Using the close ties of Hochschild homology to
Lie algebra homology, this gives residue maps in Lie homology, which we show
to agree with those \`{a} la Tate--Beilinson.

\end{abstract}
\maketitle

The coniveau spectral sequence and the Gersten complex originally arose in
algebraic $K$-theory. But if one replaces in its construction $K$-theory by
Hochschild homology, everything goes through. Still, it appears that this
analogue has not really been studied much so far (if at all). We discuss it in
this paper, building on the work of B. Keller \cite{MR1647519} and P. Balmer
\cite{MR2439430}.

Let $X/k$ be a Noetherian scheme. Write $HH^{x}$ for the Hochschild spectrum
with support in a point $x$. Our Hochschild--Cousin complex will take the form%
\[
\cdots\longrightarrow\coprod_{x\in X^{0}}HH_{-q}^{x}(\mathcal{O}%
_{X,x})\longrightarrow\coprod_{x\in X^{1}}HH_{-q-1}^{x}(\mathcal{O}%
_{X,x})\longrightarrow\cdots\longrightarrow\coprod_{x\in X^{n}}HH_{-q-n}%
^{x}(\mathcal{O}_{X,x})\longrightarrow\cdots
\]
and appears as the rows in the $E_{1}$-page of a corresponding Hochschild
coniveau spectral sequence,%
\begin{equation}
\left.  ^{HH}E_{1}^{p,q}\right.  :=\coprod_{x\in X^{p}}HH_{-p-q}%
^{x}(\mathcal{O}_{X,x})\Rightarrow HH_{-p-q}(X)\text{.}\label{aj1}%
\end{equation}
Some things are different from $K$-theory: As Hochschild homology does
\textit{not} satisfy d\'{e}vissage, one cannot replace the $HH^{x}$ by
Hochschild homology of the residue field.

There is a similar, but much older complex: The \textsl{coherent cohomology}
Cousin complex from \emph{Residues \&\ Duality} \cite{MR0222093}. It has the
form%
\begin{equation}
\cdots\longrightarrow\coprod_{x\in X^{0}}H_{x}^{q}(X,\mathcal{F}%
)\longrightarrow\coprod_{x\in X^{1}}H_{x}^{q+1}(X,\mathcal{F})\longrightarrow
\cdots\longrightarrow\coprod_{x\in X^{n}}H_{x}^{q+n}(X,\mathcal{F}%
)\longrightarrow\cdots\text{,}\label{aj2}%
\end{equation}
where $H_{x}^{\bullet}$ denotes (coherent) local cohomology of a coherent
sheaf $\mathcal{F}$ with support in a point $x$. It also arises as a row in
the $E_{1}$-page of a spectral sequence $\left.  ^{Cous}E_{1}^{p,q}%
(\mathcal{F})\right.  \Rightarrow H^{p+q}(X,\mathcal{F})$. We prove that if
$X/k$ is smooth and $\mathcal{F}:=\Omega_{X/k}^{\ast}$, this complex is
canonically isomorphic to our Hochschild--Cousin complex. We do this by a
Hochschild--Kostant--Rosenberg (HKR) isomorphism with supports:

\begin{nonumtheorem}
[HKR with supports]Let $k$ be a field, $R$ a smooth $k$-algebra and
$t_{1},\ldots,t_{n}$ a regular sequence. Then there is a canonical isomorphism%
\[
H_{(t_{1},\ldots,t_{n})}^{n}(R,\Omega^{n+i})\overset{\sim}{\longrightarrow
}HH_{i}^{(t_{1},\ldots,t_{n})}(R)\text{.}%
\]
For $n=0$, this becomes the classical HKR isomorphism. On the left,
$H_{I}^{\ast}$ refers to (coherent) local cohomology. On the right hand side,
$HH_{\ast}^{I}$ refers to the Hochschild homology of the category of
$I$-supported perfect complexes.
\end{nonumtheorem}

Although not being spelled out in this form, this is a consequence of the much
more general theory due to Keller \cite{MR1647519}. We shall need this
explicit formulation, and give a very elementary proof. This will be Prop.
\ref{marker_Prop_HKRWithSupport}. We feel that this straight-forward extension
of the standard HKR isomorphism would deserve to be much more widely known.

\begin{nonumtheorem}
Let $k$ be a field and $X/k$ a smooth scheme. For every integer $q$, the
$q$-th row on the $E_{1}$-page of the Hochschild coniveau spectral sequence is
isomorphic to the zero-th row of the $E_{1}$-page of the \textsl{coherent
cohomology} Cousin coniveau spectral sequence of $\Omega^{-q}$. That is: For
every integer $q$,\ there is a canonical isomorphism of chain complexes%
\[
\left.  ^{HH}E_{1}^{\bullet,q}\right.  \overset{\sim}{\longrightarrow}\left.
^{Cous}E_{1}^{\bullet,0}(\Omega^{-q})\right.  \text{.}%
\]
Entry-wise, this isomorphism is induced from the\ HKR isomorphism with supports.
\end{nonumtheorem}

This will be Theorem \ref{marker_Thm_ComparisonOfRowsOnEOnePage}. As all the
other rows on the right-hand side turn out to vanish, one can re-package this
claim as follows:
\[
\left.  ^{HH}E_{1}^{\bullet,\bullet}\right.  \overset{\sim}{\longrightarrow
}\left.  ^{Cous}E_{1}^{\bullet,\bullet}(\tilde{\Omega})\right.  \qquad
\text{with}\qquad\tilde{\Omega}:=\bigoplus_{m}\Omega^{m}[m]\text{.}%
\]

\begin{nonumtheorem}
Let $k$ be a field of characteristic zero and $X/k$ a smooth scheme. Then the
Hochschild coniveau spectral sequence of line \ref{aj1} degenerates on the
$E_{2}$-page.
\end{nonumtheorem}

This kind of behaviour is exactly the same as for a spectral sequence due to
C. Weibel \cite{MR1469140}, which also converges to the Hochschild homology of
$X$, starting from $\left.  ^{Weibel}E_{2}^{p,q}\right.  =H^{p}(X,\mathcal{HH}%
_{-q})$, where $\mathcal{HH}_{-q}$ is Zariski-sheafified Hochschild homology.
However, his spectral sequence is constructed in a different fashion
(hypercohomology spectral sequence) and does not come with a description of
the $E_{1}$-page as we give it in Equation \ref{aj1}.\medskip

The Chern character from algebraic $K$-theory, $K\rightarrow HH$, then induces
a morphism of coniveau spectral sequences, and by the above comparison to
\emph{Residues \&\ Duality}, to (coherent) local cohomology. On the $E_{1}%
$-page, these maps are induced by pointwise maps:

\begin{nonnumdef}
[Chern character with supports]If $X/k$ is smooth and $x\in X$ any scheme
point, then we construct a map (\S \ref{subsect_ChernCharWithSupps})%
\[
K_{m}(\kappa(x))\longrightarrow H_{x}^{p}(X,\Omega^{p+m})
\]
with $p:=\operatorname*{codim}_{X}\overline{\{x\}}$, inducing maps $\left.
^{K}E_{1}^{p,q}\right.  \longrightarrow\left.  ^{Cous}E_{1}^{p,0}(\Omega
^{-q})\right.  $, where $\left.  ^{K}E_{1}^{p,q}\right.  $ is the usual
coniveau spectral sequence for algebraic $K$-theory, as in Quillen
\cite{MR0338129}.
\end{nonnumdef}

See \S \ref{subsect_ChernCharWithSupps} for the actual definition.\medskip

So far, we work in analogy to algebraic $K$-theory. In the second part of the
paper, we focus on a completely different issue. The coherent Cousin complex,
line \ref{aj2}, appears in \emph{Residues \&\ Duality} \cite{MR0222093} as an
injective resolution $-$ and is usually looked at from a quite different
perspective: If $X/k$ is a smooth proper variety of pure dimension $n$ with
$f:X\rightarrow\operatorname*{Spec}k$ the structure map, the shriek pullback
is known concretely: $f^{!}\mathcal{O}_{k}\cong\Omega_{X/k}^{n}[n]$.
Grothendieck duality then stems from the adjunction $f_{\ast}\leftrightarrows
f^{!}$ and the co-unit map $\operatorname*{Tr}_{f}:f_{\ast}f^{!}%
\mathcal{O}_{k}\rightarrow\mathcal{O}_{k}$, which induces%
\[
\operatorname*{Tr}\nolimits_{f}:H^{n}(X,\Omega_{X/k}^{n})\longrightarrow
k\text{.}%
\]
The coherent Cousin complex then provides an injective resolution of
$\Omega_{X/k}^{n}$; even more than that it is a so-called dualizing complex.
Although we will not explain this here, the map $\operatorname*{Tr}%
\nolimits_{f}$ can be unravelled explicitly in terms of (higher) residues.
Tate \cite{MR0227171} and Beilinson \cite{MR565095} have proposed an approach
to residues based on higher ad\`{e}les of a scheme. Ad\`{e}les provide a
further resolution of the sheaf $\Omega_{X/k}^{n}$ and give rise to a certain
Lie homology map $H_{n+1}^{\operatorname*{Lie}}(-,-)\rightarrow k$ which turns
out to give an explicit description of these residue maps. The duality theory
aspect of the ad\`{e}les (in dimension $>1$) is due to Yekutieli
\cite{MR1213064}, \cite{MR1652013}.

In \cite{olliTateRes} it was shown that this approach to the residue can also
be re-phrased in terms of the Hochschild homology of certain (non-commutative)
algebras defined from the ad\`{e}les. Along with the first part of the paper,
it seems more than tempting to believe that this should allow us to phrase the
Tate--Beilinson residue in terms of differentials in the Hochschild--Cousin
complex. We show that this is indeed the case:

\begin{nonumtheorem}
[Main Comparison Theorem]The Tate--Beilinson residue in the Lie homology of
ad\`{e}les \cite{MR0227171}, \cite{MR565095} can be expressed in terms of the
differentials of our Hochschild--Cousin complex: Specifically, the
Tate--Beilinson Lie homology residue symbol%
\[
\Omega_{\operatorname*{Frac}L_{n}/k}^{n}\longrightarrow H_{n+1}%
^{\operatorname*{Lie}}((A_{n})_{Lie},k)\overset{\phi_{Beil}}{\longrightarrow}k
\]
(as defined in \cite[\S 1, Lemma, (b)]{MR565095}) also agrees with%
\[
\Omega_{\operatorname*{Frac}L_{n}/k}^{n}\longrightarrow HH_{n}^{\eta_{0}%
}(L_{n})\longrightarrow HH_{n}^{\eta_{0}}(C_{0})\longrightarrow HH_{n}%
(A_{n})\overset{\phi_{HH}}{\longrightarrow}k\text{.}%
\]
(see Theorem \ref{thm_MainOfLastPart} for details and notation)
\end{nonumtheorem}

This statement is intentionally vague since we do not want to introduce the
necessary notation and background on ad\`{e}les of schemes in this
introduction. This result will be stated in precise form in
\S \ref{subsect_ComparisonTateBeilResidue}, along with a review of the
ad\`{e}le theory. In coarse strokes, we paint the following picture:%
\[%
\begin{array}
[c]{ccccc}%
\begin{array}
[c]{c}%
\text{coherent Cousin}\\
\text{complex}%
\end{array}
& \longleftrightarrow &
\begin{array}
[c]{c}%
\text{Hochschild}\\
\text{Cousin complex}%
\end{array}
& \longleftrightarrow & \text{ad\`{e}les of }\Omega^{n}\\
\uparrow &  & \uparrow &  & \uparrow\\
\text{local coherent} &  & \text{Hochschild homology} &  &
\text{non-commutative}\\
\text{cohomology} &  & \text{with supports} &  & \text{Hochschild homology}\\
&  &  &  & \downarrow\\
&  &  &  & \text{Lie homology}%
\end{array}
\]
It should be said that J. Lipman had the idea to use Hochschild homology for
residues already in 1987 \cite{MR868864}. However, his construction is very
different from ours. He constructs the residue manually, while we use Keller's
analogue of the localization sequence in $K$-theory \cite{MR1667558}, which
only became available in 1999, along with the particularly flexible technique
for coniveau due to Balmer \cite{MR2439430}, which is even more
recent.\medskip

\begin{outline}
We proceed as follows: In \S \ref{marker_Sect_ReviewDefsOfHochschildHomology}
we recall the necessary material on Hochschild homology. In
\S \ref{sect_HKR_with_supports} we prove the HKR isomorphism with supports. In
\S \ref{subsect_CubicallyDecomposedAlgebrasAndTheirDerivates} we give an
independent treatment of the Hochschild residue \`{a} la \cite{olliTateRes}.
In \S \ref{sect_TateCategories} we provide the necessary material on Tate
categories. These categories provide the crucial bridge to transport
Hochschild homology from a classical geometric to an ad\`{e}le perspective. In
\S \ref{sect_RelativeMoritaTheory} we develop a `\textsl{relative Morita
theory}'. If an exact category $\mathcal{C}$ happens to be equivalent to a
projective module category, say $\mathcal{C}\overset{\sim}{\rightarrow}%
P_{f}(E)$ for an algebra $E$, we will need to understand how such a
presentation changes if we consider a fully exact sub-category $\mathcal{C}%
^{\prime}\hookrightarrow\mathcal{C}$, or a quotient exact category
$\mathcal{C}/\mathcal{C}^{\prime}$, provided $\mathcal{C}^{\prime}$ is left or
right $s$-filtering. This might be of independent interest. In
\S \ref{sect_TheBeilRes} we combine all these tools to establish a commutative
square relating the Beilinson--Tate residue with boundary maps in Keller's
localization sequence for Hochschild homology.
\end{outline}

\begin{acknowledgement}
We thank Shane\ Kelly and Charles Weibel for their comments on an earlier
version of this text. Moreover, we heartily wish to thank Michael Groechenig.
This paper does not only benefit profoundly from previous joint papers with
him, but also from numerous conversations going well beyond what exists in
writing. The first part of the paper owes great intellectual debt to the works
of P. Balmer and B. Keller.
\end{acknowledgement}

\section{\label{marker_Sect_ReviewDefsOfHochschildHomology}The many
definitions of Hochschild homology}

Let us quickly survey what we understand as Hochschild homology. There are a
large number of definitions which apply in greater or smaller generality. We
will quickly sketch the transition from the classical definition of Hochschild
up to the definition for dg categories of Keller.\medskip

For $k$ a commutative ring and $A$ a flat $k$-algebra one classically defines
a complex $(C_{\bullet},b)$ by $C_{i}(A)=A^{\otimes i+1}$,%
\begin{align}
b(a_{0}\otimes\cdots\otimes a_{i}):=  & \sum_{j=0}^{i-1}(-1)^{j}a_{0}%
\otimes\cdots\otimes a_{j}a_{j+1}\otimes\cdots\otimes a_{i}\label{lHHC7}\\
& +(-1)^{i}a_{i}a_{0}\otimes a_{1}\otimes\cdots\otimes a_{i-1}\nonumber
\end{align}
and its homology is the \emph{Hochschild homology} of $A$. Philosophically,
this is conveniently viewed as a concrete complex quasi-isomorphic to a
certain derived tensor product, namely%
\begin{equation}
C_{\bullet}\sim A\otimes_{A\otimes_{k}A^{op}}^{\mathbf{L}}A\text{,}%
\label{lHHC9}%
\end{equation}
but it is the former definition which led to Mitchell's generalization to
categories \cite{MR0294454}: For $\mathcal{A}$ a $k$-linear small category
such that all $\operatorname*{Hom}\nolimits_{k}(-,-)$ are flat $k$-modules,
one defines%
\begin{equation}
C_{i}(\mathcal{A}):=\coprod\operatorname*{Hom}(X_{i},X_{0})\otimes
_{k}\operatorname*{Hom}(X_{i-1},X_{i})\otimes_{k}\cdots\otimes_{k}%
\operatorname*{Hom}(X_{0},X_{1})\text{,}\label{lHHC8}%
\end{equation}
where the coproduct runs over all $(i+1)$ tuples of objects in $\mathcal{A}$.
A differential $b$ can be defined by the same formula as before, this time
instead of multiplying elements of $A$, one composes the respective morphisms.
In order to stress the analogy with Equation \ref{lHHC7} the reader might at
first sight prefer to use an indexing starting with $\operatorname*{Hom}%
(X_{0},X_{1})\otimes\ldots$, but this comes with the disadvantage that in
composing $X_{0}\overset{a_{0}}{\rightarrow}X_{1}$ and $X_{1}\overset{a_{1}%
}{\rightarrow}X_{2}$ the composition is $a_{1}a_{0}$ and not $a_{0}a_{1}$.
Thus, in order to use the same formula for $b$, one has to use a reversed numbering.

\begin{remark}
\label{marker_CyclicNerveForCategoryOfARing}If we regard a ring $A$ as a
category $\mathbf{A}$ with one object \textquotedblleft$A$\textquotedblright%
\ and $\operatorname*{Hom}\nolimits_{\mathbf{A}}(A,A):=A$, the classical
definition of Equation \ref{lHHC7} literally becomes a special case of
Mitchell's categorical definition.
\end{remark}

However, both constructions are only the `correct' ones in very special cases.
For example, for $A$ a general $k$-algebra, i.e. not necessarily flat over
$k$, one works instead with $C_{i}(A):=A_{\bullet}\otimes_{k}\cdots\otimes
_{k}A_{\bullet}$, where $A_{\bullet}\rightarrow A$ is a flat resolution of $A$
and adapts the definition of the differential to deal with dg algebras, as
done by Keller in \cite{MR1492902}. In a similar spirit, for $\mathcal{A}$ a
$k$-linear dg category, one replaces the definition of Equation \ref{lHHC8} by
a version where the complexes $\operatorname*{Hom}(-,-)$ get replaced by flat
resolutions. This leads to the definition of Keller that we shall also use in
the present paper; the flat case suffices for our purposes:

\begin{definition}
Let $k$ be a commutative ring and $\mathcal{A}$ a small $k$-linear dg-flat dg
category. In particular, all

\begin{itemize}
\item $\operatorname*{Hom}\nolimits_{\mathcal{A}}(X,Y)$ are $k$-flat dg
$k$-modules, and

\item the composition%
\[
\operatorname*{Hom}\nolimits_{\mathcal{A}}(Y,Z)\otimes_{k}\operatorname*{Hom}%
\nolimits_{\mathcal{A}}(X,Y)\longrightarrow\operatorname*{Hom}%
\nolimits_{\mathcal{A}}(X,Z)
\]
is a morphism of dg $k$-modules.

Then define for homogeneous morphisms $(a_{i},\ldots,a_{0})\in C_{i}%
(\mathcal{A})$ (as in Equation \ref{lHHC8})%
\begin{align*}
b(a_{i},\ldots,a_{0}):=  & \sum_{j=0}^{i-1}(-1)^{j}(a_{i},\ldots,a_{i}\circ
a_{i-1},\ldots,a_{0})\\
& +(-1)^{n+\sigma}(a_{0}\circ a_{i},a_{i-1},\ldots,a_{1})\text{,}%
\end{align*}
where $\sigma=(\deg a_{0})(\deg a_{1}+\cdots+\deg a_{i-1})$.
\end{itemize}
\end{definition}

See for example \cite[\S 3.2]{MR1647519} or \cite[\S 1.3]{MR1667558}. The
general version without flatness assumption is constructed in \cite[\S 3.9]%
{MR1667558}. We also remind the reader that for an exact category the category
of complexes itself does not reflect any datum of the exact structure, so that
the derived category of an exact category $\mathcal{E}$ has to be defined as
the Verdier quotient $D^{b}\mathcal{E}:=\mathcal{K}^{b}(\mathcal{E}%
)/\mathcal{A}c^{b}(\mathcal{E})$, where $\mathcal{K}^{b}(\mathcal{E})$ is the
triangulated category of bounded complexes in $\mathcal{E}$ modulo chain
homotopies and $\mathcal{A}c^{b}(\mathcal{E})$ the subcategory of acyclic
complexes, subtly depending on the exact structure. See \cite[\S 10]%
{MR2606234} for a very detailed excellent review.

As is suggested from the derived category of an exact category, the Hochschild
homology of $\mathcal{E}$ is then defined as follows:

\begin{definition}
[{Keller, \cite[\S 1.4]{MR1667558}}]\label{marker_DefHHKellerExactCats}Let
$\mathcal{E}$ be a flat $k$-linear exact category. Then its \emph{Hochschild
homology} is%
\[
HH(\mathcal{E}):=Cone\left(  C_{\bullet}\mathcal{A}c^{b}(\mathcal{E}%
)\longrightarrow C_{\bullet}\mathcal{C}^{b}(\mathcal{E})\right)  \text{,}%
\]
where $\mathcal{C}^{b}(\mathcal{E})$ is the dg category of bounded complexes
in $\mathcal{E}$ and $\mathcal{A}c^{b}(\mathcal{E})$ its dg subcategory of
acyclic complexes.
\end{definition}

In the present paper we will mostly be interested in the Hochschild homology
of perfect complexes with support. This category sadly does not fit into the
framework of exact categories. One may view perfect complexes either as a
stable $\infty$-category, a dg category or as a Waldhausen $1$-category. The
dg perspective leads directly to a very similar definition as before:

\begin{definition}
[{Keller, \cite[\S 4.3]{MR1647519}}]\label{marker_DefHHKellerPerfectComplexes}%
Let $X$ be a scheme and $Z$ a closed subset. Define the \emph{Hochschild
homology of} $X$ \emph{with support in} $Z$ by%
\[
HH^{Z}\left(  X\right)  :=Cone\left(  C_{\bullet}(\mathcal{A}%
c\operatorname*{Perf}X)\longrightarrow C_{\bullet}(\operatorname*{Perf}%
\nolimits_{Z}X)\right)  \text{,}%
\]
where $\operatorname*{Perf}\nolimits_{Z}X$ is the category of perfect
complexes on $X$ acyclic on $X-Z$, and $\mathcal{A}c\operatorname*{Perf}X$ is
the category of all acyclic perfect complexes. We write $HH\left(  X\right)
:=HH^{X}\left(  X\right)  $ for the variant without support
condition.\footnote{One might be tempted to prefer writing \textquotedblleft%
$HH_{Z}$\textquotedblright\ for the theory with support in $Z$, but it leads
to the impractical notation $HH_{Z,i}$. Also, for homology with closed support
(Borel--Moore), the superscript notation $H_{i}^{c}$ or $H_{i}^{cl}$ is
widespread. Ultimately, it remains a matter of taste, of course.}
\end{definition}

Instead of $\mathcal{A}c\operatorname*{Perf}X$ we could also write
$\operatorname*{Perf}\nolimits_{\varnothing}X$ of course; these are literally
the same categories. Finally, we should also discuss a sheaf perspective
\cite{MR1277141}, \cite[\S 2, end of page 59]{MR1390671}: Let $k$ be a field
now. For $X$ a $k$-scheme one can consider the presheaf of complexes of
$k$-modules%
\begin{equation}
U\mapsto C_{\bullet}(\mathcal{O}_{X}(U))\label{lmips4}%
\end{equation}
and let $\mathcal{C}_{\bullet}$ be its Zariski sheafification (Weibel denotes
it as $\mathcal{C}_{\ast}^{h}$ in \cite[\S 1]{MR1277141}). Note that
$\mathcal{O}_{X}(U)$ is a flat $k$-algebra, so for $C_{\bullet}$ one can use
the classical definition as in Equation \ref{lHHC7}. Unfortunately,
$\mathcal{C}_{\bullet}$ is not a quasi-coherent sheaf. However, its homology
turns out to be quasi-coherent.

\begin{theorem}
[Geller, Weibel]\label{Thm_GellerWeibelPropsOfHHSheaf}Let $X$ be a $k$-scheme.

\begin{enumerate}
\item The homology sheaves $\mathcal{HH}_{i}:=H_{i}(\mathcal{C}_{\bullet})$
are quasi-coherent.

\item The Zariski sheafification of $U\mapsto HH_{i}(\mathcal{O}_{X}(U))$
agrees with the sheaf $\mathcal{HH}_{i}$.

\item On each affine open $U\subseteq X$, one has canonical isomorphisms%
\[
H^{p}(U,\mathcal{HH}_{i})\cong\left\{
\begin{array}
[c]{ll}%
0 & \text{for }p\neq0\\
HH_{i}(\mathcal{O}_{X}(U)) & \text{for }p=0\text{.}%
\end{array}
\right.
\]

\item $\mathcal{HH}_{i}$ also makes sense as an \'{e}tale sheaf and
$H^{p}(X_{\acute{e}t},\mathcal{HH}_{i})\cong H^{p}(X_{Zar},\mathcal{HH}_{i})$.
\end{enumerate}
\end{theorem}

See \cite[Prop. 1.2]{MR1277141} for a discussion and \cite[Cor. 0.4]%
{MR1120653} for the proof. This is all we need for the present paper, but much
more is known, e.g. cdh descent for $X$ smooth \cite{MR2373359},
\cite{MR2415380}.

\begin{example}
\label{Example_Omega_Iso_To_HH_as_Zar_sheaves}The Zariski descent and the
Hochschild-Kostant-Rosenberg isomorphism imply that on a smooth $k$-scheme the
sheaves $\mathcal{HH}_{i}$ and $\Omega^{i}$ ($:=\Omega_{X/k}^{i}$) are isomorphic.
\end{example}

Weibel and Swan now define a version of Hochschild homology of a scheme via%
\[
HH_{i}^{Weibel}\left(  X\right)  :=H^{-i}(X_{Zar},\mathcal{C}_{\bullet
})\text{,}%
\]
where $H^{\ast}$ refers to the sheaf (hyper)cohomology of the sheaf of
complexes \cite[Eq. 1.1]{MR1277141}. Geller and Weibel show that for $X$
affine this agrees with the classical definition in terms of rings,
$HH_{i}^{Weibel}(X)\cong HH_{i}(\mathcal{O}_{X}(X))$, see \cite[Thm.
4.1]{MR1120653}. More generally, Keller established a beautiful theorem
linking this sheaf perspective with the categorical viewpoint.

\begin{theorem}
[Keller]Let $k$ be a field and $X$ a Noetherian separated $k$-scheme. Then
there is a canonical isomorphism $HH_{i}^{Weibel}\left(  X\right)
\overset{\sim}{\rightarrow}HH_{i}(X)$, where $HH_{i}(X)$ refers to the
Hochschild homology of perfect complexes as in Definition
\ref{marker_DefHHKellerPerfectComplexes}.
\end{theorem}

This is \cite[Thm. 5.2]{MR1647519}\footnote{Keller proves the result on the
level of mixed complexes. But Hochschild homology can be defined in terms of
mixed complexes, leading directly to the present formulation.}. Keller's paper
also provides details on the switch between two slightly different definitions
of the sheaf hypercohomology underlying Weibel's definition. Besides all this,
Equation \ref{lHHC9} suggests an entirely different definition of Hochschild
homology of a scheme, proposed by Swan \cite{MR1390671}. However, it turns out
to agree with the previous definition:

\begin{theorem}
[Yekutieli]Let $k$ be a field and $X$ a finite type $k$-scheme. Then there is
a quasi-isomorphism of sheaves $\mathcal{C}_{\bullet}\mathcal{O}_{X}%
\overset{\sim}{\rightarrow}\mathcal{O}_{X}\otimes_{\mathcal{O}_{X\times X}%
}^{\mathbf{L}}\mathcal{O}_{X}$.
\end{theorem}

See \cite[Prop. 3.3]{MR1940241}. In the same paper Yekutieli also constructed
an alternative complex $\widehat{\mathcal{C}}_{\bullet}$ of completed
Hochschild chains, which is itself quasi-coherent, suitably interpreted, and
not just quasi-coherent after taking homology. One can also define Hochschild
homology on the derived level, following C\u{a}ld\u{a}raru and Willerton
\cite{MR2657369}. Their paper also shows equivalence to Weibel's
approach.\medskip

Finally, we also need a completely different direction of generalization of
Hochschild homology: the case of rings \textsl{without units}. Formulations in
terms of modules or perfect complexes over non-unital rings appear to be very
subtle (but see work of Quillen \cite{MR1384911} and Mahanta \cite{MR2776888}%
), so we will not enter into the matter of setting up a categorical viewpoint,
and just stick to algebras.

\begin{conventions}
We shall reserve the word \emph{ring} for a commutative, unital associative
algebra. A ring morphism will always preserve the unit of multiplication. This
leaves us with the word \emph{algebra} whenever we need to work with more
general structures. For us an associative algebra $A$ will \textit{not} be
assumed to be unital. Likewise, we do \textit{not} require morphisms of
algebras to preserve a unit, even if it exists. As an example, note that this
implies that any one-sided or two-sided ideal $I\subseteq A$ is itself an
associative algebra and the inclusion $I\hookrightarrow A$ a morphism of algebras.
\end{conventions}

\begin{definition}
\label{def_LocallyBiUnital}The algebra $A$ is called

\begin{enumerate}
\item \emph{locally left unital} (resp. \emph{locally right unital}) if for
every finite subset $S\subseteq A$ there exists an element $e_{S}\in S$ such
that $e_{S}a=a$ (resp. $ae_{S}=a$) for all $a\in S$;

\item \emph{locally bi-unital} if it is both locally left unital and locally
right unital.
\end{enumerate}
\end{definition}

\begin{remark}
Locally bi-unital does not imply that we can find $e_{S}$ such that
$e_{S}a=a=ae_{S}$ holds for all $a$ in any finite subset $S\subseteq A$. It
makes no statement about the mutual relation of left- and right-units.
\end{remark}

If $A$ is a non-unital associative $k$-algebra, the definition of the complex
as in Equation \ref{lHHC7} still makes perfect sense. Nonetheless, it turns
out that this is not quite the right thing to do, a `correction term' is
required, as was greatly clarified and resolved by M. Wodzicki \cite{MR997314}%
: One defines the so-called bar complex $B_{\bullet}$ and with the cyclic
permutation operator $t$, and one forms the bi-complex $C_{i}%
^{\operatorname*{corr}}(A):=[C_{i}(A)\overset{1-t}{\longrightarrow}%
B_{i+1}(A)]$ (we will not define $B_{\bullet}$ or $t$ here, all details can be
found in \cite[\S 2, especially page 598 l. 5]{MR997314}). This complex turns
out to model a well-behaved theory of Hochschild homology even if $A$ is
non-unital.\ If $A$ is unital, $B_{\bullet}$ turns out to be acyclic so that
we recover the previous definition, but this works even more generally:

\begin{proposition}
[Wodzicki]\label{mai_Prop_Wodzicki_HUnitality}If $A$ is locally left unital
(or locally right unital), $B_{\bullet}$ is acyclic, so that the complex
$C_{\bullet}(A)$ models Hochschild homology.
\end{proposition}

This is \cite[Cor. 4.5]{MR997314}. Let us rephrase this: As long as we only
work with locally left or right unital associative algebras, we may just work
with the complex in line \ref{lHHC7} as the definition of Hochschild homology.
And this will be precisely the situation in this paper, so the reader may feel
free to ignore $B_{\bullet}$ and $C_{\bullet}^{\operatorname*{corr}}$ entirely.

We need one more ingredient: Suppose $A$ is a (possibly non-unital)
associative algebra and $I$ a two-sided ideal. Then we get an exact sequence
of associative algebras%
\begin{equation}
I\hookrightarrow A\twoheadrightarrow A/I\text{.}\label{lmta_1}%
\end{equation}

\begin{theorem}
[Wodzicki]\label{Thm_WodzickiLongExactSeqUsingExcision}Suppose we are given an
exact sequence as in Equation \ref{lmta_1}. If $A$ and $I$ are locally left
unital\ (or locally right unital), then there is a fiber sequence%
\[
HH(I)\longrightarrow HH(A)\longrightarrow HH(A/I)\longrightarrow+1\text{.}%
\]

\end{theorem}

\begin{proof}
This is just a special case of a more general formalism, we refer to the paper
\cite{MR997314} for the entire story, or the book \cite[\S 1.4]{MR1217970}. In
the case at hand we can proceed as follows: One gets a fiber sequence%
\[
HH(A,I)\longrightarrow HH(A)\longrightarrow HH(A/I)\longrightarrow+1\text{,}%
\]
where $HH(A,I)$ refers to relative Hochschild homology, which is just defined
as the cone, so the existence of this sequence is tautological. By Wodzicki's
Excision Theorem \cite[Theorem 3.1]{MR997314}, we get an equivalence
$HH(I)\overset{\sim}{\longrightarrow}HH(A,I)$.\ This is only true when $I$ and
$A$ are $H$-unital (in the sense of loc. cit.), as is guaranteed by our
assumptions and \cite[Cor. 4.5]{MR997314}.
\end{proof}

\subsection{\label{Sect_ConiveauFiltration}Coniveau filtration}

\subsubsection{\label{marker_Coniveau_ForCoherentCohomology}Coherent
cohomology with supports}

Let us first recall the construction of the coniveau spectral sequence in
sheaf cohomology. We briefly summarize some facts about local cohomology that
we shall need. A detailed presentation has been given by Hartshorne
\cite{MR0224620}, in a different format also in \cite[Ch. IV]{MR0222093} or
\cite[\S 3.1]{MR1804902}.

Let $X$ be a topological space. A subset $Z$ is called \emph{locally closed}
if it can be written as the intersection of an open and a closed subset.
Equivalently, a closed subset $Z\subseteq_{cl}V$ of $V\subseteq_{op}X$ an open
subset. For a sheaf $\mathcal{F}$ one defines a new sheaf%
\[
\underline{\Gamma_{Z}}\mathcal{F}(U):=\{s\in\mathcal{F}(U)\mid
\operatorname*{supp}s\subseteq Z\}\text{,}%
\]
the sheaf of sections with support in $Z$. Note that if $j:Z\hookrightarrow X
$ is an open subset, $\underline{\Gamma_{Z}}\mathcal{F}=j_{\ast}%
j^{-1}\mathcal{F}$. Moreover, $\underline{\Gamma_{Z}}$ is a left exact functor
from the category of abelian group sheaves on $X$ to itself. Right derived
functors exist, are denoted by $\mathcal{H}_{Z}^{p}\mathcal{F}:=\mathbf{R}%
^{p}\underline{\Gamma_{Z}}\mathcal{F}$, and called \emph{local cohomology}
sheaves \cite[\S 1]{MR0224620}. We write $H_{Z}^{p}(X,\mathcal{F})$ for its
global sections on $X$. Equivalently, these match the right derived functors
of the functor $\mathcal{F}\mapsto H^{0}(X,\underline{\Gamma_{Z}}\mathcal{F}%
)$. There is also a product%
\begin{equation}
\underline{\Gamma_{Z_{1}}}\mathcal{F\otimes}\underline{\Gamma_{Z_{2}}%
}\mathcal{G}\longrightarrow\underline{\Gamma_{Z_{1}\cap Z_{2}}}(\mathcal{F}%
\otimes\mathcal{G})\label{comalg_fact_D40}%
\end{equation}
for sheaves $\mathcal{F},\mathcal{G}$ and locally closed subsets $Z_{1},Z_{2}
$. We shall mainly need the following property:\ If $Z$ is a locally closed
subset, $Z^{\prime}\subseteq Z$ a closed subset, then $Z-Z^{\prime}$ is also a
locally closed subset in $X$ and there is a distinguished triangle%
\begin{equation}
\mathbf{R}\underline{\Gamma_{Z^{\prime}}}\mathcal{F}\longrightarrow
\mathbf{R}\underline{\Gamma_{Z}}\mathcal{F}\longrightarrow\mathbf{R}%
\underline{\Gamma_{Z-Z^{\prime}}}\mathcal{F}\longrightarrow
+1\label{comalg_fact_D32}%
\end{equation}
(see \cite[Lemma 1.8 or Prop. 1.9]{MR0224620}, also \cite[Ch. IV,
\textquotedblleft Variation 2\textquotedblright, p. 219]{MR0222093}). For
$Z:=X$ and $Z^{\prime}\subseteq X$ a closed subset, this specializes to
$\mathbf{R}\underline{\Gamma_{Z^{\prime}}}\mathcal{F}\longrightarrow
\mathcal{F}\longrightarrow j_{\ast}j^{-1}\mathcal{F}\longrightarrow+1$, where
$j:U\hookrightarrow X$ denotes the open immersion of the open complement
$U:=X\setminus Z^{\prime}$ (see \cite[Ch. IV, \textquotedblleft Variation
3\textquotedblright, p. 220]{MR0222093}). If $X$ is a scheme, one can say
quite a bit more:

\begin{lemma}
Suppose $X$ is a Noetherian scheme and $\mathcal{F}$ a quasi-coherent sheaf.

\begin{enumerate}
\item (\cite[Prop. 2.1]{MR0224620}) Then the $\mathcal{H}_{Z}^{p}\mathcal{F}$
are also quasi-coherent sheaves.

\item (\cite[Thm. 2.8]{MR0224620}) If $Z$ is a closed subscheme with ideal
sheaf $\mathcal{I}_{Z}$, there is a canonical isomorphism of quasi-coherent
sheaves, functorial in $\mathcal{F}$,%
\begin{equation}
\underrightarrow{\operatorname*{colim}}_{\ell}\,\mathcal{E}xt_{\mathcal{O}%
_{X}}^{p}(\mathcal{O}_{X}/\mathcal{I}_{Z}^{\ell},\mathcal{F})\overset{\sim
}{\longrightarrow}\mathcal{H}_{Z}^{p}\mathcal{F}\text{.}%
\label{comalg_fact_D31}%
\end{equation}

\end{enumerate}
\end{lemma}

\begin{lemma}
[dual \textquotedblleft dimension axiom\textquotedblright]%
\label{COA_LocalCohomVanishesIfDegreeExceedsIdealGeneratorNumber}If $R$ is a
ring and $I=(f_{1},\ldots,f_{r})$. Then for every $R$-module $M$, we have
$H_{I}^{p}(R,M)=0$ for $p>r$.
\end{lemma}

See for example \cite[Cor. 7.14]{MR2355715}. We will frequently use the
following property, often allowing us to reduce to local rings:

\begin{lemma}
[{Excision, \cite[Prop. 1.3]{MR0224620}}]\label{COA_Lemma_LocalCohomExcision}%
Let $Z$ be locally closed in $X$, $V\subseteq X$ open so that $Z\subseteq
V\subseteq X$. Then there is a canonical isomorphism $H_{Z}^{p}(X,\mathcal{F}%
)\overset{\sim}{\longrightarrow}H_{Z}^{p}(V,\mathcal{F}\mid_{V})$, induced by
the pullback of sections along the open immersion.
\end{lemma}

\begin{lemma}
[{\cite[Prop. 5.9]{MR0224620}}]%
\label{lemma_CompletionPreservesLocalCohomology}Let $R$ be a Noetherian ring,
$I^{\prime}\subseteq I$ ideals and $M$ a finitely generated $R$-module. Then
there are canonical isomorphisms $H_{I}^{p}(\operatorname*{Spec}%
R,M)\overset{\sim}{\rightarrow}H_{\hat{I}}^{p}(\operatorname*{Spec}\widehat
{R},\widehat{M})\mid_{R}$, where $\widehat{(-)}$ in both cases refers to the
$I^{\prime}$-adic completion, so that $H_{\hat{I}}^{p}(\widehat{M})$ is an
$\widehat{R}$-module.
\end{lemma}

Next, one builds the Cousin complex in \textsl{coherent cohomology}. Let
$Z^{p}$ denote a closed subset of $X$ with $\operatorname*{codim}%
\nolimits_{X}Z^{p}\geq p$. One can read the colimit (under inclusion) under
all such,%
\begin{equation}
F_{p}H^{i}(X,\mathcal{F}):=\underset{Z^{p}}{\underrightarrow
{\operatorname*{colim}}}\operatorname*{im}(H_{Z^{p}}^{i}(X,\mathcal{F}%
)\rightarrow H^{i}(X,\mathcal{F}))\text{,}\label{comalg_fact_D34b}%
\end{equation}
as a filtration of the cohomology of the sheaf $\mathcal{F}$. Taking the
underlying filtered complex spectral sequence, one arrives at the
\textquotedblleft Cousin coniveau spectral sequence\textquotedblright, due to Grothendieck:

\begin{proposition}
[{\cite[Ch. IV]{MR0222093}}]This filtration induces a convergent spectral
sequence with%
\[
\left.  ^{Cous}E_{1}^{p,q}\right.  (\mathcal{F}):=\coprod_{x\in X^{p}}%
H_{x}^{p+q}(X,\mathcal{F})\Rightarrow H^{p+q}(X,\mathcal{F})\text{.}%
\]

\end{proposition}

The rows of the $E_{1}$-page read%
\begin{equation}
0\longrightarrow\coprod_{x\in X^{0}}H_{x}^{q}(X,\mathcal{F})\overset
{d}{\longrightarrow}\coprod_{x\in X^{1}}H_{x}^{q+1}(X,\mathcal{F})\overset
{d}{\longrightarrow}\cdots\overset{d}{\longrightarrow}\coprod_{x\in X^{n}%
}H_{x}^{q+n}(X,\mathcal{F})\longrightarrow0\label{comalg_fact_D35}%
\end{equation}
(this is the $q$-th row, concentrated columnwise in the range $0\leq p\leq n$
for $n:=\dim X$). The differential $d$ agrees with the upward arrow in the
following Diagram \ref{comalg_fact_D34}: We get a long exact sequence from
Equation \ref{comalg_fact_D32} and replicating suitable excerpts twice, we get
the rows of the diagram%
\begin{equation}%
\bfig\node a(0,0)[H_{Z^{p+1}}^{i}(X,\mathcal{F})]
\node b(1000,0)[H_{Z^{p}}^{i}(X,\mathcal{F})]
\node c(2000,0)[\coprod_{x\in X^{p}}H_{x}^{i}(X,\mathcal{F})]
\node d(3000,0)[H_{Z^{p+1}}^{i+1}(X,\mathcal{F})]
\node A(0,500)[H_{Z^{p+2}}^{i+1}(X,\mathcal{F})]
\node B(1000,500)[H_{Z^{p+1}}^{i+1}(X,\mathcal{F})]
\node C(2000,500)[\coprod_{x\in X^{p+1}}H_{x}^{i+1}(X,\mathcal{F})]
\node D(3000,500)[H_{Z^{p+2}}^{i+2}(X,\mathcal{F})]
\arrow[a`b;]
\arrow[b`c;]
\arrow[c`d;{\partial}]
\arrow[A`B;]
\arrow[B`C;]
\arrow[C`D;{\partial}]
\arrow[d`B;]
\arrow[c`C;]
\efig
\label{comalg_fact_D34}%
\end{equation}
The leftward diagonal arrow is just the identity morphism. Define the upward
arrow to be the composition. It is precisely the map $d$. We refer to
\cite{MR0222093} for more details. For \textit{local} Cohen-Macaulay schemes
we are in the pleasant situation that the complex in Equation
\ref{comalg_fact_D35} is exact. More precisely:

\begin{proposition}
Suppose $X=\operatorname*{Spec}R$ for $R$ a Noetherian Cohen-Macaulay local
ring. Then the sequence in Equation \ref{comalg_fact_D35} is exact. In
particular, if we read its entries as sheaves, i.e.%
\[
U\mapsto\coprod_{x\in U^{p}}H_{x}^{p+q}(X,\mathcal{O}_{X})\qquad\text{(for
}U\subseteq X\text{ Zariski open),}%
\]
the complex in Equation \ref{comalg_fact_D35} provides a flasque resolution of
the sheaf $U\mapsto\mathcal{H}^{q}(U,\mathcal{O}_{X})$.
\end{proposition}

This is a special case of \cite[Ch. IV, Prop. 2.6]{MR0222093}, in a mild
variation of \cite[Ch. IV, Example on p. 239]{MR0222093}.

\begin{corollary}
\label{marker_Cor_ExistenceCousinResolution}For $X$ Noetherian and
Cohen-Macaulay, suppose $\mathcal{F}$ is a coherent sheaf. Then there is a
flasque resolution of the sheaf $\mathcal{F}$, namely%
\[
0\longrightarrow\mathcal{F}\longrightarrow\underline{\coprod_{x\in X^{0}}%
H_{x}^{0}(X,\mathcal{F})}\longrightarrow\cdots\longrightarrow\underline
{\coprod_{x\in X^{n}}H_{x}^{n}(X,\mathcal{F})}\longrightarrow0\text{.}%
\]
This is known as the \emph{Cousin resolution} of $\mathcal{F}$. If $X$ is
Gorenstein and $\mathcal{F}$ locally free, this is an injective resolution of
$\mathcal{F}$.
\end{corollary}

Since we shall mostly work with smooth schemes, the weaker Gorenstein and
Cohen-Macaulay conditions are usually implied. In the literature one often
allows more general filtrations than those by codimension, but we have no use
for this increase in flexibility, see however \cite[Ch. III]{MR1804902} or
\cite{MR2439430}.

\subsubsection{\label{marker_Sect_HochschildHomologyWithSupports}Hochschild
homology with supports}

We will now repeat the story of \S \ref{marker_Coniveau_ForCoherentCohomology}
for the Hochschild homology of categories of perfect complexes with support,
see \S \ref{marker_Sect_ReviewDefsOfHochschildHomology} for the
definition.\ In principle, we shall do precisely the same constructions, but
the inner machinery is of quite a different type. This has not so much to do
with perfect complexes, but rather with a very different homological
perspective. Whereas we based the last section on the distinguished triangle%
\begin{equation}
\mathbf{R}\underline{\Gamma_{Z^{\prime}}}\mathcal{F}\longrightarrow
\mathbf{R}\underline{\Gamma_{Z}}\mathcal{F}\longrightarrow\mathbf{R}%
\underline{\Gamma_{Z-Z^{\prime}}}\mathcal{F}\longrightarrow+1\text{,}%
\label{l414}%
\end{equation}
regarding just the cohomology with supports of coherent sheaves, we will now
replace this by the localization sequence of categories%
\[
\operatorname*{Perf}\nolimits_{Z^{\prime}}X\longrightarrow\operatorname*{Perf}%
\nolimits_{Z}X\longrightarrow\operatorname*{Perf}\nolimits_{Z-Z^{\prime}%
}X\longrightarrow+1\text{.}%
\]
Just as the above sequence induces a long exact sequence in cohomology, the
latter induces a long exact sequence in the Hochschild homology of these
categories (and in fact, just as well for their algebraic $K$-theory, cyclic
homology, etc.).

The constructions of this section play a fundamental and classical r\^{o}le in
algebraic $K$-theory and originate essentially from Quillen \cite{MR0338129}.
We adapt them to Hochschild homology, however. In order to do so, we use a
particularly strong variant of the construction due to Balmer \cite{MR2439430}.

Let us briefly recall the concept of perfect complexes, following
Thomason-Trobaugh \cite[\S 2]{MR1106918}: Let $X$ be a scheme. A complex
$\mathcal{F}_{\bullet}$ of $\mathcal{O}_{X}$-module sheaves is called
\emph{perfect} if it is locally quasi-isomorphic to a bounded complex of
vector bundles. This definition goes back to \cite[Exp. I]{MR0354655}.

For $f:X\rightarrow Y$ a morphism, the total left derived functor
$\mathbf{L}f^{\ast}:D^{-}(\mathcal{M}od_{\mathcal{O}_{Y}})\rightarrow
D^{-}(\mathcal{M}od_{\mathcal{O}_{X}})$ of the pullback $f^{\ast}$ preserves
perfection so that there is a pullback functor $f^{\ast}:\operatorname*{Perf}%
(Y)\rightarrow\operatorname*{Perf}(X)$. If $f$ is flat, this functor is
literally just the entrywise pullback of a perfect complex (see
\cite[\S 2.5.1]{MR1106918}).

For $f:X\rightarrow Y$ a proper and perfect\footnote{e.g. smooth morphisms or
regular closed immersions. A general closed immersion need not be perfect, in
particular a general finite morphism need not be perfect. Any finite type
morphism to a smooth scheme is perfect.} morphism of Noetherian schemes, the
total right derived functor $\mathbf{R}f_{\ast}:D^{-}(\mathcal{M}%
od_{\mathcal{O}_{X}})\rightarrow D^{-}(\mathcal{M}od_{\mathcal{O}_{Y}})$
preserves perfection so that there is a pushforward functor $f_{\ast
}:\operatorname*{Perf}(X)\rightarrow\operatorname*{Perf}(Y)$, see \cite[Thm.
2.5.4]{MR1106918}.

Let $Z$ be a closed subset. We write $\operatorname*{Perf}\nolimits_{Z}(X)$
for the category of perfect complexes on $X$ whose pullback to the open
$U:=X-Z$ is acyclic. The homological support $\operatorname*{supph}\left(
\mathcal{F}\right)  $ of a perfect complex $\mathcal{F}_{\bullet}$ is the
support of the total homology of $\mathcal{F}_{\bullet}$, i.e. it is the union
of the supports of the sheaves $\bigcup_{i}\operatorname*{supp}\mathcal{H}%
^{i}(\mathcal{F}_{\bullet})$. According to taste, the reader may view
$\operatorname*{Perf}\nolimits_{Z}(X)$ (and then $\operatorname*{Perf}%
(X)=\operatorname*{Perf}\nolimits_{X}(X) $ as well) as a stable $\infty
$-category. Alternatively, it can also be modelled as a Waldhausen category,
i.e. a classical $1$-category with quasi-isomorphisms of complexes as weak
equivalences. Using either formalism there is an associated homotopy category,
$D_{Z}(X):=\operatorname*{Ho}(\operatorname*{Perf}\nolimits_{Z}(X))$, a
triangulated category. Define the full subcategory of perfect complexes of
homological support $\geq p$ by%
\[
\operatorname*{Perf}\nolimits_{Z^{p}}(X):=\{\mathcal{F}_{\bullet}%
\in\operatorname*{Perf}(X)\mid\operatorname*{codim}\nolimits_{X}%
(\operatorname*{supph}\mathcal{F})\geq p\}\text{.}%
\]
As before, this can either be viewed as a stable $\infty$-category or a
Waldhausen category with quasi-isomorphisms as weak equivalences. The analogue
of the filtration in Equation \ref{comalg_fact_D34b} is now the filtration%
\begin{equation}
\cdots\hookrightarrow\operatorname*{Perf}\nolimits_{Z^{2}}(X)\hookrightarrow
\operatorname*{Perf}\nolimits_{Z^{1}}(X)\hookrightarrow\operatorname*{Perf}%
\nolimits_{Z^{0}}(X)=\operatorname*{Perf}(X)\text{.}\label{lHHC4}%
\end{equation}
Then%
\begin{equation}
\operatorname*{Perf}\nolimits_{Z}(X)\longrightarrow\operatorname*{Perf}%
(X)\longrightarrow\operatorname*{Perf}(U)\label{lHHC1}%
\end{equation}
is an exact sequence of stable $\infty$-categories, known as the
\emph{localization sequence}. There is also a product%
\begin{align}
\operatorname*{Perf}\nolimits_{Z_{1}}(X)\times\operatorname*{Perf}%
\nolimits_{Z_{2}}(X)  & \longrightarrow\operatorname*{Perf}\nolimits_{Z_{1}%
\cap Z_{2}}(X)\label{lEqProdPerfectComplexes}\\
\mathcal{F}_{\bullet}\otimes\mathcal{G}_{\bullet}  & \longmapsto
(\mathcal{F}\otimes^{\mathbf{L}}\mathcal{G})_{\bullet}\nonumber
\end{align}
sending bounded complexes of perfect sheaves to their derived tensor product.
By $\operatorname*{supp}(\mathcal{F}\otimes\mathcal{G})=\operatorname*{supp}%
\mathcal{F}\cap\operatorname*{supp}\mathcal{G}$, the tensor product of perfect
complexes with supports in $Z_{i}$ (for $i=1,2$) is a bi-exact functor to
perfect complexes with support in $Z_{1}\cap Z_{2}$. As before, one can
construct a convergent spectral sequence, essentially due to P. Balmer
\cite{MR2439430}, namely we obtain the following:

\begin{proposition}
\label{marker_Prop_ConiveauHHSpecseq}The filtration in Equation \ref{lHHC4}
gives rise to a convergent spectral sequence with%
\[
\left.  ^{HH}E_{1}^{p,q}\right.  :=\coprod_{x\in X^{p}}HH_{-p-q}%
^{x}(\mathcal{O}_{X,x})\Rightarrow HH_{-p-q}(X)\text{.}%
\]

\end{proposition}

The rows of the $E_{1}$-page read%
\begin{equation}
\cdots\longrightarrow\coprod_{x\in X^{0}}HH_{-q}^{x}(\mathcal{O}%
_{X,x})\overset{d}{\longrightarrow}\coprod_{x\in X^{1}}HH_{-q-1}%
^{x}(\mathcal{O}_{X,x})\overset{d}{\longrightarrow}\cdots\overset
{d}{\longrightarrow}\coprod_{x\in X^{n}}HH_{-q-n}^{x}(\mathcal{O}%
_{X,x})\longrightarrow\cdots\label{la2}%
\end{equation}
(this is the $q$-th row, concentrated columnwise in the range $0\leq p\leq n$
for $n:=\dim X$). The differential $d$ agrees with the upward arrow in the
following Diagram \ref{lHHC2}: Replicating copies of the long exact sequence
in Hochschild homology associated to localizations as in Equation \ref{lHHC1},
but adapted to the filtration $\operatorname*{Perf}\nolimits_{Z^{p}}(X)$, we
arrive at the diagram%
\begin{equation}%
\bfig\node a(0,0)[HH^{Z^{p+1}}_{i}(X)]
\node b(1000,0)[HH^{Z^{p}}_{i}(X)]
\node c(2000,0)[\coprod_{x\in X^{p}}HH^{x}_{i}({\mathcal{O}_{X,x}})]
\node d(3000,0)[HH^{Z^{p+1}}_{i-1}(X)]
\node A(0,500)[HH^{Z^{p+2}}_{i-1}(X)]
\node B(1000,500)[HH^{Z^{p+1}}_{i-1}(X)]
\node C(2000,500)[\coprod_{x\in X^{p+1}}HH^{x}_{i-1}({\mathcal{O}_{X,x}})]
\node D(3000,500)[HH^{Z^{p+2}}_{i-2}(X)]
\arrow[a`b;]
\arrow[b`c;]
\arrow[c`d;{\partial}]
\arrow[A`B;]
\arrow[B`C;]
\arrow[C`D;{\partial}]
\arrow[d`B;]
\arrow[c`C;]
\efig
\label{lHHC2}%
\end{equation}
imitating Diagram \ref{comalg_fact_D34} that we had constructed before.

\begin{remark}
[Failure of $\mathbf{A}^{1}$-invariance]The complex in line \ref{la2} is the
analogue of the Gersten complex in algebraic $K$-theory. In the $K$-theory of
coherent sheaves, one can replace the analogous $K$-theory groups with support
by the $K$-theory of the residue field by d\'{e}vissage. This is why the
$K$-theory Gersten complex is usually written down in the simpler fashion
which does not mention any conditions on support. One of the starting points
of this paper was: How can one formulate a Gersten complex for Hochschild
homology? There is a general device producing Gersten complexes for
$\mathbf{A}^{1}$-invariant Zariski sheaves with transfers \cite[Thm.
4.37]{MR1764200} as well as a coniveau spectral sequence \cite[Remark
24.12]{MR2242284}. However, Hochschild homology is \emph{not} $\mathbf{A}^{1}%
$-invariant, so these tools do not apply in our context. One could still use
the technology of \cite{MR1466971}, which does not depend on $\mathbf{A}^{1}%
$-invariance in any form. Instead, we use Balmer's triangulated technique,
which also does not hinge on $\mathbf{A}^{1}$-invariance \cite[Thm.
2]{MR2439430}.
\end{remark}

\begin{proof}
We leave it to the reader to fill in the details of the construction as
described. Alternatively, the reader can just follow the argument of Balmer
\cite[Thm. 2]{MR2439430} and replace $K$-theory everywhere with Hochschild
homology:\ Namely, from the filtration of Equation \ref{lHHC4} we get an exact
sequence of dg categories%
\[
\operatorname*{Perf}\nolimits_{Z^{p+1}}(X)\longrightarrow\operatorname*{Perf}%
\nolimits_{Z^{p}}(X)\longrightarrow\operatorname*{Perf}\nolimits_{Z^{p}%
}(X)/\operatorname*{Perf}\nolimits_{Z^{p+1}}(X)
\]
and thanks to a strikingly general result of Balmer \cite[Thm. 3.24]%
{MR2354319} the idempotent completion of the right-most category can be
identified as%
\begin{equation}
\left(  \operatorname*{Perf}\nolimits_{Z^{p}}(X)/\operatorname*{Perf}%
\nolimits_{Z^{p+1}}(X)\right)  ^{ic}\overset{\sim}{\longrightarrow}%
\coprod_{x\in X^{p}}\operatorname*{Perf}\nolimits_{x}(\mathcal{O}%
_{X,x})\text{.}\label{lHHC10}%
\end{equation}
In Balmer's paper \cite[Thm. 2]{MR2439430} this argument is spelled out as an
exact sequence of triangulated categories with Waldhausen models, whereas we
have spelled it out as an exact sequence of dg categories. The exactness of
either viewpoint is equivalent to the other, see \cite[Prop. 5.15]{MR3070515}.
The r\^{o}le of Schlichting's localization theorem is taken by Keller's
localization theorem \cite{MR1667558} (a very clean and brief statement is
also found in \cite[\S 5.5, Theorem]{MR1647519}). The convergence of the
spectral sequence follows readily from the fact that its horizontal support is
bounded since $X^{p}=\varnothing$ for $p\notin\lbrack0,\dim X]$.
\end{proof}

\begin{remark}
\label{rmk_EquivalenceInducedByPullback}For later reference, let us make the
functor in line \ref{lHHC10} more precise: For each point $x\in X^{p}$, this
is the pullback along the flat morphism $j_{x}:\operatorname*{Spec}%
\mathcal{O}_{X,x}\hookrightarrow X$. Balmer shows that this induces the
relevant equivalence \cite[\S 4.1]{MR2354319}. In fact, he shows more: Perfect
complexes can be regarded as a tensor triangulated category and under fairly
weak assumptions the points of its Balmer spectrum (i.e. the prime $\otimes
$-ideals, see \cite{MR2196732}) correspond canonically to the points of the
scheme $X$. If $\mathcal{P}(x)$ denotes the prime $\otimes$-ideal of this
point, one gets an exact sequence of triangulated categories%
\[
\mathcal{P}(x)\longrightarrow\operatorname*{Perf}(X)\overset{j^{\ast}%
}{\longrightarrow}\operatorname*{Perf}(\mathcal{O}_{X,x})\text{.}%
\]
See \cite{MR2354319}.
\end{remark}

\subsection{\label{marker_subsect_HochschildHomologyOfCohSheaves}Hochschild
homology of different categories}

As for $K$-theory, one could consider the Hochschild homology not just of
perfect complexes (which is the standard choice, because it is best-behaved
for most applications), but also of coherent sheaves $\operatorname*{Coh}%
\nolimits_{Z}(X)$ with support. Both viewpoints are related by the following
standard fact:

\begin{proposition}
If $X$ is a regular finite-dimensional Noetherian separated scheme, there are
triangulated equivalences%
\[
\operatorname*{Perf}(X)\overset{\sim}{\longrightarrow}D_{coh}^{b}%
(\operatorname*{Mod}(\mathcal{O}_{X}))\overset{\sim}{\longleftarrow}%
D^{b}(\operatorname*{Coh}(X))\text{,}%
\]
where the middle term is the bounded derived category of $\mathcal{O}_{X}%
$-module sheaves whose cohomology are coherent sheaves.
\end{proposition}

This was proven in \cite[Expos\'{e} I]{MR0354655}, see also \cite[\S 3]%
{MR1106918}. The converse is also true: If the first arrow is a triangle
equivalence, $X$ must have been regular \cite[Prop. 2.1]{MR3421093}. In
analogy to Equation \ref{lHHC1} we have a localization sequence in Hochschild
homology, but for coherent sheaves, induced from the exact sequence of abelian
categories
\[
\operatorname*{Coh}\nolimits_{Z}(X)\longrightarrow\operatorname*{Coh}%
(X)\longrightarrow\operatorname*{Coh}(U)\text{,}%
\]
inducing an exact sequence of stable $\infty$-categories. If $X$ is regular,
so is $U$, and since this exact sequence determines the left-hand side term
$\operatorname*{Coh}\nolimits_{Z}(X)$, it follows that $\operatorname*{Coh}%
\nolimits_{Z}(X)\simeq\operatorname*{Perf}\nolimits_{Z}(X)$. In general,
\cite[\S 3]{MR1106918} is an excellent reference for this type of material.

\begin{corollary}
\label{TMTA_Cor_ForRegularSchemesCohShvWithSupportHasSameHHAsPerf}If $X$ is a
regular Noetherian scheme, it does not make a difference whether we carry out
the constructions of \S \ref{marker_Sect_HochschildHomologyWithSupports} for
perfect complexes with support, or coherent sheaves with support. The results
are canonically isomorphic.
\end{corollary}

\begin{remark}
As algebraic $K$-theory satisfies d\'{e}vissage, one obtains an equivalence
$K(\operatorname*{Coh}(Z))\overset{\sim}{\longrightarrow}K(\operatorname*{Coh}%
\nolimits_{Z}(X))$ for $X$ Noetherian. The Hochschild analogue%
\[
HH(\operatorname*{Coh}(Z))\overset{?}{\longrightarrow}HH(\operatorname*{Coh}%
\nolimits_{Z}(X))
\]
is \textsl{false}. The issue is not on the level of categories, but rather
that Hochschild homology does not satisfy d\'{e}vissage. The failure of
d\'{e}vissage was originally discovered by Keller \cite[\S 1.10]{MR1667558}.
\end{remark}

\begin{proposition}
[Thomason]There is a fully faithful triangular functor $D^{b}%
\operatorname*{VB}(X)\rightarrow\operatorname*{Perf}(X)$ from the bounded
derived category of vector bundles on $X$ to perfect complexes. If $X$ has an
ample family of line bundles, this is an equivalence of triangulated categories.
\end{proposition}

See \cite{MR1106918}.

\subsection{Flat pullback functoriality}

Next, we want to study the functoriality of the coniveau spectral sequences
under flat morphisms. There is the standard pullback of differential forms and
moreover the pullback of perfect complexes $f^{\ast}:\operatorname*{Perf}%
(Y)\rightarrow\operatorname*{Perf}(X)$, defined as the total left derived
functor of the pullback of complexes. If $f$ is flat, this literally sends a
strictly perfect complex $\mathcal{F}_{\bullet}$ to $f^{\ast}\mathcal{F}%
_{\bullet}=f^{-1}\mathcal{F}_{\bullet}\otimes_{\mathcal{O}_{Y}}\mathcal{O}%
_{X}$.

\begin{lemma}
Suppose $k$ is a field. Let $X,Y$ be smooth $k$-schemes and $f:X\rightarrow Y
$ any morphism.

\begin{enumerate}
\item If $X,Y$ are affine, the induced pullbacks induce a commutative square%
\[%
\bfig\Square(0,0)/>`<-`<-`>/[{\Omega_{Y}^{i}}`HH_i{\mathop{\rm Perf}%
}(Y)`{\Omega_{X}^{i}}`HH_i{\mathop{\rm Perf}}(X);{}`{{f^{\ast}}}`{{f^{\ast}}%
}`{}]
\efig
\]
with the horizontal arrows the Hochschild-Kostant-Rosenberg (\textquotedblleft
HKR\textquotedblright) isomorphisms; cf. \S \ref{sect_HKR_with_supports} for a
reminder on the HKR\ isomorphism.

\item For $X,Y$ not necessarily affine, the pullbacks of the Zariski
sheafifications%
\[%
\bfig\Square(0,0)/>`<-`<-`>/[{\Gamma(Y,\Omega^{i})}`{\Gamma(Y,\mathcal{HH}%
_{i})}`{\Gamma(X,\Omega^{i})}`{\Gamma(X,\mathcal{HH}_{i})};{}`{f^{\ast}%
}`{f^{\ast}}`{}]
\efig
\]
induce a commutative square.
\end{enumerate}
\end{lemma}

In the case of a flat morphism, we can describe the induced morphisms on the
respective spectral sequences in an explicit fashion.

\begin{proposition}
[Flat Pullbacks]\label{Prop_FlatPullbacks}Let $f:X\rightarrow Y$ be a flat
morphism between Noetherian schemes.

\begin{enumerate}
\item Then the pullback of differential forms induces a morphism of spectral
sequences,%
\[
f^{\ast}:\left.  ^{Cous}E_{r}^{p,q}(Y,\Omega^{n})\right.  \longrightarrow
\left.  ^{Cous}E_{r}^{p,q}(X,\Omega^{n})\right.  \text{.}%
\]
On the $E_{1}$-page this map unwinds as follows: Given $x\in X^{p}$, $y\in
Y^{p}$ the map between the respective summands is zero if $f(x)\neq y$ and the
canonical map $H_{y}^{p+q}(\mathcal{O}_{Y,y},\Omega^{n})\rightarrow
H_{x}^{p+q}(\mathcal{O}_{X,x},\Omega^{n})$ otherwise.

\item Then the pullback $f^{\ast}:\operatorname*{Perf}(Y)\rightarrow
\operatorname*{Perf}(X)$ is an exact functor and induces a morphism of
spectral sequences, which we shall also denote by%
\[
f^{\ast}:\left.  ^{HH}E_{r}^{p,q}(Y)\right.  \longrightarrow\left.  ^{HH}%
E_{r}^{p,q}(X)\right.  \text{.}%
\]
On the $E_{1}$-page this map unwinds as follows: Given $x\in X^{p}$, $y\in
Y^{p}$ the map between the respective summands is zero if $f(x)\neq y$ and the
canonical map $HH_{-p-q}^{y}(\mathcal{O}_{Y,y})\rightarrow HH_{-p-q}%
^{x}(\mathcal{O}_{X,x})$ otherwise.

\item In either case for a given $y\in Y^{p}$ we have $x\in X^{p}\cap
f^{-1}(y)$ exactly if $x$ is the generic point of an irreducible component of
the scheme-theoretic fibre $f^{-1}(y)$. In particular, for any given $y\in
Y^{p}$ there are only finitely many such.
\end{enumerate}
\end{proposition}

\begin{proof}
Follow \cite[Prop. 1.2]{MR533201}, which is written for $K$-theory, but can
easily be adapted.
\end{proof}

\section{Hochschild-Kostant-Rosenberg isomorphism with
supports\label{sect_HKR_with_supports}}

This section will be devoted to a crucial comparison result: We will show that
a certain excerpt of the long exact sequence in relative local homology, i.e.
coming from Equation \ref{comalg_fact_D32}, is canonically isomorphic to a
matching excerpt of the localization sequence in Hochschild homology.\ This is
heavily inspired by\ Keller's beautiful paper \cite{MR1647519}. The main
consequence is that the boundary maps of these two sequences, even though they
originate from quite different sources, actually agree.

Let us briefly recall that if $R$ is a smooth $k$-algebra, the
Hochschild-Kostant-Rosenberg map%
\begin{align}
\phi_{\ast,0}:\Omega_{R/k}^{\ast}  & \longrightarrow HH_{\ast}%
(R)\label{lCOA_30c}\\
f_{0}\mathrm{d}f_{1}\wedge\cdots\wedge\mathrm{d}f_{n}  & \longmapsto\sum
_{\pi\in S_{n+1}}\operatorname*{sgn}(\pi)f_{\pi(0)}\otimes\cdots\otimes
f_{\pi(n)}\nonumber
\end{align}
induces an isomorphism of graded algebras $-$ this is the classical
Hochschild-Kostant-Rosenberg isomorphism (\cite[Thm. 3.4.4]{MR1217970}). We
obtain the following isomorphisms as a trivial consequence:%
\[
H^{0}(R,\Omega^{i})\underset{\simeq}{\overset{\psi_{i,0}}{\longrightarrow}%
}\Omega_{R/k}^{i}\underset{\simeq}{\overset{\phi_{i,0}}{\longrightarrow}%
}HH_{i}(R)\text{.}%
\]
The first part of the following proposition can be seen as a generalization of
this fact to Hochschild homology with support in a regularly embedded closed subscheme.

\begin{proposition}
[H-K-R with support]\label{marker_Prop_HKRWithSupport}Let $k$ be a field, $R$
a smooth $k$-algebra and $t_{1},\ldots,t_{n}$ a regular sequence.

\begin{enumerate}
\item \emph{(Isomorphisms)} There are canonical isomorphisms%
\begin{equation}
\phi_{i,n}\circ\psi_{i,n}:H_{(t_{1},\ldots,t_{n})}^{n}(R,\Omega^{n+i}%
)\longrightarrow HH_{i}^{(t_{1},\ldots,t_{n})}(R)\text{,}\label{lCOA_30b}%
\end{equation}
functorial in ring morphisms $R\rightarrow R^{\prime}$ sending $t_{1}%
,\ldots,t_{n}$ to a regular sequence.

\item \emph{(Boundary maps)} The following diagram%
\begin{equation}%
\bfig\node xa(-700,0)[0]
\node xaaa(-700,600)[0]
\node xc(3500,0)[0]
\node xccc(3500,600)[0]
\node a(0,0)[{H_{(t_{1},\ldots,t_{n})}^{n}(R,\Omega^{n+i})}]
\node b(1400,0)[{H_{(t_{1},\ldots,t_{n})}^{n}(R[t_{n+1}^{-1}],\Omega^{n+i})}]
\node c(2800,0)[{H_{(t_{1},\ldots,t_{n+1})}^{n+1}(R,\Omega^{n+i})}]
\node aaa(0,600)[{HH_{i}^{(t_{1},\ldots,t_{n})}(R)}]
\node bbb(1400,600)[{HH_{i}^{(t_{1},\ldots,t_{n})}(R[t_{n+1}^{-1}])}]
\node ccc(2800,600)[{HH_{i-1}^{(t_{1},\ldots,t_{n+1})}(R)}]
\arrow[a`aaa;{\cong}]
\arrow[b`bbb;{\cong}]
\arrow[c`ccc;{\cong}]
\arrow[a`b;{}]
\arrow[b`c;{\partial}]
\arrow[aaa`bbb;{}]
\arrow[bbb`ccc;{\partial}]
\arrow[xa`a;{}]
\arrow[xaaa`aaa;{}]
\arrow[c`xc;{}]
\arrow[ccc`xccc;{}]
\efig
\label{lCOA_30}%
\end{equation}
commutes, where the top row is an excerpt of the localization sequence for
Hochschild homology, the bottom row of the long exact relative local homology
sequence coming from Equation \ref{comalg_fact_D32} and the upward arrows are
the isomorphisms $\phi_{i,n}\circ\psi_{i,n}$. In particular, these excerpts of
the long exact sequences are short exact.

\item \emph{(Products)} Suppose $t_{1},\ldots,t_{n}$ and $t_{1}^{\prime
},\ldots,t_{m}^{\prime}$ are regular sequences such that their concatenation
is also a regular sequence (this is also known as \textquotedblleft%
\emph{transversally intersecting}\textquotedblright). The isomorphisms in (1)
respect the natural product structures, i.e. of Equation \ref{comalg_fact_D40}
and Equation \ref{lEqProdPerfectComplexes}, so that the diagram%
\[%
\bfig\Square[{H_{(t_{1},\ldots,t_{n})}^{n}(R,\Omega^{n+i})\otimes
H_{(t_{1}^{\prime},\ldots,t_{m}^{\prime})}^{m}(R,\Omega^{m+j})}`{H_{(t_{1}%
,\ldots,t_{n},t_{1}^{\prime},\ldots,t_{m}^{\prime})}^{n+m}(R,\Omega
^{n+m+i+j})}`{HH_{i}^{(t_{1},\ldots,t_{n})}(R)\otimes HH_{j}^{(t_{1}^{\prime
},\ldots,t_{m}^{\prime})}(R)}`{HH_{i+j}^{(t_{1},\ldots,t_{n},t_{1}^{\prime
},\ldots,t_{m}^{\prime})}(R)};{}`{\cong}`{\cong}`{}]
\efig
\]
commutes.

\item Part (2) remains true if one replaces the usual (perfect complex)
localization sequence for Hochschild homology by the localization sequence
based on coherent sheaves with support, cf.
\S \ref{marker_subsect_HochschildHomologyOfCohSheaves}.
\end{enumerate}
\end{proposition}

\begin{proof}
This is, albeit not explicitly, a consequence of Keller \cite[\S 4-\S 5]%
{MR1647519} (loc. cit. phrases it for mixed complexes \`{a} la Kassel, but
this clearly implies the Hochschild case). One can, however, also prove the
above claims rather directly by an induction on codimension, and we will give
this alternative proof: It naturally splits into two parts, establishing first
the isomorphisms $\psi$ (this is classical, we just unwind it explicitly to be
sure that all maps agree), and then the isomorphisms $\phi$ later, so we
really want to establish the isomorphisms%
\begin{equation}
H_{(t_{1},\ldots,t_{n})}^{n}(R,\Omega^{n+i})\overset{\psi_{i,n}}%
{\longrightarrow}\frac{\Omega_{R[t_{1}^{-1},\ldots,t_{n}^{-1}]/k}^{n+i}}%
{\sum_{j=1}^{n}\Omega_{R[t_{1}^{-1},\ldots,\widehat{t_{j}^{-1}},\ldots
,t_{n}^{-1}]/k}^{n+i}}\overset{\phi_{i,n}}{\longrightarrow}HH_{i}%
^{(t_{1},\ldots,t_{n})}(R)\text{,}\label{lCOA_30d}%
\end{equation}
which is a little more detailed than the Equation \ref{lCOA_30b}. Hence, we
first focus entirely on establishing for all $i,n$, the commutative diagrams:%
\begin{equation}%
\bfig\node a(0,0)[{H_{(t_{1},\ldots,t_{n})}^{n}(R,\Omega^{n+i})}]
\node b(1400,0)[{H_{(t_{1},\ldots,t_{n})}^{n}(R[t_{n+1}^{-1}],\Omega^{n+i})}]
\node c(2800,0)[{H_{(t_{1},\ldots,t_{n+1})}^{n+1}(R,\Omega^{n+i})}]
\node aa(0,500)[{{\frac{\Omega_{R[t_{1}^{-1},\ldots,t_{n}^{-1}]/k}^{n+i}}%
{\sum_{j=1}^{n}\Omega_{R[t_{1}^{-1},\ldots,\widehat{t_{j}^{-1}},\ldots
,t_{n}^{-1}]/k}^{n+i}}}}]
\node bb(1400,500)[{{\frac{\Omega_{R[t_{1}^{-1},\ldots,t_{n}^{-1},t_{n+1}%
^{-1}]/k}^{n+i}}{\sum_{j=1}^{n}\Omega_{R[t_{1}^{-1},\ldots,\widehat{t_{j}%
^{-1}},\ldots,t_{n}^{-1},t_{n+1}^{-1}]/k}^{n+i}}}}]
\node cc(2800,500)[{{\frac{\Omega_{R[t_{1}^{-1},\ldots,t_{n+1}^{-1}]/k}^{n+i}%
}{\sum_{j=1}^{n+1}\Omega_{R[t_{1}^{-1},\ldots,\widehat{t_{j}^{-1}}%
,\ldots,t_{n+1}^{-1}]/k}^{n+i}}}}]
\arrow[a`aa;{\cong}]
\arrow[b`bb;{\cong}]
\arrow[c`cc;{\cong}]
\arrow[a`b;{}]
\arrow[b`c;{\partial}]
\arrow[aa`bb;{}]
\arrow[bb`cc;{\partial}]
\efig
\label{lCOA_33}%
\end{equation}
We do this by induction on $n$. The case $n=0$ is trivial, just $H^{0}%
(R,\Omega^{i})\cong\Omega_{R/k}^{i}$. Suppose the case $n$ is settled. The
long exact sequence from Equation \ref{comalg_fact_D32} tells us that%
\begin{equation}
\cdots\rightarrow H_{(t_{1},\ldots,t_{n+1})}^{n}(R,\Omega^{n+i})\overset
{(\ast)}{\rightarrow}H_{(t_{1},\ldots,t_{n})}^{n}(R,\Omega^{n+i})\rightarrow
H_{(t_{1},\ldots,t_{n})}^{n}(R[t_{n+1}^{-1}],\Omega^{n+i})\rightarrow
H_{(t_{1},\ldots,t_{n+1})}^{n+1}(R,\Omega^{n+i})\cdots\label{lCOA_32}%
\end{equation}
is exact. The two middle terms by induction hypothesis identify with%
\begin{equation}
\frac{\Omega_{R[t_{1}^{-1},\ldots,t_{n}^{-1}]/k}^{n+i}}{\sum_{j=1}^{n}%
\Omega_{R[t_{1}^{-1},\ldots,\widehat{t_{j}^{-1}},\ldots,t_{n}^{-1}]/k}^{n+i}%
}\overset{\alpha}{\longrightarrow}\frac{\Omega_{R[t_{1}^{-1},\ldots
,t_{n+1}^{-1}]/k}^{n+i}}{\sum_{j=1}^{n}\Omega_{R[t_{1}^{-1},\ldots
,\widehat{t_{j}^{-1}},\ldots,t_{n+1}^{-1}]/k}^{n+i}}\text{,}\label{lCOA_31}%
\end{equation}
which is injective since we invert only nonzerodivisors (and the module of
differential forms is free). Thus, the map $(\ast)$ must be the zero map. The
next term on the right in line \ref{lCOA_32} would be $H_{(t_{1},\ldots
,t_{n})}^{n+1}(R,\Omega^{n+i})$, which is zero since the ideal is generated by
just $n$ elements (Lemma
\ref{COA_LocalCohomVanishesIfDegreeExceedsIdealGeneratorNumber}). This proves
that the bottom row in Equation \ref{lCOA_30} is short exact, and as a result
its third term is just the quotient of the map in Equation \ref{lCOA_31}, thus
establishing Diagram \ref{lCOA_33}. Now take the upward isomorphism on the
right as the definition for%
\[
\psi_{i-1,n+1}:H_{(t_{1},\ldots,t_{n+1})}^{n+1}(R,\Omega^{n+i})\longrightarrow
\frac{\Omega_{R[t_{1}^{-1},\ldots,t_{n+1}^{-1}]/k}^{n+i}}{\sum_{j=1}%
^{n+1}\Omega_{R[t_{1}^{-1},\ldots,\widehat{t_{j}^{-1}},\ldots,t_{n+1}^{-1}%
]/k}^{n+i}}\text{,}%
\]
establishing the isomorphism $\psi_{i-1,n+1}$ in the first part of the claim
(note that $i$ was arbitrary all along, so it is no problem that we
constructed $\psi_{i-1,n+1}$ on the basis of $\psi_{i,n}$). From now on we can
assume to have all $\psi_{-,-}$ and Diagrams \ref{lCOA_33} available (for all
$i$ and $n$).

Next, we employ the localization sequence for the corresponding categories of
perfect complexes with support, giving the long exact sequence%
\begin{equation}
\cdots\rightarrow HH_{i}^{(t_{1},\ldots,t_{n+1})}(R)\rightarrow HH_{i}%
^{(t_{1},\ldots,t_{n})}(R)\rightarrow HH_{i}^{(t_{1},\ldots,t_{n})}%
(R[t_{n+1}^{-1}])\overset{\partial}{\rightarrow}HH_{i-1}^{(t_{1}%
,\ldots,t_{n+1})}(R)\rightarrow\cdots\text{.}\label{lCOA_34}%
\end{equation}
We start a new induction, again along $n$. For $n=0$ this sequence reads%
\[
\cdots\rightarrow HH_{i}^{(t_{1})}(R)\rightarrow HH_{i}(R)\rightarrow
HH_{i}(R[t_{1}^{-1}])\overset{\partial}{\rightarrow}HH_{i-1}^{(t_{1}%
)}(R)\rightarrow\cdots\text{.}%
\]
and via the Hochschild-Kostant-Rosenberg isomorphism identifies with%
\[
\cdots\rightarrow HH_{i}^{(t_{1})}(R)\overset{\beta}{\rightarrow}\Omega
_{R/k}^{i}\overset{\alpha}{\rightarrow}\Omega_{R[t_{1}^{-1}]/k}^{i}%
\overset{\partial}{\rightarrow}HH_{i-1}^{(t_{1})}(R)\overset{\beta
}{\rightarrow}\Omega_{R/k}^{i-1}\overset{\alpha}{\rightarrow}\Omega
_{R[t_{1}^{-1}]/k}^{i-1}\overset{\partial}{\rightarrow}\cdots\text{.}%
\]
The maps denoted by $\alpha$ in the localization sequence are induced from the
pullback of a perfect complex to the open along $\operatorname*{Spec}%
R[t_{1}^{-1}]\hookrightarrow\operatorname*{Spec}R$, and is known to correspond
on differential forms to the same $-$ the pullback to the open. Thus, the
morphisms $\alpha$ are injective and thus the maps denoted by $\beta$ must be
zero maps. This settles the exactness of the top row in diagram \ref{lCOA_30}
for $n=0$. In fact, by direct inspection of the maps, it establishes the
commutativity of the entire diagram. Now suppose the case $n$ is settled.
Using the induction hypothesis we can identify the middle bit of Equation
\ref{lCOA_34} with the map of Equation \ref{lCOA_31}. This yields the
identification%
\begin{align*}
\cdots & \rightarrow HH_{i}^{(t_{1},\ldots,t_{n+1})}(R)\overset{\beta
}{\rightarrow}\frac{\Omega_{R[t_{1}^{-1},\ldots,t_{n}^{-1}]/k}^{n+i}}%
{\sum_{j=1}^{n}\Omega_{R[t_{1}^{-1},\ldots,\widehat{t_{j}^{-1}},\ldots
,t_{n}^{-1}]/k}^{n+i}}\\
& \qquad\cdots\overset{\alpha}{\rightarrow}\frac{\Omega_{R[t_{1}^{-1}%
,\ldots,t_{n}^{-1},t_{n+1}^{-1}]/k}^{n+i}}{\sum_{j=1}^{n}\Omega_{R[t_{1}%
^{-1},\ldots,\widehat{t_{j}^{-1}},\ldots,t_{n}^{-1},t_{n+1}^{-1}]/k}^{n+i}%
}\overset{\partial}{\rightarrow}HH_{i-1}^{(t_{1},\ldots,t_{n+1})}%
(R)\overset{\beta}{\rightarrow}\cdots
\end{align*}
and again the injectivity of $\alpha$ (it is the same map as in Equation
\ref{lCOA_31}) implies that the maps $\beta$ must be zero. This settles the
exactness of the top row and the commutativity of diagram \ref{lCOA_30} in
general. Patching it to the diagrams \ref{lCOA_33} of the first part of the
proof finishes the argument.\newline It remains to prove (3). The product is
induced in local cohomology from Equation \ref{comalg_fact_D40}, composed with
the product of the exterior algebra on $1$-forms, i.e.%
\[
\underline{\Gamma_{Z_{1}}}\Omega^{i}\mathcal{\otimes}\underline{\Gamma_{Z_{2}%
}}\Omega^{j}\longrightarrow\underline{\Gamma_{Z_{1}\cap Z_{2}}}(\Omega
^{i}\otimes\Omega^{j})\longrightarrow\underline{\Gamma_{Z_{1}\cap Z_{2}}%
}(\Omega^{i+j})\text{;}%
\]
and in Hochschild homology from the bi-exact tensor functor%
\[
\operatorname*{Perf}\nolimits_{Z_{1}}X\times\operatorname*{Perf}%
\nolimits_{Z_{2}}X\longrightarrow\operatorname*{Perf}\nolimits_{Z_{1}\cap
Z_{2}}X\text{.}%
\]
The compatibility of products for $n=m=0$ follows directly from the classical
HKR isomorphism in Equation \ref{lCOA_30c}. Consider the commutative diagram
of Equation \ref{lCOA_30} for $n=0$. We see that both upward arrows respect
the product, and the horizontal arrows (i.e. pullback to an open subscheme)
respect the product as well. We see that a product with a term in $H^{1}$ can
be computed by lifting it along $\partial$ to $H^{0}$ and computing the
product there and mapping it back to $H^{1}$ (e.g. in the middle term of
Equation \ref{lCOA_30d}). This deduces the claim for all products
$H^{i}\otimes H^{j}$ with $i,j\leq1$ from the $H^{0}$-case. With the same
argument lift elements along $\partial$ from $H^{n}$ to $H^{n-1}$, compute
products there, to inductively prove the claim for all products $H^{i}\otimes
H^{j}$ with $i,j\leq n$ once it is proven for all $i,j\leq n-1$. For (4) it
suffices to invoke Cor.
\ref{TMTA_Cor_ForRegularSchemesCohShvWithSupportHasSameHHAsPerf}; everything
carries over verbatim.
\end{proof}

\subsection{The $E_{1}$-pages}

We can use the results of the previous section in order to compare the
different coniveau spectral sequences from \S \ref{Sect_ConiveauFiltration}.
We need some basic facts regarding the vanishing of Hochschild or local
cohomology groups for local rings:

\begin{proposition}
\label{prop_ComputeLocalRingHHWithSupport}Let $k$ be a field and
$(R,\mathfrak{m})$ an essentially smooth local $k$-algebra of dimension $n$.
Then%
\[
HH_{i}^{\mathfrak{m}}(R)=\left\{
\begin{array}
[c]{ll}%
0 & \text{for }i>0\\
H_{\mathfrak{m}}^{n}(R,\Omega^{n+i}) & \text{for }-n\leq i\leq0\\
0 & \text{for }i<-n
\end{array}
\right.
\]
and if $M$ is a finitely generated $R$-module,%
\[
H_{\mathfrak{m}}^{p}(R,M)=\left\{
\begin{array}
[c]{ll}%
\operatorname*{Hom}\nolimits_{R}(M,\Omega^{n})^{\vee} & \text{for }p=n\\
0 & \text{for }p\neq n\text{,}%
\end{array}
\right.
\]
where $(-)^{\vee}:=\operatorname*{Hom}\nolimits_{R}(-,E)$ denotes the Matlis
dual (for $E$ some injective hull of $\kappa(\mathfrak{p})=R/\mathfrak{m}$ as
an $R$-module).
\end{proposition}

\begin{proof}
(see \cite{MR2355715} for background) Let $t_{1},\ldots,t_{n}$ be a regular
sequence, so that $\mathfrak{m}=(t_{1},\ldots,t_{n})$. By Prop.
\ref{marker_Prop_HKRWithSupport} and because $\Omega^{1}$ is free of rank $n$,
we have%
\[
HH_{i}^{\mathfrak{m}}(R)\cong H_{\mathfrak{m}}^{n}(R,\Omega^{n+i})\cong
H_{\mathfrak{m}}^{n}(R,R)\otimes\Omega^{n+i}\text{.}%
\]
Since $\Omega^{n+i}$ is zero for $i>0$, and similarly for $i<-n$, we
immediately get the first claim. Next, by (the simplest form of) Local Duality
we have%
\[
H_{(t_{1},\ldots,t_{n})}^{p}(R,M)\cong\left\{
\begin{array}
[c]{ll}%
\operatorname*{Hom}\nolimits_{R}(M,\omega_{R})^{\vee} & \text{for }p=n\\
0 & \text{for }p\neq n\text{,}%
\end{array}
\right.
\]
where $M$ is an arbitrary finitely generated $R$-module and $\omega_{R}$ a
canonical module over $k$. Since $R$ is a smooth $k$-algebra, $\omega
_{R}:=\Omega^{n}$ is a canonical module, and so we get the second claim.
\end{proof}

Let us compare the $E_{1}$-pages of the two different spectral sequences. They
are fairly different. For the coherent Cousin coniveau spectral sequence, it
is supported in the first quadrant and has the following shape:%
\[%
\begin{tabular}
[c]{r|cccccc}%
$q$ & $\vdots$ &  &  &  &  & \\
$2$ & $%
{\textstyle\coprod\limits_{x\in X^{0}}}
H_{x}^{2}(\Omega^{n})$ &  & $\ddots$ &  &  & \\
$1$ & $%
{\textstyle\coprod\limits_{x\in X^{0}}}
H_{x}^{1}(\Omega^{n})$ & $\rightarrow$ & $%
{\textstyle\coprod\limits_{x\in X^{1}}}
H_{x}^{2}(\Omega^{n})$ & $\rightarrow$ & $%
{\textstyle\coprod\limits_{x\in X^{2}}}
H_{x}^{3}(\Omega^{n})$ & \\
$0$ & $%
{\textstyle\coprod\limits_{x\in X^{0}}}
H_{x}^{0}(\Omega^{n})$ & $\rightarrow$ & $%
{\textstyle\coprod\limits_{x\in X^{1}}}
H_{x}^{1}(\Omega^{n})$ & $\rightarrow$ & $%
{\textstyle\coprod\limits_{x\in X^{2}}}
H_{x}^{2}(\Omega^{n})$ & $\rightarrow\cdots$\\\hline
\multicolumn{1}{c|}{} & $0$ &  & $1$ &  & $2$ & $p$%
\end{tabular}
\]
We have $H_{x}^{p}(X,\Omega^{n})=H_{x}^{p}(\mathcal{O}_{X,x},\Omega^{n})$ by
excision, Lemma \ref{COA_Lemma_LocalCohomExcision}, and since $\dim
(\mathcal{O}_{X,x})=\operatorname*{codim}_{X}(x)$, Prop.
\ref{prop_ComputeLocalRingHHWithSupport} implies that the groups on this
$E_{1}$-page vanish unless the cohomological degree matches the codimension of
the point in question. However, this is only the case for the $q=0$ row. We
are left with the following $E_{1}$-page:%
\[%
\begin{tabular}
[c]{r|cccccc}%
$q$ & $\vdots$ &  &  &  &  & \\
$1$ & $0$ & $\rightarrow$ & $0$ &  & $\ddots$ & \\
$0$ & $%
{\textstyle\coprod\limits_{x\in X^{0}}}
H_{x}^{0}(\Omega^{n})$ & $\rightarrow$ & $%
{\textstyle\coprod\limits_{x\in X^{1}}}
H_{x}^{1}(\Omega^{n})$ & $\rightarrow$ & $%
{\textstyle\coprod\limits_{x\in X^{2}}}
H_{x}^{2}(\Omega^{n})$ & $\rightarrow\cdots$\\\hline
\multicolumn{1}{c|}{} & $0$ &  & $1$ &  & $2$ & $p$%
\end{tabular}
\]
Thus, it collapses to a single row already on the $E_{1}$-page. As it
converges to $H^{p+q}(X,\Omega^{n})$, we re-obtain a special case of Cor.
\ref{marker_Cor_ExistenceCousinResolution}:

\begin{corollary}
\label{cor_CohomOfOmegaAgreesWithCousinCohomology}If $X/k$ is smooth, there
are canonical isomorphisms%
\[
H^{p}(X,\Omega^{n})\cong H^{p}(X,Cous_{\bullet}(\Omega^{n}))\text{,}%
\]
coming from the $E_{1}$-page degeneration of the coherent Cousin coniveau
spectral sequence.
\end{corollary}

Now let us compare these results to the Hochschild coniveau spectral sequence.
It is supported in the first and fourth quadrant.%
\[%
\begin{tabular}
[c]{r|cccccc}%
$q$ & $\vdots$ &  &  &  &  & \\
$1$ & $%
{\textstyle\coprod\limits_{x\in X^{0}}}
HH_{-1}^{x}(X)$ &  &  &  &  & \\
$0$ & $%
{\textstyle\coprod\limits_{x\in X^{0}}}
HH_{0}^{x}(X)$ & $\rightarrow$ & $%
{\textstyle\coprod\limits_{x\in X^{1}}}
HH_{-1}^{x}(X)$ & $\rightarrow$ & $\ddots$ & \\
$-1$ & $%
{\textstyle\coprod\limits_{x\in X^{0}}}
HH_{1}^{x}(X)$ & $\rightarrow$ & $%
{\textstyle\coprod\limits_{x\in X^{1}}}
HH_{0}^{x}(X)$ & $\rightarrow$ & $%
{\textstyle\coprod\limits_{x\in X^{2}}}
HH_{-1}^{x}(X)$ & \\
$-2$ & $%
{\textstyle\coprod\limits_{x\in X^{0}}}
HH_{2}^{x}(X)$ & $\rightarrow$ & $%
{\textstyle\coprod\limits_{x\in X^{1}}}
HH_{1}^{x}(X)$ & $\rightarrow$ & $%
{\textstyle\coprod\limits_{x\in X^{2}}}
HH_{0}^{x}(X)$ & $\rightarrow\cdots$\\
& $\vdots$ &  & $\vdots$ &  & $\vdots$ & \\\hline
\multicolumn{1}{c|}{} & $0$ &  & $1$ &  & $2$ & $p$%
\end{tabular}
\]
If we make use of our HKR theorem with supports, this can be rephrased in
terms of local cohomology groups.%
\[%
\begin{tabular}
[c]{r|cccccc}%
$q$ & $\vdots$ &  & $\vdots$ &  &  & \\
$1$ & $0$ &  & $0$ &  & $0$ & \\
$0$ & $%
{\textstyle\coprod\limits_{X^{0}}}
H_{x}^{0}(\Omega^{0})$ & $\rightarrow$ & $%
{\textstyle\coprod\limits_{X^{1}}}
H_{x}^{1}(\Omega^{0})$ & $\rightarrow$ & $\ddots$ & \\
$-1$ & $%
{\textstyle\coprod\limits_{X^{0}}}
H_{x}^{0}(\Omega^{1})$ & $\rightarrow$ & $%
{\textstyle\coprod\limits_{X^{1}}}
H_{x}^{1}(\Omega^{1})$ & $\rightarrow$ & $%
{\textstyle\coprod\limits_{X^{2}}}
H_{x}^{2}(\Omega^{1})$ & \\
$-2$ & $%
{\textstyle\coprod\limits_{X^{0}}}
H_{x}^{0}(\Omega^{2})$ & $\rightarrow$ & $%
{\textstyle\coprod\limits_{X^{1}}}
H_{x}^{1}(\Omega^{2})$ & $\rightarrow$ & $%
{\textstyle\coprod\limits_{X^{2}}}
H_{x}^{2}(\Omega^{2})$ & $\rightarrow\cdots$\\
& $\vdots$ &  & $\vdots$ &  & $\vdots$ & \\\hline
\multicolumn{1}{c|}{} & $0$ &  & $1$ &  & $2$ & $p$%
\end{tabular}
\]
In particular, this interpretation reveals that we are actually facing a
spectral sequence which is supported exclusively in the fourth quadrant. So
far, this leaves open how the HKR\ isomorphism with supports interacts with
the rightward arrows. We will rectify this now.

\begin{theorem}
[Row-by-row Comparison]\label{marker_Thm_ComparisonOfRowsOnEOnePage}Let $k$ be
a field and $X/k$ a smooth scheme. The $(-q)$-th row on the $E_{1}$-page of
the Hochschild coniveau spectral sequence is isomorphic to the zero-th row of
the $E_{1}$-page of the coherent Cousin coniveau spectral sequence of
$\Omega^{q}$. That is: For every integer $q$,\ there is a canonical
isomorphism of chain complexes%
\[
\left.  ^{HH}E_{1}^{\bullet,-q}\right.  \overset{\sim}{\longrightarrow}\left.
^{Cous}E_{1}^{\bullet,0}(\Omega^{q})\right.  \text{.}%
\]
Entry-wise, this isomorphism is induced from the\ HKR isomorphism with supports.
\end{theorem}

\begin{proof}
The idea is the following: From Prop. \ref{marker_Prop_HKRWithSupport} we
already know that if $U$ is a smooth affine $k$-scheme and $Z$ a closed
subscheme, defined by a regular element $f\in\mathcal{O}(U)$, the boundary
maps in Hochschild homology with supports resp. local cohomology are
compatible. Thus, in order to prove this claim, we need to show that the
evaluation of the differential $d$ on the respective $E_{1}$-pages can be
reduced to evaluating such boundary maps.

To carry this out, recall that the equivalence in line Equation \ref{lHHC10}
is induced from the pullback $j_{x}:\operatorname*{Spec}\mathcal{O}%
_{X,x}\hookrightarrow X$ (Remark \ref{rmk_EquivalenceInducedByPullback}). For
every open subscheme neighbourhood $U\subseteq X$ containing $x\in X$, we get
a canonical factorization of $j^{\ast}$ as%
\begin{equation}
\operatorname*{Perf}(X)\longrightarrow\operatorname*{Perf}(U)\longrightarrow
\operatorname*{Perf}(\mathcal{O}_{X,x})\text{.}\label{lmuu1}%
\end{equation}
Each perfect complex on $\operatorname*{Spec}\mathcal{O}_{X,x}$ comes from all
sufficiently small open subschemes $U\ni x$, and they become isomorphic if and
only if this already happens on a sufficiently small open $U$. See
\cite[\S 4.1, especially p. 1247]{MR2354319} for a discussion of this. Suppose
$Z^{[0]}\supseteq Z^{[1]}\supseteq\ldots$ are closed subsets of $X$ such that
$\operatorname*{codim}_{X}(Z^{[p]})\geq p$. Then we get a filtration%
\[
\cdots\hookrightarrow\operatorname*{Perf}\nolimits_{Z^{[2]}}(X)\hookrightarrow
\operatorname*{Perf}\nolimits_{Z^{[1]}}(X)\hookrightarrow\operatorname*{Perf}%
\nolimits_{Z^{[0]}}(X)=\operatorname*{Perf}(X)\text{,}%
\]
analogous to the one in line \ref{lHHC4}. There is a partial order on the set
of all such filtrations $Z^{[0]}\supseteq Z^{[1]}\supseteq\ldots$, where
$Z^{\prime}\geq Z$ if and only if $Z^{\prime\lbrack p]}\supseteq Z^{[p]}$
holds for all $p$. We may form a spectral sequence based on this filtration,
as above. We will not have an analogue of Equation \ref{lHHC10} available in
this context, but we still get a convergent spectral sequence converging to
$HH_{-p-q}(X)$. Taking the colimit of this spectral sequence over all
filtrations $\{Z^{[\cdot]}\}$, we obtain the above spectral sequence. The
advantage of a spectral sequence of a filtration $\{Z^{[\cdot]}\}$ is that the
$Z^{[i]}$ are reduced closed subschemes with open subschemes as complements so
that the boundary maps $\partial$ of this spectral sequence correspond to a
localization sequence for a true open-closed complement. By the above colimit
argument, for any element $\alpha\in HH_{i}^{x}(\mathcal{O}_{X,x})$ for $x\in
X^{p}$, we may compute the differential $d$ on the $E_{1}$-page by performing
the computation on the $E_{1}$-page of a concrete filtration $\{Z^{[\cdot]}%
\}$. In particular, we may choose this filtration sufficiently fine such that
(1) we can work with an \emph{affine} neighbourhood $U\ni x$ in line
\ref{lmuu1}, and (2) such that there exists some $f\in\mathcal{O}_{X}(U)$ so
that the codimension $\geq1$ closed subset in $\mathcal{O}_{X,x}$ is cut out
by $f$, i.e. the stalk of the ideal sheaf $\mathcal{I}_{Z^{[p+1]},x}%
\subseteq\mathcal{O}_{X,x}$ is generated by $f$, and (3) the class $\alpha$ is
pulled back from some $\tilde{\alpha}\in HH_{i}^{x}(U)$. If one finds a $U$
such that (1) holds, one may need to shrink it further to ensure (2) holds as
well, and then even smaller to ensure (3). Then, inspecting Diagram
\ref{lHHC2}, we may compute the differential $d$ on the $\left.  ^{HH}%
E_{1}\right.  $-page by%
\[
\underset{\ni\alpha}{\coprod_{x\in X^{p}}HH_{i}^{x}(\mathcal{O}_{X,x}%
)}\leftarrow\underset{\ni\tilde{\alpha}}{\coprod_{x\in X^{p}}HH_{i}^{x}%
(U)}\overset{\partial}{\longrightarrow}HH_{i-1}^{\tilde{Z}}(U)\longrightarrow
\coprod_{x\in X^{p+1}}HH_{i-1}^{x}(\mathcal{O}_{X,x})\text{.}%
\]
As we assume that $X$ is smooth, this means that $d$ reduces to evaluating the
boundary map $\partial$ in the localization sequence corresponding to cutting
out a regular element $f$ from $\operatorname*{Spec}\mathcal{O}_{X}(U)$, a
smooth affine $k$-scheme. On the other hand, the $\left.  ^{Cous}E_{1}\right.
$-differential $d$ is given by Diagram \ref{comalg_fact_D34}, and we can
factor this analogously (with the same open subset $U$) as
\[
\underset{\ni HKR(\alpha)}{\coprod_{x\in X^{p}}H_{x}^{p}(\mathcal{O}%
_{X,x},\Omega^{\ast})}\leftarrow\underset{\ni HKR(\tilde{\alpha})}%
{\coprod_{x\in X^{p}}H_{x}^{p}(U,\Omega^{\ast})}\overset{\partial
}{\longrightarrow}H_{\tilde{Z}}^{p+1}(U,\Omega^{\ast})\longrightarrow
\coprod_{x\in X^{p+1}}H_{x}^{p+1}(\mathcal{O}_{X,x},\Omega^{\ast})\text{,}%
\]
By Prop. \ref{marker_Prop_HKRWithSupport} the boundary map $\partial$ in local
cohomology here is compatible with the corresponding boundary map in
Hochschild homology with supports. In other words: As the differential $d$ of
the coherent Cousin coniveau spectral sequence can be reduced in a completely
analogous way to the same affine open $U$, it follows that the
HKR\ isomorphism commutes with computing the boundary map in the respective
row of the $\left.  ^{Cous}E_{1}\right.  $-page. Our claim follows.
\end{proof}

We know from Corollary \ref{cor_CohomOfOmegaAgreesWithCousinCohomology} that
the cohomology of the $q$-th row agrees with the sheaf cohomology of
$\Omega^{-q}$. Thus, the $E_{2}$-page of the Hochschild coniveau spectral
sequence reads%
\begin{equation}%
\begin{tabular}
[c]{r|cccccc}%
$q$ & $\vdots$ &  & $\vdots$ &  &  & \\
$1$ & $0$ &  & $0$ &  & $0$ & \\
$0$ & $H^{0}(X,\Omega^{0})$ & $\quad$ & $H^{1}(X,\Omega^{0})$ & $\quad$ &
$\ddots$ & \\
$-1$ & $H^{0}(X,\Omega^{1})$ & $\quad$ & $H^{1}(X,\Omega^{1})$ & $\quad$ &
$H^{2}(X,\Omega^{1})$ & \\
$-2$ & $H^{0}(X,\Omega^{2})$ & $\quad$ & $H^{1}(X,\Omega^{2})$ & $\quad$ &
$H^{2}(X,\Omega^{2})$ & $\quad\cdots$\\
& $\vdots$ &  & $\vdots$ &  & $\vdots$ & \\\hline
\multicolumn{1}{c|}{} & $0$ &  & $1$ &  & $2$ & $p$%
\end{tabular}
\label{line_RelevantE2Page}%
\end{equation}

\begin{remark}
If $X$ is affine, say $X:=\operatorname*{Spec}A$, the higher sheaf cohomology
groups vanish, i.e. $H^{p}(X,\Omega^{-q})=0$ for all $p\neq0$. Thus, the
$E_{2}$-page has collapsed to a single column, and the convergence of the
spectral sequence just becomes the statement that%
\[
H^{0}(X,\Omega^{-q})\cong HH_{-q}(X)\qquad\text{for all }q\in\mathbf{Z}%
\]
and since the left-hand side agrees with $\Omega_{A/k}^{-q}$, we recover the
ordinary HKR\ isomorphism.
\end{remark}

At least if the base field $k$ has characteristic zero, this $E_{2}$-page
degenerates in general, even if $X$ is not affine. This follows from
incompatible Hodge degrees, as we explain in the following sub-section:

\subsubsection{Interplay with Hodge degrees}

Suppose $k$ is a field of characteristic zero. Then the Hochschild homology of
commutative $k$-algebras comes with a filtration, known either as Hodge or
$\lambda$-filtration. It was introduced by Gerstenhaber and\ Schack
\cite{MR917209} and\ Loday \cite{MR981743}. Weibel has extended this
filtration to separated Noetherian $k$-schemes\footnote{Actually to an even
broader class of schemes.} in his paper \cite{MR1469140}. One obtains a
canonical and functorial direct sum decomposition%
\begin{equation}
HH_{p}(X)=\bigoplus_{j}HH_{p}(X)^{(j)}\text{.}\label{lwuim1}%
\end{equation}
See \cite[Prop. 1.3]{MR1469140}. He also proved that $HH_{p}(X)^{(j)}%
=H^{j-p}(X,\Omega^{j})$ holds for smooth $k$-schemes $X$, providing a very
explicit relation to the usual Hodge decomposition \cite[Cor. 1.4]{MR1469140}.
Based on this, we can define a Hodge decomposition on Hochschild homology with
supports as well.

If the base field $k$ has characteristic zero, we may define%
\[
HH_{Z}(X)^{(j)}:=\operatorname*{hofib}\left(  HH\left(  X\right)
^{(j)}\longrightarrow HH\left(  X-Z\right)  ^{(j)}\right)  \text{.}%
\]
Since the usual Hochschild homology just splits into direct summands
functorially, as in Equation \ref{lwuim1}, the spectral sequence constructed
in \S \ref{marker_Sect_HochschildHomologyWithSupports} splits into a direct
sum of spectral sequences. The same happens to our HKR\ isomorphism with
supports, Prop. \ref{marker_Prop_HKRWithSupport}:

\begin{theorem}
[Compatibility with\ Hodge degrees]%
\label{marker_thm_interplaywithhodgeweights}Let $k$ be a field of
characteristic zero.

\begin{enumerate}
\item Suppose $R$ is a smooth $k$-algebra and $t_{1},\ldots,t_{n}$ a regular
sequence. Then the Hodge decomposition refines the isomorphism of Prop.
\ref{marker_Prop_HKRWithSupport} in the following fashion:
\[
HH_{i}^{(t_{1},\ldots,t_{n})}\left(  R\right)  ^{(j)}=\left\{
\begin{array}
[c]{ll}%
H_{(t_{1},\ldots,t_{n})}^{n}(R,\Omega^{n+i}) & \text{if }n+i=j\\
0 & \text{if }n+i\neq j\text{.}%
\end{array}
\right.
\]

\item Suppose $X$ is a smooth $k$-scheme. Then the spectral sequence of
\S \ref{marker_Sect_HochschildHomologyWithSupports} splits as a direct sum of
spectral sequences%
\[
\left.  \left(  ^{HH}E_{1}^{p,q}\right)  ^{(j)}\right.  :=\coprod_{x\in X^{p}%
}HH_{-p-q}^{x}(\mathcal{O}_{X,x})^{(j)}\Rightarrow HH_{-p-q}(X)^{(j)}\text{.}%
\]

\item Suppose $X$ is a smooth $k$-scheme. The spectral sequence $\left.
^{HH}E\right.  $ degenerates on the $E_{2}$-page, i.e. all differentials in
Figure \ref{line_RelevantE2Page} are zero.
\end{enumerate}
\end{theorem}

\begin{remark}
The results (2) and (3) are very close to well-known older results of Weibel.
For example, the spectral sequence in (2) has a large formal resemblance to
the one constructed in Weibel \cite[Prop. 1.2]{MR1469140}. However, he uses a
quite different construction to set up his spectral sequence. He uses the
hypercohomology spectral sequence of his sheaf approach to the Hochschild
homology of a scheme, as in line \ref{lmips4}. He also obtains an $E_{2}%
$-degeneration statement with essentially the same proof as ours, \cite[Cor.
1.4]{MR1469140}, for his spectral sequence.
\end{remark}

\begin{proof}
(1) The proof is exactly the same as we have given for Prop.
\ref{marker_Prop_HKRWithSupport}. By functoriality the Hodge decomposition can
be dragged through the entire proof systematically. Only the first step of the
proof changes, where one has to use that the ordinary HKR isomorphism is
supported entirely in the $n$-Hodge part:%
\[
\Omega_{R/k}^{n}\overset{\sim}{\longrightarrow}HH_{n}(R)^{(n)}\qquad
\text{and}\qquad HH_{n}(R)^{(j)}=0\text{\quad(for }j\neq n\text{).}%
\]
This is \cite[3.7. Th\'{e}or\`{e}me]{MR981743} or \cite[Thm. 4.5.12]%
{MR1217970}, for example.\newline(2) Immediate.\newline(3) The $E_{1}$-page of
the Hodge degree $j$ graded part takes the shape%
\[%
\begin{tabular}
[c]{r|cccccc}%
$q$ & $\vdots$ &  &  &  &  & \\
$1$ & $%
{\textstyle\coprod\limits_{x\in X^{0}}}
HH_{-1}^{x}(X)^{(j)}$ &  &  &  &  & \\
$0$ & $%
{\textstyle\coprod\limits_{x\in X^{0}}}
HH_{0}^{x}(X)^{(j)}$ & $\rightarrow$ & $%
{\textstyle\coprod\limits_{x\in X^{1}}}
HH_{-1}^{x}(X)^{(j)}$ & $\rightarrow$ & $\ddots$ & \\
$-1$ & $%
{\textstyle\coprod\limits_{x\in X^{0}}}
HH_{1}^{x}(X)^{(j)}$ & $\rightarrow$ & $%
{\textstyle\coprod\limits_{x\in X^{1}}}
HH_{0}^{x}(X)^{(j)}$ & $\rightarrow$ & $%
{\textstyle\coprod\limits_{x\in X^{2}}}
HH_{-1}^{x}(X)^{(j)}$ & \\
$-2$ & $%
{\textstyle\coprod\limits_{x\in X^{0}}}
HH_{2}^{x}(X)^{(j)}$ & $\rightarrow$ & $%
{\textstyle\coprod\limits_{x\in X^{1}}}
HH_{1}^{x}(X)^{(j)}$ & $\rightarrow$ & $%
{\textstyle\coprod\limits_{x\in X^{2}}}
HH_{0}^{x}(X)^{(j)}$ & $\rightarrow\cdots$\\
& $\vdots$ &  & $\vdots$ &  & $\vdots$ & \\\hline
\multicolumn{1}{c|}{} & $0$ &  & $1$ &  & $2$ & $p$%
\end{tabular}
\]
and applying the refined HKR isomorphism with supports to these entries, part
(1) of our claim implies that all rows vanish except for the row with $q=-j$.
As a result, it follows that the spectral sequence degenerates. As our
original spectral sequence $\left.  ^{HH}E^{\bullet,\bullet}\right.  $ is just
a direct sum of these $\left.  \left(  ^{HH}E_{1}^{p,q}\right)  ^{(j)}\right.
$, it follows that all differentials of the $\left.  ^{HH}E_{2}\right.  $-page
must be zero (because the differentials then also are direct sums of the
differentials of the individual $\left.  \left(  ^{HH}E_{1}^{p,q}\right)
^{(j)}\right.  $, so they cannot map between different Hodge graded parts).
\end{proof}

This also leads to a version of the `Gersten resolution', which differs from
the classical coherent Cousin resolution in the way it is constructed, but not
in its output. For an abelian group $A$, we write $(i_{x})_{\ast}A$ to denote
the constant sheaf $A$ on the scheme point $x$ and extended by zero elsewhere.

\begin{corollary}
[\textquotedblleft Hochschild--Cousin resolution\textquotedblright]Suppose
$X/k$ is a smooth scheme over a field $k$. Then%
\begin{equation}
\mathcal{HH}_{n}\overset{\sim}{\longrightarrow}\left[  \coprod_{x\in X^{0}%
}(i_{x})_{\ast}HH_{n}^{x}(\mathcal{O}_{X,x})\longrightarrow\coprod_{x\in
X^{1}}(i_{x})_{\ast}HH_{n-1}^{x}(\mathcal{O}_{X,x})\longrightarrow
\cdots\right]  _{0,n}\label{lve1}%
\end{equation}
is a quasi-isomorphism of sheaves, and yields a flasque resolution of the
Zariski sheaf $\mathcal{HH}_{n}\cong\Omega^{n}$ for any $n$.
\end{corollary}

It will also be possible to prove this using a transfer-based method \`{a} la
\cite{MR1466971}, \cite{MR2158763}.

\begin{proof}
We give two proofs: (1) We can use Theorem
\ref{marker_Thm_ComparisonOfRowsOnEOnePage}, proving that the complex of
sheaves on the right is canonically isomorphic to the coherent Cousin
resolution of Corollary \ref{marker_Cor_ExistenceCousinResolution}. The latter
is a resolution even under far less restrictive assumptions than smoothness,
relying on the tools of \cite{MR0222093}. (2) Alternatively, suppose $k$ is of
characteristic zero. We may consider the Hochschild coniveau spectral sequence
$\left.  ^{HH}E^{\bullet,\bullet}\right.  $ of $U$ for any open immersion
$U\hookrightarrow X$. We obtain a presheaf of spectral sequences, which we
sheafify in the Zariski topology. We denote it by $\left.  ^{HH}%
\mathcal{E}^{p,q}\right.  $. As this process also sheafifies the limit of the
spectral sequence, we get a spectral sequence of sheaves%
\[
\left.  ^{HH}\mathcal{E}_{1}^{p,q}\right.  :=\coprod_{x\in X^{p}}(i_{x}%
)_{\ast}HH_{-p-q}^{x}(\mathcal{O}_{X,x})\Rightarrow\mathcal{HH}_{-p-q}%
(X)\text{.}%
\]
The direct sum decomposition of Theorem
\ref{marker_thm_interplaywithhodgeweights} is functorial in pullback along
opens, so $\left.  ^{HH}\mathcal{E}^{p,q}\right.  $ degenerates on the second
page. Restrict to the direct summand of the $\mathcal{HH}_{n}$ which we are
interested in. This leaves only one non-zero entry on the $E_{2}$-page. The
sheaves in Equation \ref{lve1} are clearly flasque and since the $E_{2}$-page
has just one entry, the resolution property follows easily (it implies the
exactness in all higher degrees).
\end{proof}

\subsubsection{\label{subsect_ChernCharWithSupps}Chern character with
supports}

Let $X/k$ be a smooth scheme and $x\in X$ a scheme point of codimension
$\operatorname*{codim}_{X}\overline{\{x\}}=p$. We will define a \emph{Chern
character with supports},%
\[
\mathcal{T}(x):K_{m}(\kappa(x))\longrightarrow H_{x}^{p}(X,\Omega
^{p+m})\text{.}%
\]
The definition is simple: As $X/k$ is smooth, d\'{e}vissage and excision for
$K$-theory yield a canonical isomorphism $K(\kappa(x))\cong K^{x}%
(\mathcal{O}_{X,x})\cong K^{x}(X)$, where $K^{x}$ denotes $K$-theory with
support in $\{x\}$. The spectrum-level Chern character $K\rightarrow HH$
(\`{a} la McCarthy \cite{MR1275967}, in the version of Keller \cite{MR1667558}%
) induces a map $K^{x}(X)\rightarrow HH^{x}(X)$. Excision and the HKR
isomorphism with supports for Hochschild homology then yield $HH_{m}%
^{x}(X)\cong HH_{m}^{x}(\mathcal{O}_{X,x})\cong H_{x}^{p}(\mathcal{O}%
_{X,x},\Omega^{p+m})$. We call the composition of these maps $\mathcal{T}(x)$.

\begin{example}
If $X/k$ is an integral smooth scheme with generic point $\eta$, the map
$\mathcal{T}(\eta):K_{\ast}(k\left(  X\right)  )\rightarrow\Omega_{k\left(
X\right)  /k}^{\ast}$ is just the trace map $K\rightarrow HH$, applied to the
rational function field of $X$.
\end{example}

\begin{proposition}
Let $X/k$ be a Noetherian scheme over a field $k$.

\begin{enumerate}
\item Then the Chern character (a.k.a. trace map)%
\[
K\left(  X\right)  \longrightarrow HH\left(  X\right)
\]
induces a morphism of spectral sequences $\left.  ^{K}E^{\bullet,\bullet
}\right.  \rightarrow\left.  ^{HH}E^{\bullet,\bullet}\right.  $, where
$\left.  ^{K}E^{\bullet,\bullet}\right.  $ denotes Balmer's coniveau spectral
sequence of \cite{MR2439430}.

\item If $X$ is smooth over $k$, we may compose it with the comparison map to
the coherent Cousin spectral sequence, and then the map between the $E_{1}%
$-pages is the Chern character for supports:%
\begin{align*}
\mathcal{T}(x):\left.  ^{K}E_{1}^{p,q}\right.    & \longrightarrow\left.
^{Cous}E_{1}^{p,0}(\Omega^{-q})\right.  \\
K_{-p-q}(\kappa(x))  & \longrightarrow H_{x}^{p}(X,\Omega^{-q})
\end{align*}
Here we have used that Balmer's coniveau spectral sequence agrees with
Quillen's from \cite{MR0338129} thanks to the smoothness assumption.
\end{enumerate}
\end{proposition}

\begin{proof}
(1) This is true by functoriality. We have constructed $\left.  ^{HH}%
E^{\bullet,\bullet}\right.  $ based on the same filtration that Balmer uses
for $K$-theory (cf. Prop. \ref{marker_Prop_ConiveauHHSpecseq}). The Chern
character $K\rightarrow HH$ is compatible with the respective localization
sequences, and thus the trace functorially induces a morphism of spectral
sequences. On the $E_{1}$-page, this morphism induces morphisms%
\[
T_{p,q}:\coprod_{x\in X^{p}}K_{-p-q}^{x}(\mathcal{O}_{X,x})\longrightarrow
\coprod_{x\in X^{p}}HH_{-p-q}^{x}(\mathcal{O}_{X,x})
\]
and by comparing supports, these morphisms are Cartesian in the sense that the
direct summand of $x\in X^{p}$ on the left maps exclusively to the direct
summand belonging to the same $x$ on the right-hand side. That is,
$T_{p,q}=\sum T_{p,q}^{x}$ with%
\[
T_{p,q}^{x}:K_{-p-q}^{x}(\mathcal{O}_{X,x})\longrightarrow HH_{-p-q}%
^{x}(\mathcal{O}_{X,x})\text{.}%
\]
(2) Now, assume that $X/k$ is smooth. We may then, equivalently, use the
$K$-theory of coherent sheaves on the left, and then d\'{e}vissage. So, using
the d\'{e}vissage isomorphism on the left-hand side in the above equation, and
the HKR isomorphism with supports on the right-hand side, we obtain%
\[
T_{p,q}^{x\prime}:K_{-p-q}(\kappa(x))\cong K_{-p-q}^{x}(\mathcal{O}%
_{X,x})\longrightarrow HH_{-p-q}^{x}(\mathcal{O}_{X,x})\cong H_{x}%
^{p}(\mathcal{O}_{X,x},\Omega^{-q})\text{.}%
\]
Using excision of the right-hand side, this transforms into the definition of
$\mathcal{T}(x)$.
\end{proof}

\section{\label{subsect_CubicallyDecomposedAlgebrasAndTheirDerivates}Cubical
algebras and their residue symbol}

\subsection{Introduction to the comparison problem}

The next sections will be devoted to relating our Hochschild--Cousin complex
to the residue theory of \cite{MR0227171}, \cite{MR565095}. Let us briefly
sketch the story in dimension one, in order to motivate how we shall proceed.

\subsubsection{Residue \`{a} la Tate}

In \cite{MR0227171} Tate defines the residue of a rational $1$-form on an
integral curve $X/k$ at a closed point $x\in X$ as follows: Let $\widehat
{\mathcal{K}}_{X,x}:=\operatorname*{Frac}\widehat{\mathcal{O}}_{X,x}$ denote
the local field at the point $x$ (it can also be constructed by completing the
function field with respect to the metric of the valuation associated to $x$).
Now $\widehat{\mathcal{K}}_{X,x}$ is a linearly locally compact topological
$k$-vector space. It has infinite dimension. Any rational functions $f,g\in
k(X)$ act on it by continuous $k$-linear endomorphisms, i.e. we could also
read them as elements $f,g\in\operatorname*{End}\nolimits_{k}^{cts}%
(\widehat{\mathcal{K}}_{X,x})$. If $P^{+}$ denotes any projector splitting the
inclusion $\widehat{\mathcal{O}}_{X,x}\hookrightarrow\widehat{\mathcal{K}%
}_{X,x}$, Tate shows that the commutator $[P^{+}f,g]$ has sufficiently small
image to define a trace on it, and defines a map
\[
\Omega_{\widehat{\mathcal{K}}_{X,x}/k}^{1}\longrightarrow k\text{,}%
\qquad\qquad f\,\mathrm{d}g\longmapsto\operatorname*{Tr}[P^{+}f,g]\text{.}%
\]
He shows that this agrees with the usual residue of the $1$-form
$\omega:=f\,\mathrm{d}g$ at $x$. In \cite{MR565095}, \cite{BFM_Conformal},
\cite{MR1988970} this construction gets interpreted in terms of a central
extension of Lie algebras, giving a Lie algebra cohomology class%
\begin{equation}
\phi_{Tate}\in H_{\operatorname*{Lie}}^{2}((\widehat{\mathcal{K}}_{X,x}%
)_{Lie},k)\text{.}\label{lvb2}%
\end{equation}

\subsubsection{Residue via localization}

A completely different approach to think about the residue might be to use the
boundary map $\partial$ in Keller's localization sequence,%
\begin{equation}
HH_{1}(\widehat{\mathcal{O}}_{X,x})\longrightarrow HH_{1}(\widehat
{\mathcal{K}}_{X,x})\overset{\partial}{\longrightarrow}HH_{0}^{x}%
(\widehat{\mathcal{O}}_{X,x})\overset{\operatorname*{Tr}}{\longrightarrow
}k\text{.}\label{lvb1}%
\end{equation}
Via the HKR\ comparison map, $\Omega_{\widehat{\mathcal{K}}_{X,x}/k}%
^{1}\rightarrow HH_{1}(\widehat{\mathcal{K}}_{X,x})$, this also produces a map
$\Omega_{\widehat{\mathcal{K}}_{X,x}/k}^{1}\rightarrow k$, which
\textit{should} be the residue.

\subsubsection{Comparison}

We will connect both approaches now, using the following approach: (1) one
interprets the r\^{o}le of local compactness in terms of a Tate category (we
recall this in \S \ref{sect_TateCategories}). Then, using \cite[Theorem
5]{bgwTateModule}, one obtains%
\[
\operatorname*{End}\nolimits_{k}^{cts}(\widehat{\mathcal{K}}_{X,x}%
)\cong\operatorname*{End}\nolimits_{\mathsf{Tate}_{\aleph_{0}}(k)}%
(\widehat{\mathcal{K}}_{X,x})\text{,}%
\]
i.e. the endomorphism algebra of Tate's approach becomes \textquotedblleft
representable\textquotedblright\ as a genuine endomorphism algebra in
the\ Tate category. (2) Identifying an element in $\widehat{\mathcal{K}}%
_{X,x}$ by its multiplication endomorphism, we get a morphism%
\[
\widehat{\mathcal{K}}_{X,x}\longrightarrow\operatorname*{End}%
\nolimits_{\mathsf{Tate}_{\aleph_{0}}(k)}(\widehat{\mathcal{K}}_{X,x}%
)\qquad\text{and thus}\qquad HH(\widehat{\mathcal{K}}_{X,x})\longrightarrow
HH(\operatorname*{End}\nolimits_{\mathsf{Tate}_{\aleph_{0}}(k)}(\widehat
{\mathcal{K}}_{X,x}))\text{.}%
\]
Now, it remains to identify a counterpart of the boundary map $\partial$ in
line \ref{lvb1}, defined on $HH(\widehat{\mathcal{K}}_{X,x})$, on the
right-hand side. We will show that this counterpart is given by Tate's
construction on the right-hand side. The bridge to switch between\ Hochschild
and Lie homology was already set up in \cite{olliTateRes}.

Beilinson has generalized Tate's approach to arbitrary dimension $n$ in
\cite{MR565095}. We shall directly work in this generality. In general, the
Lie cohomology class in line \ref{lvb2} becomes a class in
$H_{\operatorname*{Lie}}^{n+1}(-)$ and the Tate category needs to be replaced
by $n$-Tate categories.

\subsection{Definition of the abstract symbol}

\begin{definition}
[\cite{MR565095}]\label{def_BeilinsonNFoldAlgebra}Let $k$ be a field. A
\emph{Beilinson }$n$\emph{-fold cubical algebra} is

\begin{enumerate}
\item an associative $k$-algebra $A$;

\item two-sided ideals $I_{i}^{+},I_{i}^{-}$ such that $I_{i}^{+}+I_{i}^{-}=A
$ for $i=1,\ldots,n$;

\item define $I_{i}^{0}=I_{i}^{+}\cap I_{i}^{-}$ and call $I_{tr}%
:=\bigcap_{i=1,\ldots,n}I_{i}^{0}$ the \emph{trace-class} operators of $A$.
\end{enumerate}

A \emph{trace} on an $n$-fold cubical algebra is a morphism $\tau
:I_{tr}/[I_{tr},A]\longrightarrow k$.
\end{definition}

Although the following property is stronger than necessary to develop the
formalism, it will be handy to single out a particularly friendly type of such algebras:

\begin{definition}
\label{def_CubicalAlgebraGood}We say that $(A,(I_{i}^{\pm}))$ is \emph{good}
if for every $c=1,\ldots,n$ the intersection $I_{1}^{0}\cap\cdots\cap
I_{c}^{0}$ is locally bi-unital (in the sense of Definition
\ref{def_LocallyBiUnital}).
\end{definition}

This assumption will be made for the following reasons: Firstly, the
simplifications due to Wodzicki's Prop. \ref{mai_Prop_Wodzicki_HUnitality}
apply, and secondly we can easily define a whole hierarchy of further cubical
algebras, which we may imagine as going down dimension by dimension.

\begin{lemma}
\label{RES_Lemma_DerivedCubicalAlgebra}Let $(A,(I_{i}^{\pm})_{i=1,\ldots,n})$
be a good $n$-fold cubical algebra. Define%
\[
A^{\prime}:=I_{1}^{0}\qquad\text{and}\qquad J_{i-1}^{\pm}:=I_{i}^{\pm}\cap
I_{1}^{0}%
\]
for $i=2,\ldots,n$.

\begin{enumerate}
\item Then $(A^{\prime},(J_{i})_{i=1,\ldots,n-1})$ is a good $(n-1)$-fold
cubical algebra and both algebras have the same trace-class operators.

\item The natural homomorphism $I_{tr}/[I_{tr},I_{tr}]\overset{\sim
}{\longrightarrow}I_{tr}/[I_{tr},A]$ is an isomorphism.
\end{enumerate}
\end{lemma}

The proof of the second claim is very easy, but based on a trick which might
not be particularly obvious if one only looks at the claim.

\begin{proof}
\textit{(Step 1)} Clearly $A^{\prime}=I_{1}^{0}$ is a (non-unital) associative
$k$-algebra. For $i=1,\ldots,n-1$ we compute%
\[
J_{i}^{+}+J_{i}^{-}=(I_{i+1}^{+}\cap I_{1}^{0})+(I_{i+1}^{-}\cap I_{1}%
^{0})\subseteq A\cap I_{1}^{0}=I_{1}^{0}\text{.}%
\]
For the converse inclusion, let $x\in I_{1}^{0}$ be given. Since
$(A,(I_{i}^{\pm}))$ is assumed good, there is a local left unit for the
singleton finite set $\{x\}$ in $I_{1}^{0}$, say $x=ex$ with $e\in I_{1}^{0}$.
By $I_{i+1}^{+}+I_{i+1}^{-}=A$, write $x=x^{+}+x^{-}$ with $x^{\pm}\in
I_{i+1}^{\pm}$. Thus, $x=ex=e(x^{+}+x^{-})=ex^{+}+ex^{-}$ and since $I_{1}%
^{0}$ is a two-sided ideal, $ex^{s}\in I_{1}^{0}\cap I_{i+1}^{s}$ for
$s\in\{+,-\}$. As this works for all $x\in I_{1}^{0}$, we get $I_{1}%
^{0}\subseteq J_{i}^{+}+J_{i}^{-}$. Thus, $A^{\prime}$ is an $(n-1)$-fold
cubical algebra. Note that this argument would work just as well with local
right units. The trace-class operators are%
\[
I_{tr}(A^{\prime})=%
{\textstyle\bigcap_{i=1,\ldots,n-1}}
J_{i}^{+}\cap J_{i}^{-}=%
{\textstyle\bigcap_{i=1,\ldots,n-1}}
I_{i+1}^{+}\cap I_{i+1}^{-}\cap I_{1}^{0}=%
{\textstyle\bigcap_{i=1,\ldots,n}}
I_{i}^{+}\cap I_{i}^{-}=I_{tr}(A)\text{.}%
\]
\textit{(Step 2)} We need to check that $(A^{\prime},(J_{i}^{\pm}%
)_{i=1,\ldots,n-1})$ is good, i.e. the local bi-unitality of%
\[
J_{1}^{0}\cap\cdots\cap J_{c}^{0}=(I_{1}^{0}\cap I_{i+1}^{0})\cap\cdots
\cap(I_{1}^{0}\cap I_{c+1}^{0})=I_{1}^{0}\cap\cdots\cap I_{c+1}^{0}%
\]
for any $c=1,\ldots,n-1$. And these are locally bi-unital since $(A,(I_{i}%
^{\pm}))$ is good. This completes the proof of the first claim.\newline%
\textit{(Step 3)} It remains to prove the second claim. In fact, this is true
as soon as $A$ is any associative algebra and $I$ any locally right unital
two-sided ideal: It is clear that $[I,I]\subseteq\lbrack A,I]$ and we shall
show the reverse inclusion: Let any $t\in I$ and $a\in A$ be given. Let $e$ be
a local right unit for the singleton set $\{t\}\subset I$. By the ideal
property, $ea\in I$ and $at\in I$. Thus, the left-hand side of the following
equation lies in $[I,A]$, namely
$[ea,t]-[e,at]=eat-tea-eat+ate=ate-tea\overset{(\ast)}{=}at-ta=[a,t]$, where
we have used the local right unit property for $(\ast)$. Without right
unitality, there would have been no chance for this kind of argument.
\end{proof}

Suppose $A$ is a good $n$-fold cubical algebra with a trace $\tau$. In the
paper \cite{olliTateRes} a canonical functional $\phi_{C}:HC_{n}(A)\rightarrow
k$ was constructed, functorial in morphisms of cubical algebras. We will give
a self-contained exposition of this construction:

Let $(A_{n},(I_{i}^{\pm}),\tau)$ be a good $n$-fold cubical algebra over $k $
with a trace $\tau$. We define%
\begin{equation}
A_{n-1}:=I_{1}^{0}\qquad\text{and}\qquad J_{i-1}^{\pm}:=I_{i}^{\pm}\cap
I_{1}^{0}\text{,}\label{lmai_5}%
\end{equation}
where $i=0,\ldots,n-1$. By Lemma \ref{RES_Lemma_DerivedCubicalAlgebra} this is
again a good cubical algebra over $k$. Define%
\begin{equation}
\Lambda:A_{n}\longrightarrow A_{n}/A_{n-1}\text{,}\qquad x\longmapsto
x^{+}\label{lEquationToeplitzLAMBDA}%
\end{equation}
where $x=x^{+}+x^{-}$ is any decomposition with $x^{\pm}\in I_{1}^{\pm}$.

\begin{remark}
This map is \emph{not} the natural quotient map!
\end{remark}

\begin{lemma}
The map $\Lambda$ is well-defined.
\end{lemma}

\begin{proof}
By the axiom $I_{1}^{+}+I_{1}^{-}=A_{n}$ of a cubical algebra, such an element
$x^{+}$ always exists. It is not unique, but if $x^{\ast}$ is another choice,
by the exactness of $I_{1}^{0}\rightarrow I_{1}^{+}\oplus I_{1}^{-}\rightarrow
A_{n}\rightarrow0$ we have $x^{+}-x^{\ast}\in I_{1}^{+}\cap I_{1}^{-}%
=I_{1}^{0}=A_{n-1}$.
\end{proof}

As $A_{n-1}$ is a two-sided ideal in $A_{n}$, we get an exact sequence of
associative algebras%
\begin{equation}
0\longrightarrow A_{n-1}\longrightarrow A_{n}\overset{\operatorname*{quot}%
}{\longrightarrow}A_{n}/A_{n-1}\longrightarrow0\text{.}\label{lml_28}%
\end{equation}
This sequence induces a long exact sequence in Hochschild homology via Theorem
\ref{Thm_WodzickiLongExactSeqUsingExcision}, and we shall denote the boundary
map by $\delta$.

\begin{definition}
[\cite{olliTateRes}]\label{marker_Def_map_d}Define%
\begin{equation}
d:HH(A_{n})\overset{\Lambda}{\longrightarrow}HH(A_{n}/A_{n-1})\overset{\delta
}{\longrightarrow}\Sigma HH(A_{n-1})\text{.}\label{lml_40}%
\end{equation}

\end{definition}

We can repeat this construction and obtain a morphism:

\begin{definition}
[\cite{olliTateRes}]\label{Def_AbstractHochschildResidueSymbol}Suppose
$(A,(I_{i}^{\pm}))$ is good $n$-fold cubical algebra over $k$ with a trace
$\tau$. Define%
\begin{equation}
\phi_{C}:HH_{n}(A)\longrightarrow HH_{0}(I_{tr})\longrightarrow k\text{,}%
\qquad\gamma\mapsto\tau\underset{n\text{ times}}{\underbrace{d\circ\cdots\circ
d}}\gamma\text{.}\label{lmai_37}%
\end{equation}
We call this the \emph{abstract Hochschild symbol} of $A$.
\end{definition}

\section{\label{sect_TateCategories}Relation with Tate categories}

\subsection{Exact categories}

We will give a brief, almost self-contained review of the formalism of Tate
categories. Let $\mathcal{C}$ be an exact category \cite{MR2606234}. In
particular, among its morphisms, we reserve the symbols \textquotedblleft%
$\hookrightarrow$\textquotedblright\ resp. \textquotedblleft%
$\twoheadrightarrow$\textquotedblright\ for admissible monics resp. admissible
epics. An admissible sub-object refers to a sub-object such that the inclusion
is an admissible monic. We write $\mathcal{C}^{ic}$ to denote the idempotent
completion of $\mathcal{C}$.

\begin{lemma}
\label{TMT_LemmaIdemCompletionOfSplitExactIsSplitExact}Let $\mathcal{C}$ be an
exact category.

\begin{enumerate}
\item The idempotent completion $\mathcal{C}^{ic}$ has a canonical exact
structure such that $\mathcal{C}\hookrightarrow\mathcal{C}^{ic}$ is an exact
functor reflecting exactness.

\item If $\mathcal{C}$ is split exact, so is $\mathcal{C}^{ic}$.
\end{enumerate}
\end{lemma}

\begin{proof}
(1) \cite[Prop. 6.13]{MR2606234}, (2) \cite[Lemma 11]{bgwTateModule}.
\end{proof}

Next, recall that every exact category can $2$-universally be embedded into a
Grothendieck abelian category, called $\mathsf{Lex}(\mathcal{C})$, such that
it becomes an extension-closed full sub-category and a kernel-cokernel pair is
exact if and only if this is so in $\mathsf{Lex}(\mathcal{C})$, in the
classical sense of exactness. This is known as the \textit{Quillen embedding}
$\mathcal{C}\hookrightarrow\mathsf{Lex}(\mathcal{C})$. See \cite[\S A.7]%
{MR1106918}, \cite[\S 1.2]{MR2079996} or \cite[Appendix A]{MR2606234} for a
detailed treatment.

\subsection{\label{subsect_TateCatsViaIndDiagrams}Ind- and Pro-categories,
Tate categories}

Let $\kappa$ be an infinite cardinal. An \emph{admissible Ind-diagram of
cardinality} $\kappa$ is a functor $X:I\rightarrow\mathcal{C}$ with $I$ a
directed poset of cardinality at most $\kappa$ which maps the arrows of $I$ to
admissible monics in $\mathcal{C}$. Since $\mathsf{Lex}(\mathcal{C})$ is
co-complete, any such diagram has a colimit in this category. Thus, the
following definition makes sense:

\begin{definition}
\label{def_BasicAdmissibleIndAndProObjs}Let $\mathcal{C}$ be an exact category
and $\kappa$ an infinite cardinal.

\begin{enumerate}
\item The essential image of all admissible Ind-diagrams of cardinality
$\kappa$ in $\mathsf{Lex}(\mathcal{C})$ is the category of \emph{admissible
Ind-objects}, and is denoted by $\mathsf{Ind}_{\kappa}^{a}(\mathcal{C})$.

\item Define $\mathsf{Pro}_{\kappa}^{a}(\mathcal{C}):=\mathsf{Ind}_{\kappa
}^{a}(\mathcal{C}^{op})^{op}$, the exact category of admissible Pro-objects.
\end{enumerate}
\end{definition}

See also \cite[\S 4]{TateObjectsExactCats} for a different perspective on
Pro-objects. Definition \ref{def_BasicAdmissibleIndAndProObjs} is due to
Bernhard Keller for $\kappa=\aleph_{0}$ \cite{MR1052551}. See \cite[\S 3]%
{TateObjectsExactCats} for a detailed treatment of the general case. One shows
that $\mathsf{Ind}_{\kappa}^{a}(\mathcal{C})$ is extension-closed inside
$\mathsf{Lex}(\mathcal{C})$ and therefore carries a canonical exact structure
induced from $\mathsf{Lex}(\mathcal{C})$ \cite[Lemma 10.20]{MR2606234},
\cite[Thm. 3.7]{TateObjectsExactCats}, and the functor $\mathcal{C}%
\hookrightarrow\mathsf{Ind}_{\kappa}^{a}(\mathcal{C})$, sending objects to the
constant diagram, is exact. We write $\mathsf{Ind}^{a}(\mathcal{C})$,
$\mathsf{Pro}^{a}(\mathcal{C})$ etc. without a qualifier $\kappa$ if we do not
wish to impose any restriction on the cardinality.

For the sake of legibility, we shall henceforth mostly drop $\kappa$ from the
notation, but all these results would also be valid for the variants
constrained by an infinite cardinal $\kappa$ bound. Precise information about
such variations can always be found in the cited sources.

\begin{definition}
\label{Def_TateObjectAndLattices}Consider the commutative square of exact
categories and exact functors,%
\begin{equation}%
\bfig\Square(0,0)[\mathcal{C}`\mathsf{Ind}^{a}\mathcal{C}`\mathsf{Pro}%
^{a}\mathcal{C}`\mathsf{Ind}^{a}\mathsf{Pro}^{a}\mathcal{C}\text{.};```]
\efig
\label{lmai_35}%
\end{equation}
A \emph{lattice} in an object $X\in\mathsf{Ind}^{a}\mathsf{Pro}^{a}%
(\mathcal{C})$ is an admissible sub-object $L\hookrightarrow X$ such that
$L\in\mathsf{Pro}^{a}(\mathcal{C)}$ and $X/L\in\mathsf{Ind}^{a}(\mathcal{C})
$. \cite[\S 5]{TateObjectsExactCats}

\begin{enumerate}
\item The category of elementary Tate objects, denoted by $\mathsf{Tate}%
^{el}(\mathcal{C})$ or $\left.  1\text{-}\mathsf{Tate}^{el}(\mathcal{C}%
)\right.  $, is the full sub-category of $\mathsf{Ind}^{a}\mathsf{Pro}%
^{a}(\mathcal{C})$ of objects having a lattice (this lattice is not part of
the data). The category of Tate objects, denoted $\mathsf{Tate}(\mathcal{C})$
or $\left.  1\text{-}\mathsf{Tate}(\mathcal{C})\right.  $, is the idempotent
completion $\mathsf{Tate}^{el}(\mathcal{C})^{ic}$. \cite[\S 5, Thm.
5.6]{TateObjectsExactCats}

\item More generally, define $\left.  n\text{-}\mathsf{Tate}^{el}%
(\mathcal{C})\right.  :=\mathsf{Tate}^{el}\left(  \,\left.  (n-1)\text{-}%
\mathsf{Tate}(\mathcal{C})\right.  \,\right)  $ and $\left.  n\text{-}%
\mathsf{Tate}(\mathcal{C})\right.  :=\left.  n\text{-}\mathsf{Tate}%
^{el}(\mathcal{C})^{ic}\right.  $ as its idempotent completion. We will refer
to the objects of these categories as $n$\emph{-Tate objects}.

\item Once we fix an elementary\ Tate object $X\in\mathsf{Tate}^{el}%
(\mathcal{C})$, the lattices form a poset, called the \emph{Sato Grassmannian}
$Gr(X)$, by defining $L^{\prime}\leq L$ whenever $L^{\prime}\hookrightarrow L$
is an admissible monic.
\end{enumerate}
\end{definition}

We refer to \cite{TateObjectsExactCats} for a detailed treatment. See
\cite{MR2656941}, \cite{MR2872533} for earlier work on iterating Tate
categories. The following facts are of essential importance:

\begin{theorem}
\label{Thm_SatoGrassmannianProperties}Let $\mathcal{C}$ be an exact category.

\begin{enumerate}
\item If $L^{\prime}\hookrightarrow L$ are lattices in an object
$X\in\mathsf{Tate}^{el}(\mathcal{C})$, then $L/L^{\prime}\in\mathcal{C}$.

\item Suppose $\mathcal{C}$ is idempotent complete. Then the poset $Gr(V)$ is
directed and co-directed, i.e. any finite set of lattices has a common
sub-lattice and a common over-lattice.

\item If $X\in\mathsf{Tate}^{el}(\mathcal{C})$ lies in the sub-categories of
Pro-objects and Ind-objects simultaneously, we have $X\in\mathcal{C}$. If
$\mathcal{C}$ is idempotent complete, the same holds true for $X\in
\mathsf{Tate}(\mathcal{C})$.
\end{enumerate}
\end{theorem}

\begin{proof}
(1) \cite[Prop. 6.6]{TateObjectsExactCats}, (2) \cite[Thm. 6.7]%
{TateObjectsExactCats}, (3) \cite[Prop. 5.9]{TateObjectsExactCats},
\cite[Prop. 5.28]{TateObjectsExactCats}.
\end{proof}

There are also some basic factorizations for in- and out-going morphisms under
the inclusions of categories in Diagram \ref{lmai_35} and lattices:

\begin{proposition}
\label{Prop_LatticeLeftFiltAndRightFilt}Let $\mathcal{C}$ be an exact category.

\begin{enumerate}
\item Every morphism $Y\overset{a}{\longrightarrow}X$ in $\mathsf{Tate}%
^{el}(\mathcal{C})$ with $Y\in\mathsf{Pro}^{a}(\mathcal{C})$ can be factored
as $Y\overset{\tilde{a}}{\rightarrow}L\hookrightarrow X$ with $L$ a lattice in
$X$.

\item Every morphism $X\overset{a}{\longrightarrow}Y$ in $\mathsf{Tate}%
^{el}(\mathcal{C})$ with $Y\in\mathsf{Ind}^{a}(\mathcal{C})$ can be factored
as $X\twoheadrightarrow X/L\overset{\tilde{a}}{\rightarrow}Y$ with $L$ a
lattice in $X$.
\end{enumerate}
\end{proposition}

\begin{proof}
A complete proof is given in \cite[Proposition 2.7]{bgwRelativeTateObjects}.
\end{proof}

\subsection{\label{subsect_QuotientExactCategories}Quotient exact categories}

If $\mathcal{C}\hookrightarrow\mathcal{D}$ is an exact sub-category, this does
not yet suffice to define a quotient exact category \textquotedblleft%
$\mathcal{D}/\mathcal{C}$\textquotedblright. However, as was shown by
Schlichting, a sufficient condition for such a category to exist is that
$\mathcal{C}\hookrightarrow\mathcal{D}$ is \textquotedblleft left or right $s
$-filtering\textquotedblright. This is a technical notion and we refer to the
original paper \cite{MR2079996}, or for a quick review to \cite[\S 2]%
{TateObjectsExactCats}. Ultimately, $\mathcal{D}/\mathcal{C}$ arises as the
localization $\mathcal{D}[\Sigma^{-1}]$, where $\Sigma$ is the smallest class
of morphisms encompassing those with (1) admissible epics with kernels in
$\mathcal{C}$, (2) admissible monics with cokernels in $\mathcal{C}$, (3) and
is closed under composition. The left/right $s$-filtering conditions imply the
existence of a calculus of left/right fractions. As was observed by T.
B\"{u}hler, in the left $s$-filtering case, these conditions also imply that
inverting admissible epics with kernels in $\mathcal{C}$ is sufficient (see
\cite[Prop. 2.19]{TateObjectsExactCats} for a careful formulation of the
latter). Let us summarize a number of fully exact sub-categories which have
these particular properties:

\begin{proposition}
\label{TMT_PropFiltering}Let $\mathcal{C}$ be an exact category.

\begin{enumerate}
\item $\mathcal{C}\hookrightarrow\mathsf{Ind}^{a}(\mathcal{C})$ is left $s$-filtering.

\item $\mathcal{C}\hookrightarrow\mathsf{Pro}^{a}(\mathcal{C})$ is right $s$-filtering.

\item $\mathsf{Pro}^{a}(\mathcal{C})\hookrightarrow\mathsf{Tate}%
^{el}(\mathcal{C})$ is left $s$-filtering.

\item $\mathsf{Ind}^{a}(\mathcal{C})\hookrightarrow\mathsf{Tate}%
^{el}(\mathcal{C})$ is right $s$-filtering if $\mathcal{C}$ is idempotent complete.

\item $\mathsf{Ind}^{a}(\mathcal{C})\cap\mathsf{Pro}^{a}(\mathcal{C}%
)=\mathcal{C}$, viewed as full sub-categories of $\mathsf{Tate}^{el}%
(\mathcal{C})$.
\end{enumerate}
\end{proposition}

\begin{proof}
(1) \cite[Prop. 3.10]{TateObjectsExactCats}, (2) \cite[Thm. 4.2]%
{TateObjectsExactCats}, (3) \cite[Prop. 5.8]{TateObjectsExactCats}, (4)
\cite[Remark 5.35]{TateObjectsExactCats}, \cite[Cor. 2.3]{bgwTensor}, (5)
\cite[Prop. 5.9]{TateObjectsExactCats}.
\end{proof}

The construction of this type of quotient category is compatible with the
formation of derived categories in the following sense:

\begin{proposition}
[Schlichting]\label{Prop_SchlichtingLocalizationMakesExactSeqOfDerivedCats}Let
$\mathcal{C}$ be an idempotent complete exact category and $\mathcal{C}%
\hookrightarrow\mathcal{D}$ a right (or left) $s$-filtering inclusion as a
full sub-category of an exact category $\mathcal{D}$. Then%
\[
D^{b}(\mathcal{C})\hookrightarrow D^{b}(\mathcal{D})\twoheadrightarrow
D^{b}(\mathcal{D}/\mathcal{C})
\]
is an exact sequence of triangulated categories.
\end{proposition}

This is \cite[Prop. 2.6]{MR2079996}. The construction of the derived category
of an exact category is explained in Keller \cite[\S 11]{MR1421815} or
B\"{u}hler \cite[\S 10]{MR2606234}.

\begin{theorem}
[{Keller's Localization Theorem, \cite[\S 1.5, Theorem]{MR1667558}}%
]\label{Thm_KellersLocalizationTheorem}Let $\mathcal{C}$ be an exact category
and $\mathcal{C}\hookrightarrow\mathcal{D}$ a right (or left) $s$-filtering
inclusion as a full sub-category of an exact category $\mathcal{D}$. Then%
\[
HH(\mathcal{C})\longrightarrow HH(\mathcal{D})\longrightarrow HH(\mathcal{D}%
/\mathcal{C})\longrightarrow+1
\]
is a fiber sequence in Hochschild homology.
\end{theorem}

This is due to Keller \cite[\S 1.5, Theorem]{MR1667558}. The following result
was first proven for countable cardinality by Sho Saito \cite{MR3317759}:

\begin{proposition}
For any infinite cardinal $\kappa$ and exact category $\mathcal{C}$, there is
an exact equivalence of exact categories%
\[
\left.  \mathsf{Tate}_{\kappa}^{el}(\mathcal{C})/\mathsf{Pro}_{\kappa}%
^{a}(\mathcal{C})\right.  \overset{\sim}{\longrightarrow}\left.
\mathsf{Ind}_{\kappa}^{a}(\mathcal{C})/\mathcal{C}\right.  \text{.}%
\]

\end{proposition}

See \cite[Prop. 5.32]{TateObjectsExactCats} for a detailed proof. It turns out
that this result admits a symmetric dual statement, which will be more useful
for the purposes of this paper.

\begin{proposition}
[{\cite[Prop. 5.34]{TateObjectsExactCats}}]%
\label{TMT_PropTateModIndEqualsProModC}Let $\mathcal{C}$ be an idempotent
complete exact category. There is an exact equivalence of exact categories%
\[
\left.  \mathsf{Tate}^{el}(\mathcal{C})/\mathsf{Ind}^{a}(\mathcal{C})\right.
\overset{\sim}{\longrightarrow}\left.  \mathsf{Pro}^{a}(\mathcal{C}%
)/\mathcal{C}\right.  \text{,}%
\]
sending an object $X\in\mathsf{Tate}^{el}(\mathcal{C})$ to $L$, where $L$ is
any lattice $L\hookrightarrow X$, and morphisms $f:X\rightarrow X^{\prime}$ to
a suitable restriction $\left.  f\mid_{L}\right.  :L\rightarrow L^{\prime}$
with $L^{\prime}\hookrightarrow X^{\prime}$ a suitable lattice. This defines a
well-defined functor. The inverse equivalence is induced from the inclusion of
categories $\mathsf{Pro}^{a}(\mathcal{C})\hookrightarrow\mathsf{Tate}%
^{el}(\mathcal{C})$.
\end{proposition}

\subsection{\label{sect_RelativeMoritaTheory}Relative Morita theory}

In this section we develop a series of results aiming at the comparison of $n
$-Tate categories with projective module categories. The following lemma is
the starting point for this type of consideration. Recall that we write
$P_{f}(R)$ to denote the category of finitely generated projective right $R$-modules.

\begin{definition}
Let $\mathcal{C}$ be an exact category. We say that $S\in\mathcal{C}$ is
a\emph{\ generator} if every object $X\in\mathcal{C}$ is a direct summand of
$S^{\oplus n}$ for $n$ sufficiently large.
\end{definition}

\begin{lemma}
\label{TMT_LemmaMoritaProjectiveGenerator}Let $\mathcal{C}$ be an idempotent
complete and split exact category with generator $S$. Then the functor%
\begin{equation}
\mathcal{C}\longrightarrow P_{f}(\operatorname*{End}\nolimits_{\mathcal{C}%
}(S))\text{,}\qquad Z\longmapsto\operatorname*{Hom}\nolimits_{\mathcal{C}%
}(S,Z)\nonumber
\end{equation}
is an exact equivalence of categories.
\end{lemma}

\begin{proof}
This is \cite[Theorem 1]{bgwTateModule}.
\end{proof}

While such comparison results have been known for decades, there seems to be
very little literature studying the $2$-functoriality of them. The rest of the
section will work out explicit descriptions of the relevant maps in all the
cases relevant for the paper. A number of these results might be of
independent interest.

\subsubsection{Sub-categories}

\begin{lemma}
\label{TMT_LemmaInclusionAndMoritaGenerator}Let $\mathcal{D}$ be a split exact
category. Suppose $\mathcal{C}\hookrightarrow\mathcal{D}$ is a fully exact
sub-category. Then $\mathcal{C}$ is also split exact. Suppose $S$ is a
generator for $\mathcal{C}$ and $\tilde{S}\in\mathcal{D}$ a generator for
$\mathcal{D}$. Suppose%
\[
\tilde{S}=S\oplus S^{\prime}%
\]
for some $S^{\prime}\in\mathcal{D}$. Then there is a commutative diagram%
\[%
\bfig\Square(0,0)[{\mathcal{C}^{ic}}`{\mathcal{D}^{ic}}`{P_{f}(\operatorname
{End}\nolimits_{\mathcal{C}}( S) )}`{P_{f}(\operatorname{End}\nolimits
_{\mathcal{D}}( \tilde{S}) )};`{\sim}`{\sim}`]
\efig
\]
whose downward arrows are exact equivalences and the bottom rightward arrow is%
\[
M\longmapsto M\otimes_{\operatorname*{End}\nolimits_{\mathcal{D}}%
(S)}\operatorname*{Hom}\nolimits_{\mathcal{D}}(\tilde{S},S)\text{,}%
\]
and equivalently this functor is induced by the (non-unital) algebra
homomorphism
\begin{equation}
\operatorname*{End}\nolimits_{\mathcal{C}}\left(  S\right)  \longrightarrow
\operatorname*{End}\nolimits_{\mathcal{D}}\left(  S\oplus S^{\prime}\right)
\text{,}\qquad f\longmapsto
\begin{pmatrix}
f & 0\\
0 & 0
\end{pmatrix}
\text{.}\label{lmai_15}%
\end{equation}

\end{lemma}

Although it feels like this should be standard, we have not been able to
locate a source in the literature.

\begin{proof}
\textit{(Step 1)}\ Since $\mathcal{C}\hookrightarrow\mathcal{D}$ reflects
exactness, $\mathcal{C}$ is also split exact. Thus, its idempotent completion
$\mathcal{C}^{ic}$ is also split exact by Lemma
\ref{TMT_LemmaIdemCompletionOfSplitExactIsSplitExact}. Moreover, if $S$ is a
generator for $\mathcal{C}$, then it is also a generator for $\mathcal{C}%
^{ic}$ since every object in $\mathcal{C}^{ic}$ is a direct summand of an
object in $\mathcal{C}$, and these are in turn direct summands of $S^{\oplus
n}$ for some $n$. The same argument works for $\mathcal{D}^{ic}$. Then the $2
$-universal property of idempotent completion \cite[Prop. 6.10]{MR2606234}
promotes $\mathcal{C}\hookrightarrow\mathcal{D}$ to the top row in the
following diagram:%
\begin{equation}%
\bfig\Square(0,0)[{\mathcal{C}^{ic}}`{\mathcal{D}^{ic}}`{P_{f}(\operatorname
{End}\nolimits_{\mathcal{C}}( S) )}`{P_{f}(\operatorname{End}\nolimits
_{\mathcal{D}}( \tilde{S}) )};`{\sim}`{\sim}`?]
\efig
\label{lmai_12}%
\end{equation}
As both $\mathcal{C}^{ic}$ and $\mathcal{D}^{ic}$ are split exact and
idempotent complete and possess generators, Lemma
\ref{TMT_LemmaMoritaProjectiveGenerator} induces exact equivalences, given by
the downward arrows. Note that an object $Z\in\mathcal{C}$ is sent to%
\[
Z\mapsto\operatorname*{Hom}\nolimits_{\mathcal{C}}(S,Z)\qquad\text{resp.}%
\qquad Z\mapsto\operatorname*{Hom}\nolimits_{\mathcal{D}}(\tilde{S},Z)\text{,}%
\]
depending on which path we follow in the above diagram. Since $\mathcal{C}$ is
a full sub-category of $\mathcal{D}$, the first functor agrees with
$Z\mapsto\operatorname*{Hom}\nolimits_{\mathcal{D}}(S,Z)$. We claim that
Diagram \ref{lmai_12} can be completed to a commutative square of exact
functors by adding the following arrow as the bottom row:%
\begin{align*}
P_{f}(\operatorname*{End}\nolimits_{\mathcal{C}}(S))  & \longrightarrow
P_{f}(\operatorname*{End}\nolimits_{\mathcal{D}}(\tilde{S}))\\
M  & \longmapsto M\otimes_{\operatorname*{End}\nolimits_{\mathcal{D}}%
(S)}\operatorname*{Hom}\nolimits_{\mathcal{D}}(\tilde{S},S)\text{.}%
\end{align*}
This claim is immediate when plugging in $Z:=S$, but since every object in
$\mathcal{C}$ is a direct summand of $S^{\oplus n}$ and this formula preserves
direct summands, this implies the claim for all objects in $\mathcal{C}$.
Since the categories are split exact, checking exactness of the functor
reduces to checking additivity, which is immediate.\newline\textit{(Step 2)}
By $\tilde{S}=S\oplus S^{\prime}$ we get the non-unital homomorphism of
associative algebras in Equation \ref{lmai_15}. If $M\in P_{f}%
(\operatorname*{End}\nolimits_{\mathcal{C}}\left(  S\right)  )$ and
$f\in\operatorname*{End}\nolimits_{\mathcal{C}}\left(  S\right)  $ this means,
just by matrix multiplication, that the equation%
\[
m\cdot f\otimes%
\begin{pmatrix}
a & b\\
c & d
\end{pmatrix}
=m\otimes%
\begin{pmatrix}
fa & fb\\
0 & 0
\end{pmatrix}
\text{,}\qquad m\in M
\]
holds in $M\otimes_{\operatorname*{End}\nolimits_{\mathcal{C}}\left(
S\right)  }\operatorname*{End}\nolimits_{\mathcal{D}}\left(  S\oplus
S^{\prime}\right)  $. Thus, we find that the map%
\begin{align*}
M\otimes_{\operatorname*{End}\nolimits_{\mathcal{C}}\left(  S\right)
}\operatorname*{End}\nolimits_{\mathcal{D}}\left(  S\oplus S^{\prime}\right)
& \longrightarrow M\otimes_{\operatorname*{End}\nolimits_{\mathcal{D}}%
(S)}\operatorname*{Hom}\nolimits_{\mathcal{D}}(\tilde{S},S)\\
m\otimes%
\begin{pmatrix}
a & b\\
c & d
\end{pmatrix}
& \longmapsto m\otimes%
\begin{pmatrix}
a & b\\
0 & 0
\end{pmatrix}
\end{align*}
is an isomorphism of right $\operatorname*{End}\nolimits_{\mathcal{D}}(S\oplus
S^{\prime})$-modules. As a result, the functor can also be described by
tensoring along the non-unital algebra homomorphism.
\end{proof}

\subsubsection{Quotient categories}

While the previous result covers the case of a fully exact sub-category, we
now want to address the same problem in the situation of a quotient category.
Suppose $\mathcal{C}\hookrightarrow\mathcal{D}$ is a fully exact sub-category.
We need stronger hypotheses to ensure the existence of the quotient category,
namely those of \S \ref{subsect_QuotientExactCategories}. We then write%
\[
\left\langle X\rightarrow C\rightarrow X\text{ with }C\in\mathcal{C}%
\right\rangle
\]
to denote the two-sided ideal of morphisms $X\rightarrow X$, for
$X\in\mathcal{D}$, which factor over an object in $\mathcal{C}$. Note that
this really produces a two-sided ideal, so it makes no difference whether we
think of this as an ideal or as the ideal generated by such morphisms.

\begin{lemma}
\label{TMT_MoritaToQuotientCategory}Let $\mathcal{D}$ be a split exact
category with a generator $S$. Suppose $\mathcal{C}\hookrightarrow\mathcal{D}$
is a left (or right) $s$-filtering sub-category.

\begin{enumerate}
\item Then $\operatorname*{End}\nolimits_{\mathcal{D}/\mathcal{C}}\left(
S\right)  =\operatorname*{End}\nolimits_{\mathcal{D}}\left(  S\right)
/\left\langle S\rightarrow C\rightarrow S\text{ with }C\in\mathcal{C}%
\right\rangle $.

\item Moreover, the diagram%
\[%
\bfig\Square(0,0)[{\mathcal{D}^{ic}}`{({\mathcal{D}/\mathcal{C}})^{ic}}%
`{P_{f}(\operatorname{End}\nolimits_{\mathcal{D}}( S) )}`{P_{f}(\operatorname
{End}\nolimits_{{\mathcal{D}/\mathcal{C}}}( S) )};`{\sim}`{\sim}`]
\efig
\]
commutes, where the lower horizontal arrow is%
\[
M\longmapsto M\otimes_{\operatorname*{End}\nolimits_{\mathcal{D}}\left(
S\right)  }\operatorname*{End}\nolimits_{\mathcal{D}/\mathcal{C}}\left(
S\right)  \text{.}%
\]

\end{enumerate}
\end{lemma}

\begin{proof}
We only prove the left $s$-filtering case: As $\mathcal{C}\hookrightarrow
\mathcal{D}$ is left $s$-filtering, the quotient category $\mathcal{D}%
/\mathcal{C}$ exists. Idempotent completion has a suitable universal property
as a $2$-functor so that the resulting exact functor $\mathcal{D}%
\rightarrow\mathcal{D}/\mathcal{C}$ induces canonically a functor
$\mathcal{D}^{ic}\rightarrow(\mathcal{D}/\mathcal{C})^{ic}$, justifying the
top row \cite[Prop. 6.10]{MR2606234}.\newline\textit{(Step 1)} By
Schlichting's original construction \cite[Lemma 1.13]{MR2079996} the quotient
category $\mathcal{D}/\mathcal{C}$ arises as the localization $\mathcal{D}%
[\Sigma^{-1}]$ with a calculus of left fractions, where the class $\Sigma$ is
formed of (1) admissible monics with cokernel in $\mathcal{C}$, (2) admissible
epics with kernels in $\mathcal{C}$, (3) and closed under composition. For
this proof, we shall use that it suffices to localize at $\Sigma_{e}%
=\{$admissible epics with kernels in $\mathcal{C}\}$, i.e. $\mathcal{D}%
/\mathcal{C}:=\mathcal{D}[\Sigma^{-1}]=\mathcal{D}[\Sigma_{e}^{-1}]$, by
\cite[Prop. 2.19]{TateObjectsExactCats}. This idea is due to T. B\"{u}hler. By
loc. cit. this localization admits a calculus of left fractions\footnote{The
conventions of left and right fractions are as in Gabriel--Zisman
\cite{MR0210125} or B\"{u}hler \cite{MR2606234}. This means that the meaning
of left and right is opposite to the usage in \cite{MR1950475},
\cite{MR2182076}.}. This means that every morphism $X\rightarrow Y$ in
$\mathcal{D}/\mathcal{C}$ is represented by a left roof
\[
X\longrightarrow W\overset{\in\Sigma_{e}}{\longleftarrow}Y\text{.}%
\]
First of all, we shall show that all morphisms are equivalent to morphisms
coming from $\mathcal{D}$, i.e. left roofs of the shape $X\longrightarrow
Y\overset{1}{\longleftarrow}Y$: Suppose we are given an arbitrary left roof%
\[%
\xymatrix{ & W & \\ X \ar[ur]^{f} &  & Y. \ar[ul]_{h}}%
\]
As $h\in\Sigma_{e}$ is a split epic, we may write $Y=W\oplus C$ for some
object $C\in\mathcal{C}$ so that our roof takes the shape%
\[%
\xymatrix{ & W & \\ X \ar[ur]^{f} &  & {W\oplus C}. \ar[ul]_{pr_W}}%
\]
Thanks to the commutative diagram%
\[%
\xymatrix{ & W \ar[d]^{1} & \\ X \ar[ur]^{f} \ar[dr]_{f \oplus0}
& W & {W\oplus C} \ar[ul]_{pr_W} \ar[dl]^{1} \\ & {W\oplus C}, \ar[u]_{pr_W}
& }%
\]
we learn that this left roof is equivalent to the roof $X\longrightarrow
W\oplus C\overset{1}{\longleftarrow}W\oplus C$ and rewriting this using
$Y=W\oplus C$ in terms of $Y$, we have proven our claim. This means that
$\operatorname*{Hom}\nolimits_{\mathcal{D}}(X,Y)\rightarrow\operatorname*{Hom}%
\nolimits_{\mathcal{D}/\mathcal{C}}(X,Y)$ is surjective. It remains to
determine the kernel. By the calculus of left fractions, two left roofs
(which, as we had just proven, we may assume to come from genuine morphisms in
$\mathcal{D}$) are equivalent if and only if there exists a commutative
diagram of the shape%
\[%
\xymatrix{ & Y \ar[d]^{{\Sigma_{e}}} & \\ X \ar[ur]^{f} \ar[dr]_{g}
& H & Y \ar[ul]_{1} \ar[dl]^{1} \\ & Y. \ar[u]_{{\Sigma_{e}}} & }%
\]
The existence of such a diagram is equivalent to the the equality of morphism%
\[
X\underset{g}{\overset{f}{\rightrightarrows}}Y\overset{\in\Sigma_{e}%
}{\longrightarrow}H
\]
and thus to $f-g$ mapping to zero in $H$. By the universal property of
kernels, this is equivalent to the existence of a factorization
$f-g:X\rightarrow\ker(Y\twoheadrightarrow H)\rightarrow Y$. Since
$Y\twoheadrightarrow H$ lies in $\Sigma_{e}$, we have $\ker
(Y\twoheadrightarrow H)\in\mathcal{C}$. The converse direction works the same
way. Thus,%
\[
\operatorname*{End}\nolimits_{\mathcal{D}/\mathcal{C}}\left(  S\right)
=\operatorname*{End}\nolimits_{\mathcal{D}}\left(  S\right)  /\left\langle
S\rightarrow C\rightarrow S\text{ with }C\in\mathcal{C}\right\rangle
\]
\newline and since the embedding $\mathcal{D}/\mathcal{C}\rightarrow
(\mathcal{D}/\mathcal{C})^{ic}$ is fully faithful \cite[Remark 6.3]%
{MR2606234}, this description also applies to $(\mathcal{D}/\mathcal{C})^{ic}%
$.\newline\textit{(Step 2)} Since $\mathcal{D}$ is split exact, $\mathcal{D}%
^{ic}$ is idempotent complete and still split exact by Lemma
\ref{TMT_LemmaIdemCompletionOfSplitExactIsSplitExact}. Hence, Lemma
\ref{TMT_LemmaMoritaProjectiveGenerator} implies that the left-hand side
downward arrow, $Z\mapsto\operatorname*{Hom}\nolimits_{\mathcal{D}^{ic}}%
(S,Z)$, is an equivalence of categories. The quotient category $\mathcal{D}%
/\mathcal{C}$ is an exact category where a kernel-cokernel sequence%
\[
A\longrightarrow B\longrightarrow C
\]
is considered exact iff it is isomorphic to the image of an exact sequence in
$\mathcal{D}$, \cite[Prop. 1.16]{MR2079996}. This is the canonical exact
structure on $\mathcal{D}/\mathcal{C}$, making $\mathcal{D}\rightarrow
\mathcal{D}/\mathcal{C}$ an exact functor. Since $\mathcal{D}$ is split exact,
this means $B$ is isomorphic to the direct sum of the outer terms and this
property stays true in $\mathcal{D}/\mathcal{C}$. Thus, $\mathcal{D}%
/\mathcal{C}$ also has the split exact structure (however there is no reason
why it would have to be idempotent complete). The functor $\mathcal{D}%
/\mathcal{C}\rightarrow(\mathcal{D}/\mathcal{C})^{ic}$ to the idempotent
completion is exact, and the idempotent completion $(\mathcal{D}%
/\mathcal{C})^{ic}$ must also be split exact by Lemma
\ref{TMT_LemmaIdemCompletionOfSplitExactIsSplitExact}. If every object in
$\mathcal{D}$ is a direct summand of $S^{\oplus n}$, this property stays true
in $\mathcal{D}/\mathcal{C}$, and thus in $(\mathcal{D}/\mathcal{C})^{ic}$.
Hence, Lemma \ref{TMT_LemmaMoritaProjectiveGenerator} also applies to
$(\mathcal{D}/\mathcal{C})^{ic}$, with the image of the same object, and thus
there is an exact equivalence of categories via $Z\mapsto\operatorname*{Hom}%
\nolimits_{(\mathcal{D}/\mathcal{C})^{ic}}(S,Z)$. This is a priori a right
$\operatorname*{End}\nolimits_{(\mathcal{D}/\mathcal{C})^{ic}}\left(
S\right)  $-module, but by the full faithfulness of the idempotent completion
this algebra agrees with $\operatorname*{End}\nolimits_{\mathcal{D}%
/\mathcal{C}}\left(  S\right)  $. Finally, we observe that in order to make
the diagram commute the lower horizontal arrow must be%
\[
\operatorname*{Hom}\nolimits_{\mathcal{D}^{ic}}(S,Z)\mapsto\operatorname*{Hom}%
\nolimits_{(\mathcal{D}/\mathcal{C})^{ic}}(S,Z)
\]
for all $Z\in\mathcal{D}$. But by Step 1 this is just quotienting out the
ideal $\left\langle S\rightarrow C\rightarrow S\text{ with }C\in
\mathcal{C}\right\rangle $ from the right, or equivalently tensoring with the
corresponding quotient ring. This proves our claim.
\end{proof}

Although this diverts a bit from our storyline and will not be used anywhere
else in this paper, let us record an immediate application of this lemma:

\begin{proposition}
Drinfeld's Calkin category $\mathcal{C}_{R}^{\operatorname*{Kar}}$, as defined
in \cite[\S 3.3.1]{MR2181808}, is equivalent to the Calkin category
$\mathsf{Calk}(\mathcal{C}):=(\mathsf{Ind}^{a}(\operatorname*{Mod}%
(R))/\operatorname*{Mod}(R))^{ic}$ of \cite[Def. 3.40]{TateObjectsExactCats}.
\end{proposition}

\begin{proof}
The category $\operatorname*{Mod}(R)$ is split exact with generator $R$. We
obtain the claim by applying Lemma \ref{TMT_MoritaToQuotientCategory} to the
left $s $-filtering inclusion $\operatorname*{Mod}(R)\hookrightarrow
\mathsf{Ind}^{a}(\operatorname*{Mod}(R))$, and the description of this
category given by the lemma agrees with the definition used by Drinfeld in
\cite[\S 3.3.1]{MR2181808}.
\end{proof}

\subsection{\label{subsect_ApplyMoritaToTateCategories}Applications to Tate
categories}

\subsubsection{Iterated Morita calculus for $n$-Tate categories}

Now we can apply these results to Ind-, Pro- and Tate categories. For the sake
of legibility we have divided the following arguments into several separate
propositions. However, there will be a great overlap in notation so that it
appears to be reasonable to introduce some overall notation for the length of
this section.

The starting point of these definitions will be the following key
ingredient:\medskip

Assume $\mathcal{C}$ is any split exact category with a generator
$S\in\mathcal{C}$. Then we define objects
\[
S[t^{-1}]:=%
{\textstyle\coprod_{\mathbf{N}}}
S\in\mathsf{Ind}^{a}(\mathcal{C})\qquad\text{and}\qquad S[[t]]:=%
{\textstyle\prod_{\mathbf{N}}}
S\in\mathsf{Pro}^{a}(\mathcal{C})\text{,}%
\]
where the (co)product is interpreted as the corresponding formal Ind- resp.
Pro-limit object. In order to be absolutely precise, let us spell out what
this means explicitly in terms of the actual definition of the respective
exact categories, as in \S \ref{subsect_TateCatsViaIndDiagrams}:

We define an admissible Ind-diagram and admissible Pro-diagram by%
\begin{equation}
S[t^{-1}]:\mathbf{N}\longrightarrow\mathcal{C}\text{,}\quad n\mapsto%
{\textstyle\coprod\nolimits_{i=1}^{n}}
S\text{,}\qquad\qquad S[[t]]:\mathbf{N}\longrightarrow\mathcal{C}\text{,}\quad
n\mapsto%
{\textstyle\prod\nolimits_{i=1}^{n}}
S\text{.}\label{lmai_17}%
\end{equation}
More specifically, we view the natural numbers $\mathbf{N}$ as a directed
poset and define admissible diagrams by these formulae, where for $S[t^{-1}]$
a morphism $n\mapsto n+1$ in $\mathbf{N}$ is sent to the inclusion of $%
{\textstyle\coprod\nolimits_{i=1}^{n}}
S\rightarrow%
{\textstyle\coprod\nolimits_{i=1}^{n+1}}
S$, while for $S[[t]]$ we send it to the projection in the opposite direction.
Finally, define%
\begin{equation}
S((t)):=S[[t]]\oplus S[t^{-1}]\in\mathsf{Tate}^{el}(\mathcal{C})\text{.}%
\label{lmai_41}%
\end{equation}
This is obviously an elementary Tate object since $S[[t]]$ is a lattice (see
Definition \ref{Def_TateObjectAndLattices}). These objects completely
characterize the Tate object categories in the following way:

\begin{theorem}
[\cite{bgwTateModule}]\label{Thm_NTateIsModuleCatWithGoodCubicalEndoAlgebra}%
Let $\mathcal{C}$ be an idempotent complete split exact category with a
generator\footnote{Of course, it we instead have a finite system of
generators, we can just take their direct sum as a single-object generator.}
$S\in\mathcal{C}$.

\begin{enumerate}
\item The category $\left.  n\text{-}\mathsf{Tate}_{\aleph_{0}}(\mathcal{C}%
)\right.  $ is split exact and idempotent complete.

\item The object $\tilde{S}_{n}:=S((t_{1}))\cdots((t_{n}))$ is a generator of
$\left.  n\text{-}\mathsf{Tate}^{el}(\mathcal{C})\right.  $ and $\left.
n\text{-}\mathsf{Tate}(\mathcal{C})\right.  $ and we have an exact equivalence
of exact categories%
\begin{equation}
\left.  n\text{-}\mathsf{Tate}_{\aleph_{0}}(\mathcal{C})\right.  \overset
{\sim}{\longrightarrow}P_{f}(A_{n})\quad\text{with}\quad A_{n}%
:=\operatorname*{End}(\tilde{S}_{n})\text{.}\label{lmai_16}%
\end{equation}

\item For every object $X\in\left.  n\text{-}\mathsf{Tate}_{\aleph_{0}}%
^{el}(\mathcal{C})\right.  $ its endomorphism algebra canonically carries the
structure of a Beilinson $n$-fold cubical algebra. We refer to
\cite{bgwTateModule} for the construction.
\end{enumerate}
\end{theorem}

This theorem hinges crucially on our restriction to Tate objects of
\emph{countable} cardinality. See \cite{TateObjectsExactCats} for a detailed
discussion and counter-examples due to J. \v{S}\v{t}ov\'{\i}\v{c}ek and J.
Trlifaj for strictly greater cardinalities. Also, if $\mathcal{C}$ is not
split exact, there is no way to save the conclusions of this result. We refer
to the introduction of \cite{bgwTateModule} for an overview.

\begin{proof}
\textit{(Claim 1)} By \cite[Theorem 7.2]{TateObjectsExactCats} the category
$\left.  n\text{-}\mathsf{Tate}_{\aleph_{0}}^{el}(\mathcal{C})\right.  $ is
split exact. Thus, its idempotent completion $\left.  n\text{-}\mathsf{Tate}%
_{\aleph_{0}}(\mathcal{C})\right.  $ fulfills the claim by Lemma
\ref{TMT_LemmaIdemCompletionOfSplitExactIsSplitExact}.\newline\textit{(Claim
2)} The statement about the generator for $\left.  n\text{-}\mathsf{Tate}%
^{el}(\mathcal{C})\right.  $ is proven in \cite[Prop. 7.4]%
{TateObjectsExactCats}. Since $\left.  n\text{-}\mathsf{Tate}(\mathcal{C}%
)\right.  $ is just the idempotent completion of this category, each of its
objects is a direct summand of an object in $\left.  n\text{-}\mathsf{Tate}%
^{el}(\mathcal{C})\right.  $ and thus this generator also works for the
idempotent completion. The equivalence of Equation \ref{lmai_16} stems from
Lemma \ref{TMT_LemmaMoritaProjectiveGenerator}. One has to check that the
assumptions of the lemma hold true. See \cite[Theorem 1]{bgwTateModule} for
the details.\newline\textit{(Claim 3)} This is \cite[Theorem 1]{bgwTateModule}%
. We give a brief survey: Call a morphism $f:X\rightarrow Y$ of $n$-Tate
objects \textit{bounded} if it factors through a lattice in the target, say
$X\rightarrow L\hookrightarrow Y$, and \textit{discrete}, if it sends a
lattice of the source to zero. For $X=Y$, one checks that these morphisms form
two-sided ideals, $I_{1}^{\pm}$ and moreover $I_{1}^{+}+I_{1}^{-}=R$. For the
latter, most assumptions are needed, especially $\mathcal{C}$ split exact and
cardinality $\kappa:=\aleph_{0}$. See \cite{bgwTateModule} for
counter-examples when these assumptions are not met. The ideals $I_{2}^{\pm} $
are defined inductively: For any nested pair of lattices $L^{\prime
}\hookrightarrow L\hookrightarrow X$, the quotient $L/L^{\prime}$ is an
$(n-1)$-Tate object, and one defines $I_{2}^{+}$ to be those morphisms such
that for any factorization $\overline{f}:L_{1}/L_{1}^{\prime}\rightarrow
L_{2}^{\prime}/L_{2}$ for $f\mid_{L_{1}}$, over suitable lattices $L_{1}%
,L_{1}^{\prime},L_{2},L_{2}^{\prime}$, the morphism $\overline{f}$ is bounded,
as a morphism of $(n-1)$-Tate objects. Similarly for $I_{2}^{-}$. This pattern
can be extended inductively to define $I_{i}^{\pm}$ for $i=1,\ldots,n$.
\end{proof}

Next, we need to check that the cubical algebra actually meets the
well-behavedness criteria we are intending to use later.

\begin{proposition}
\label{TMT_Prop_ConfirmCubicalAlgebraIsGood}The cubical algebra
$\operatorname*{End}(\tilde{S}_{n})$ is good in the sense of Definition
\ref{def_CubicalAlgebraGood}.
\end{proposition}

\begin{proof}
Let us only treat the case of local left units. We prove this by induction in
$n$, starting from $n=1$. Suppose $\{f:X_{1}\rightarrow X_{2}\}$ is a finite
set of trace-class morphisms. In particular, each such $f$ is a finite
morphism (viewed as a $1$-Tate object of $(n-1)$-Tate objects). Then, for each
$f$, being both bounded and discrete, we can find lattices $L_{1}^{\prime
}\hookrightarrow X_{1}$ and $L_{2}\hookrightarrow X_{2}$ so that this $f$
factors as%
\[
f:X_{1}\twoheadrightarrow X_{1}/L_{1}^{\prime}\longrightarrow L_{2}%
\hookrightarrow X_{2}\text{.}%
\]
These being found, we find \textit{one} $L_{1}^{\prime}$ resp. \textit{one}
$L_{2}$ having this property simultaneously for all $f$ in the set by taking
common sub- resp. over-lattices of the corresponding lattices for the
individual $f$ $-$ this uses the (co-)directedness of the Sato Grassmannian,
Thm. \ref{Thm_SatoGrassmannianProperties}.

Fix any over-lattice $L_{1}$ of $L_{1}^{\prime}$. Then for every sub-lattice
$L_{2}^{\prime}$ of $L_{2}$ we get an induced morphism%
\[
f\mid_{L_{1}}:L_{1}/L_{1}^{\prime}\longrightarrow L_{2}\longrightarrow
L_{2}/L_{2}^{\prime}\text{.}%
\]
By assumption, each such $f\mid_{L_{1}}$ is a trace-class morphism of $(n-1)
$-Tate objects, so we look at a finite set of trace-class morphisms and can
find a local left unit, say $e_{1}$, by induction (if $n=1$ arbitrary
morphisms between objects in $\mathcal{C}$ are trace-class, so we can just use
the identity morphism of $\mathcal{C}$. If $n\geq2$ we argue by induction). It
remains to lift these local left units to a map $X_{1}$ to $X_{2}$%
.\newline\textsc{(Step A)} If we replace $L_{2}^{\prime}$ by a sub-lattice
$L_{2}^{\prime\prime}$, we get a commutative diagram, depicted below on the
left:%
\begin{equation}%
\bfig\square/```^{ (}->/[``{L_{2}^{\prime}/L_{2}^{\prime\prime}}`{L_{2}%
/L_{2}^{\prime\prime}};```]
\btriangle(500,0)/>`>`>>/[{L_{1}/L_{1}^{\prime}}`{L_{2}/L_{2}^{\prime\prime}%
}`{L_{2}/L_{2}^{\prime}};``]
\efig
\qquad\qquad%
\bfig\dtriangle/<-`<-`^{ (}->/[L_{2}`{L_{1}/L_{1}^{\prime}}`{L_{1}^{+}%
/L_{1}^{\prime}};``]
\square(500,0)/```>>/[``{L_{1}^{+}/L_{1}^{\prime}}`{L_{1}^{+}/L_{1}};```]
\efig
\label{l_shr_Diag1}%
\end{equation}
The downward arrow exists since we even have a map to $L_{2}$ without
quotienting out anything. We get%
\[
f-\sigma_{L_{2}/L_{2}^{\prime}}f:L_{1}/L_{1}^{\prime}\longrightarrow
L_{2}^{\prime}/L_{2}^{\prime\prime}\text{,}%
\]
where $\sigma$ is a section of the right-hand side epimorphism. Since $f$ is
trace-class and trace-class morphisms form an ideal, this morphism is also
trace-class. Thus, we again face a trace-class morphism of $(n-1)$-Tate
objects and again by induction, we find a local left unit, say $e_{2}$. Now
the diagonal $(2\times2)$-matrix $(e_{1}\oplus e_{2})$ is a local left unit on
$L_{2}/L_{2}^{\prime\prime}=L_{2}^{\prime}/L_{2}^{\prime\prime}\oplus
L_{2}/L_{2}^{\prime}$.\newline Since we work with a Tate object of countable
cardinality, perform this inductively on an co-exhaustive family of lattices
$L_{2}^{\prime}$, going step-by-step to smaller sub-lattices. This produces a
local left unit to the morphisms $f$, each restricted to $L_{1}$,%
\[
L_{1}/L_{1}^{\prime}\longrightarrow L_{2}\text{.}%
\]
\textsc{(Step B)} Now, we proceed analogously and step-by-step replace $L_{1}
$ by an over-lattice $L_{1}^{+}$. We get the commutative Diagram
\ref{l_shr_Diag1} (depicted on the right) above. This diagram commutes since
our morphism was actually defined on $X/L_{1}^{\prime}$, so the restrictions
to any lattices are necessarily compatible. Again, picking a left section
$\sigma$ in the top row, we get%
\[
f-f\sigma:L_{1}^{+}/L_{1}\longrightarrow L_{2}%
\]
and since $f$ is trace-class, so is this morphism. Now by Step A, we can find
a local left unit $e_{2}$ for these morphisms (as $f$ runs through our finite
set of morphisms) and so the diagonal $(2\times2)$-matrix $e_{1}\oplus e_{2}$
is a local left unit on $L_{1}^{+}/L_{1}^{\prime}=L_{1}/L_{1}^{\prime}\oplus
L_{1}^{+}/L_{1}$. For local right units an analogous argument works. This
finishes the proof.
\end{proof}

Based on the preceding theorem, we make the following section-wide
definitions: Let $\mathcal{C}$ be a split exact and idempotent complete exact
category with a generator $S$.

Let $n\geq0$ be arbitrary. Define%
\begin{equation}
\mathcal{C}_{n}:=\left.  n\text{-}\mathsf{Tate}_{\aleph_{0}}(\mathcal{C}%
)\right.  \quad\text{,}\quad\tilde{S}_{n}:=S((t_{1}))\cdots((t_{n}%
))\quad\text{,}\quad A_{n}:=\operatorname*{End}(\left.  \tilde{S}_{n}\right.
)\label{lmai_27}%
\end{equation}
By Theorem \ref{Thm_NTateIsModuleCatWithGoodCubicalEndoAlgebra} the algebra
$A_{n}$ is a good $n$-fold cubical algebra and there is an exact equivalence
of exact categories%
\begin{equation}
\mathcal{C}_{n}\longrightarrow P_{f}(A_{n})\text{,}\qquad Z\longmapsto
\operatorname*{Hom}(\left.  \tilde{S}_{n}\right.  ,Z)\text{.}\label{lmai_28}%
\end{equation}
In particular, all these exact categories are idempotent complete, split exact
and come equipped with a convenient fixed generator. Since all $A_{n}$ are
cubical algebras, we shall freely write $I_{i}^{+},I_{i}^{-},I_{i}^{0}$ for
the respective ideals of bounded, discrete or finite morphisms. See
\cite{bgwTateModule} for further background.\medskip

Below, we shall unravel step-by-step the nature of certain quotient and
boundary homomorphisms coming from Theorem
\ref{Thm_NTateIsModuleCatWithGoodCubicalEndoAlgebra}.

\begin{proposition}
\label{TMT_PropTateQuotientsToTateModInd}As always in this section, assume
$\mathcal{C}$ is an idempotent complete split exact category with a generator
$S\in\mathcal{C}$. Then the diagram%
\[%
\bfig\Square(0,0)[{\mathsf{Tate}_{\aleph_{0}}(\mathcal{C}_{n})}`{(\mathsf
{Tate}_{\aleph_{0}}^{el}(\mathcal{C}_{n})/\mathsf{Ind}_{\aleph_{0}}%
^{a}(\mathcal{C}_{n}))^{ic}}`{P_{f}(A_{n+1})}`{P_{f}(A_{n+1}/I_{1}^{-}%
)};`{\sim}`{\sim}`]
\efig
\]
commutes, where the top row rightward arrow is induced from the quotient
functor of $\mathsf{Ind}_{\aleph_{0}}^{a}(\mathcal{C}_{n})\hookrightarrow
\mathsf{Tate}_{\aleph_{0}}^{el}(\mathcal{C}_{n})$, and the bottom row
rightward arrow is induced from the quotient morphism of the ideal inclusion
$I_{1}^{-}\hookrightarrow A_{n+1}$. The downward arrows are exact equivalences.
\end{proposition}

\begin{proof}
Firstly, since $\mathcal{C}_{n}$ is split exact, the categories $\mathsf{Pro}%
_{\aleph_{0}}^{a}(\mathcal{C}_{n})$ and $\mathsf{Tate}_{\aleph_{0}}%
^{el}(\mathcal{C}_{n})$ are also split exact categories \cite[Thm. 4.2
(6)]{TateObjectsExactCats}, \cite[Prop. 5.23]{TateObjectsExactCats}. Moreover,
$\mathcal{C}_{n}$ is idempotent complete and thus $\mathsf{Ind}_{\aleph_{0}%
}^{a}(\mathcal{C}_{n})\hookrightarrow\mathsf{Tate}_{\aleph_{0}}^{el}%
(\mathcal{C}_{n})$ is right $s$-filtering by Prop. \ref{TMT_PropFiltering}.
Furthermore, every object in $\mathsf{Tate}_{\aleph_{0}}^{el}(\mathcal{C}%
_{n})$ is a direct summand of $\tilde{S}:=\tilde{S}_{n+1}$. We use Lemma
\ref{TMT_MoritaToQuotientCategory} in order to deduce that the diagram%
\[%
\bfig\Square(0,0)[{\mathsf{Tate}_{\aleph_{0}}(\mathcal{C}_{n})}`{(\mathsf
{Tate}_{\aleph_{0}}^{el}(\mathcal{C}_{n})/\mathsf{Ind}_{\aleph_{0}}%
^{a}(\mathcal{C}_{n}))^{ic}}`{P_{f}(A_{n+1})}`{P_{f}(A_{n+1}/I^{\ast})}%
;`{\sim}`{\sim}`]
\efig
\]
commutes, where we have used that $\mathsf{Tate}_{\aleph_{0}}^{el}%
(\mathcal{C}^{n})^{ic}=\mathsf{Tate}_{\aleph_{0}}(\mathcal{C}_{n})$ in the
upper left corner and $A_{n+1}:=\operatorname*{End}\nolimits_{\left.
(n+1)\text{-}\mathsf{Tate}_{\aleph_{0}}(\mathcal{C})\right.  }(\tilde{S})$,
and where the ideal $I^{\ast}$ is generated by morphisms admitting a
factorization $\tilde{S}\rightarrow I\rightarrow\tilde{S}$ with $I\in
\mathsf{Ind}_{\aleph_{0}}^{a}(\mathcal{C}_{n})$. We claim that $I^{\ast}%
=I_{1}^{-}$, where $I_{1}^{-}$ refers to the structure of $A_{n+1}$ as an
$(n+1)$-fold cubical algebra: Suppose $f\in I^{\ast}$. Then $f$ factors as
$\tilde{S}\rightarrow I\rightarrow\tilde{S}$ with $I\in\mathsf{Ind}%
_{\aleph_{0}}^{a}(\mathcal{C}_{n})$ and by Prop.
\ref{Prop_LatticeLeftFiltAndRightFilt} there exists a lattice
$L\hookrightarrow\tilde{S}$ such that we obtain a further factorization
$\tilde{S}\twoheadrightarrow\tilde{S}/L\rightarrow I\rightarrow\tilde{S}$. In
particular, $f$ sends the lattice $L$ to zero so that $f\in I_{1}^{-}$.
Conversely, suppose $f\in I_{1}^{-}$. Let $L\hookrightarrow\tilde{S}$ be a
lattice which is sent to zero. Then $f$ factors as $\tilde{S}%
\twoheadrightarrow\tilde{S}/L\rightarrow\tilde{S}$ just by the universal
property of quotients. As $L$ is a lattice, $\tilde{S}/L\in\mathsf{Ind}%
_{\aleph_{0}}^{a}(\mathcal{C}_{n})$, proving $f\in I^{\ast}$. This finishes
the proof of $I^{\ast}=I_{1}^{-}$.
\end{proof}

We shall also need the following variation of the same idea.

\begin{proposition}
\label{Prop_ModuleInterpretationProToProModC}As always in this section, assume
$\mathcal{C}$ is an idempotent complete split exact category with a generator
$S\in\mathcal{C}$. Define%
\[
\hat{S}:=S((t_{1}))\cdots((t_{n}))[[t_{n+1}]]\in\mathsf{Pro}_{\aleph_{0}}%
^{a}(\mathcal{C}_{n})\text{,}%
\]
as in Equation \ref{lmai_17}. Then $E:=\operatorname*{End}%
\nolimits_{\mathsf{Pro}_{\aleph_{0}}^{a}(\mathcal{C}_{n})}(\hat{S})$ is an
$(n+1)$-fold cubical algebra and we have a commutative diagram%
\begin{equation}%
\bfig\Square(0,0)[{\left[ \mathsf{Pro}_{\aleph_{0}}^{a}(\mathcal{C}_{n}%
)\right] ^{ic}}`{\left[ \mathsf{Pro}_{\aleph_{0}}^{a}(\mathcal{C}%
_{n})/\mathcal{C}_{n}\right] ^{ic}}`{P_{f}(E)}`{P_{f}(E/I_{1}^{0}(\hat
{S}))\text{,}};`{\sim}`{\sim}`]
\efig
\label{lmai_18}%
\end{equation}
where the top row rightward morphism is the quotient functor induced from
$\mathcal{C}_{n}\hookrightarrow\mathsf{Pro}_{\aleph_{0}}^{a}(\mathcal{C}^{n}%
)$, the bottom row rightward morphism stems from the ideal inclusion
$I_{1}^{0}\hookrightarrow E$, and the downward arrows are exact equivalences
of exact categories.
\end{proposition}

\begin{proof}
The $(n+1)$-fold cubical algebra structure is immediate from Theorem
\ref{Thm_NTateIsModuleCatWithGoodCubicalEndoAlgebra}, employing that
$\mathsf{Pro}_{\aleph_{0}}^{a}(\mathcal{C}_{n})\hookrightarrow\mathcal{C}%
_{n+1}$ is a full sub-category, so it does not matter whether we consider
endomorphisms in $\mathsf{Pro}_{\aleph_{0}}^{a}(\mathcal{C}_{n})$ or the
$(n+1)$-Tate category $\mathcal{C}_{n+1}$. By Prop. \ref{TMT_PropFiltering}
the inclusion $\mathcal{C}_{n}\hookrightarrow\mathsf{Pro}_{\aleph_{0}}%
^{a}(\mathcal{C}^{n})$ is right $s$-filtering. This produces the top row of
the following diagram:%
\begin{equation}%
\begin{array}
[c]{ccccc}%
\mathcal{C}_{n} & \hookrightarrow & \mathsf{Pro}_{\aleph_{0}}^{a}%
(\mathcal{C}^{n}) & \twoheadrightarrow & \mathsf{Pro}_{\aleph_{0}}%
^{a}(\mathcal{C}^{n})/\mathcal{C}_{n}\\
&  & \downarrow &  & \downarrow\\
&  & \left[  \mathsf{Pro}_{\aleph_{0}}^{a}(\mathcal{C}_{n})\right]  ^{ic} &
\twoheadrightarrow & \left[  \mathsf{Pro}_{\aleph_{0}}^{a}(\mathcal{C}%
_{n})/\mathcal{C}_{n}\right]  ^{ic}\\
&  & \downarrow &  & \downarrow\\
&  & P_{f}(E) & \twoheadrightarrow & P_{f}(E/I_{1}^{0})\text{.}%
\end{array}
\label{lmai_4}%
\end{equation}
We construct the second row from the first by taking the fully faithful
embedding into the idempotent completion; the right-ward functor exists by the
$2$-universal property \cite[Prop. 6.10]{MR2606234}. Next, construct the third
row by Lemma \ref{TMT_MoritaToQuotientCategory}. To this end, we employ the
shorthands%
\[
\hat{S}:=S((t_{1}))\cdots((t_{n}))[[t_{n+1}]]\qquad\text{and}\qquad
E:=\operatorname*{End}\nolimits_{\mathsf{Pro}_{\aleph_{0}}^{a}(\mathcal{C}%
_{n})}\hat{S}%
\]
so that this Lemma literally yields the third row%
\begin{equation}
P_{f}(\operatorname*{End}\nolimits_{\mathsf{Pro}_{\aleph_{0}}^{a}%
(\mathcal{C}^{n})}\hat{S})\twoheadrightarrow P_{f}(\operatorname*{End}%
\nolimits_{\mathsf{Pro}_{\aleph_{0}}^{a}(\mathcal{C}_{n})/\mathcal{C}_{n}}%
\hat{S})\label{lmai_1}%
\end{equation}
along with the description%
\begin{equation}
\operatorname*{End}\nolimits_{\mathsf{Pro}_{\aleph_{0}}^{a}(\mathcal{C}%
_{n})/\mathcal{C}_{n}}(\hat{S})=(\operatorname*{End}\nolimits_{\mathsf{Pro}%
_{\aleph_{0}}^{a}(\mathcal{C}_{n})}\hat{S})/\left\langle \hat{S}\rightarrow
C\rightarrow\hat{S}\text{ with }C\in\mathcal{C}_{n}\right\rangle
\text{.}\label{lmai_2}%
\end{equation}
However, $\mathsf{Pro}_{\aleph_{0}}^{a}(\mathcal{C}_{n})$ is a full
sub-category of $\mathcal{C}_{n+1}$, so in Equation \ref{lmai_1} we could just
as well compute the left-hand side endomorphism algebra in $\mathcal{C}^{n+1}%
$. By Theorem \ref{Thm_NTateIsModuleCatWithGoodCubicalEndoAlgebra} the latter
is canonically an $(n+1)$-fold cubical algebra, so this structure is also
available for the endomorphism algebra on the left-hand side in Equation
\ref{lmai_1}, and in particular we can speak of the two-sided ideal $I_{1}%
^{0}$. Next, we claim that%
\begin{equation}
I_{1}^{0}=\left\langle \hat{S}\rightarrow C\rightarrow\hat{S}\text{ with }%
C\in\mathcal{C}_{n}\right\rangle \label{lmai_3}%
\end{equation}
as two-sided ideals in Equation \ref{lmai_2}. Suppose $f\in I_{1}^{0}$. Then
$f:\hat{S}\rightarrow\hat{S}$ is discrete as a morphism of $1$-Tate objects
(with values in $n$-Tate objects). That is, there is a lattice
$L\hookrightarrow\hat{S}$ that is sent to zero. Thus, $f$ factors as $\hat
{S}\twoheadrightarrow\hat{S}/L\rightarrow\hat{S}$, where $\hat{S}/L$ is and
Ind-object (since $L$ is a lattice), and simultaneously a Pro-object since it
is an admissible quotient of the Pro-object $\hat{S}$. Thus, $\hat{S}%
/L\in\mathcal{C}_{n}$ by Theorem \ref{Thm_SatoGrassmannianProperties}, and
thus $f$ lies in the right-hand side ideal in Equation \ref{lmai_3}.
Conversely, suppose $f$ lies in the right-hand side ideal in Equation
\ref{lmai_3}. Consider the map $\hat{S}\rightarrow C$. Since $\mathcal{C}%
\hookrightarrow\mathsf{Pro}_{\aleph_{0}}^{a}(\mathcal{C})$ is right filtering
by Prop. \ref{TMT_PropFiltering}, this arrow admits a factorization $\hat
{S}\twoheadrightarrow\tilde{C}\rightarrow C$ with $\tilde{C}\in\mathcal{C}%
_{n}$. Here $\ker(\hat{S}\twoheadrightarrow\tilde{C})$ exists, it is a lattice
(since it is a sub-object of a Pro-object and thus itself a Pro-object, and
the quotient by it lies in $\mathcal{C}_{n}$, which can trivially be viewed as
an Ind-object), and so $\hat{S}\rightarrow C $ sends a lattice to zero and
therefore so does $f:\hat{S}\rightarrow C\rightarrow\hat{S}$. Hence, $f\in
I_{1}^{-}$. As $\hat{S}$ is a Pro-object, we trivially have $f\in I_{1}^{+}$
and thus $f\in I_{1}^{+}\cap I_{1}^{-}=I_{1}^{0}$. This finishes the proof of
Equation \ref{lmai_3}. Thus, Equation \ref{lmai_1} becomes
\[
P_{f}(E)=P_{f}(\operatorname*{End}\nolimits_{\mathsf{Pro}_{\aleph_{0}}%
^{a}(\mathcal{C}_{n})}\hat{S})\rightarrow P_{f}(\operatorname*{End}%
\nolimits_{\mathsf{Pro}_{\aleph_{0}}^{a}(\mathcal{C}_{n})/\mathcal{C}_{n}}%
\hat{S})=P_{f}(E/I_{1}^{0})\text{.}%
\]
This settles the last row in Diagram \ref{lmai_4}. Note that the explicit
description of the middle arrow in Lemma \ref{TMT_MoritaToQuotientCategory}
under these identifications also confirms that $P_{f}(E)\rightarrow
P_{f}(E/I_{1}^{0})$ just comes from the map $E\twoheadrightarrow E/I_{1}^{0}$.
\end{proof}

\subsubsection{Relation to the abstract Hochschild symbol}

Next, we shall replace the associative algebra $E$ in the previous proposition
by a certain ideal: With the notation of the proposition, write%
\begin{equation}
\tilde{S}_{n+1}=S((t_{1}))\cdots((t_{n+1}))=\hat{S}\oplus S((t_{1}%
))\cdots((t_{n}))[t_{n+1}^{-1}]\text{.}\label{lmai_before6}%
\end{equation}
Now there is an (non-unital) embedding of algebras%
\begin{equation}
E\hookrightarrow\operatorname*{End}S((t_{1}))\cdots((t_{n+1}))=A_{n+1}%
\qquad\text{by}\qquad f\mapsto%
\begin{pmatrix}
f & 0\\
0 & 0
\end{pmatrix}
\text{,}\label{lmai_IdealEmbeddingMap}%
\end{equation}
acting only on $\hat{S}$. Clearly, with this interpretation, $f$ is sent into
the ideal $I_{1}^{+}(\tilde{S}_{n+1})$ since each of these morphisms factors
through the lattice $\hat{S}\hookrightarrow\tilde{S}_{n+1}$ of $\tilde
{S}_{n+1}$ as an $(n+1)$-Tate object. Analogously, if $f$ lies in $I_{1}%
^{0}(\hat{S})$, this embedding maps it to $I_{1}^{0}$ of $\tilde{S}_{n+1}$.

Diagram \ref{lmai_18} induces a commutative square in Hochschild homology,
depicted below as the upper square.%
\begin{equation}%
\begin{array}
[c]{ccc}%
HH\left[  \mathsf{Pro}_{\aleph_{0}}^{a}(\mathcal{C}_{n})\right]  ^{ic} &
\longrightarrow & HH\left[  \mathsf{Pro}_{\aleph_{0}}^{a}(\mathcal{C}%
_{n})/\mathcal{C}_{n}\right]  ^{ic}\\
\downarrow &  & \downarrow\\
HH(E) & \longrightarrow & HH(E/I_{1}^{0}(\hat{S}))\text{,}\\
\downarrow &  & \downarrow\\
HH(I_{1}^{+}) & \longrightarrow & HH(I_{1}^{+}/I_{1}^{0})
\end{array}
\label{lmai_19}%
\end{equation}
The lower square arises from the embedding morphism which we have just
discussed, Equation \ref{lmai_IdealEmbeddingMap}. Note that $I_{1}^{+}$ and
$I_{1}^{0}$ refer to the ideals of $A_{n+1}$!

\begin{lemma}
[\cite{olliTateRes}]\label{lemma_CrucialLAMBDADiagram}The following diagram of
$A_{n+1}$-bimodules%
\[%
\xymatrix{
I^{0}_{1} \ar[r]^-{\text{diag}} \ar[d]_{=} & I^{+}_{1} \oplus I^{-}_{1}
\ar[r]^-{\text{diff}} \ar[d]_{pr_{I^{+}_{1}}} & A_{n+1} \ar[d]^{\text{(1)}}
\ar@/^4pc/[dd]^{\Lambda} \\
A_{n} \ar[r] \ar[d]_{=} & I^{+}_{1} \ar[r] \ar[d]_{\text{incl}} & A_{n+1}%
/{I^{-}_{1}} \ar[d]^{\text{(2)}}\\
A_{n} \ar[r] & A_{n+1} \ar[r] & A_{n+1}/{A_{n}}
}%
\]
is commutative with exact rows. The second and third row are also exact
sequences of associative algebras. See Equation \ref{lEquationToeplitzLAMBDA}
for the definition of the morphism $\Lambda$.
\end{lemma}

\begin{proof}
We will construct this diagram row by row. The exactness of the first row
stems from $I_{1}^{0}:=I_{1}^{+}\cap I_{1}^{-}$ and $I_{1}^{+}+I_{1}%
^{-}=A_{n+1}$, i.e. it comes directly from the axioms of a cubical algebra.
The second row is obtained from quotienting out $I_{1}^{-}$. The downward
arrows, in particular (1), are just the quotient maps. Note that the exactness
of the second row implies that the quotient on the right-hand side can be
rewritten as $A_{n}\hookrightarrow I_{1}^{+}\twoheadrightarrow I_{1}^{+}%
/A_{n}$. The inclusion $I_{1}^{+}\hookrightarrow A_{n+1}$ thus induces the
last exact row. Note that the arrow (2) is induced from $I_{1}^{+}%
\hookrightarrow A_{n+1}$. As a result, the composition of (1) and (2),
$A_{n+1}\rightarrow A_{n+1}/A_{n}$, is \textit{not} the quotient map, but
precisely the map $\Lambda$ (by diagram chase).
\end{proof}

Now we apply Hochschild homology to Diagram \ref{lmai_18} of Prop.
\ref{Prop_ModuleInterpretationProToProModC}. We get a commutative diagram%
\begin{equation}%
\begin{array}
[c]{ccccc}%
HH(\left.  \mathsf{Pro}_{\aleph_{0}}^{a}(\mathcal{C}_{n})\right.  ) &
\rightarrow & HH(\left.  \mathsf{Pro}_{\aleph_{0}}^{a}(\mathcal{C}%
_{n})/\mathcal{C}_{n}\right.  ) & \rightarrow & \Sigma HH(\left.
\mathcal{C}^{n}\right.  )\\
\downarrow &  & \downarrow &  & \downarrow\\
HH(E) & \rightarrow & HH(E/I_{1}^{0}(\hat{S})) & \rightarrow & \Sigma
HH(I_{1}^{0}(\hat{S}))
\end{array}
\label{lmai_20}%
\end{equation}
This is induced from a square: In the top row we could identify the homotopy
fiber by Keller's localization Sequence, which is applicable by the exactness
of the induced sequence of derived categories (Prop.
\ref{Prop_SchlichtingLocalizationMakesExactSeqOfDerivedCats}). In the bottom
row we can identify the homotopy fiber by the long exact sequence in the
Hochschild homology of algebras, Theorem
\ref{Thm_WodzickiLongExactSeqUsingExcision}, of the algebra extension
$I_{1}^{0}(\hat{S})\hookrightarrow E\twoheadrightarrow E/I_{1}^{0}(\hat{S})$,
and the (one-sided) unitality of $I_{1}^{0}$ (this holds since our cubical
algebras are good). In view of Diagram \ref{lmai_19} we can replace the lower
row of Diagram \ref{lmai_20} by the following middle row:%
\begin{equation}%
\begin{array}
[c]{ccccc}%
HH(\left.  \mathsf{Pro}_{\aleph_{0}}^{a}(\mathcal{C}_{n})\right.  ) &
\rightarrow & HH(\left.  \mathsf{Pro}_{\aleph_{0}}^{a}(\mathcal{C}%
_{n})/\mathcal{C}_{n}\right.  ) & \rightarrow & \Sigma HH(\left.
\mathcal{C}^{n}\right.  )\\
\downarrow &  & \downarrow &  & \downarrow\\
HH(I_{1}^{+}) & \rightarrow & HH(I_{1}^{+}/I_{1}^{0}) & \rightarrow & \Sigma
HH(I_{1}^{0})\\
\downarrow &  & \downarrow &  & \downarrow\\
HH(A_{n+1}) & \rightarrow & HH(A_{n+1}/A_{n}) & \rightarrow & \Sigma
HH(A_{n})\text{,}%
\end{array}
\label{lmai_20b}%
\end{equation}
where $I_{1}^{+},I_{1}^{0}$ refer to the ideals of $A_{n+1}$ as usual.
Correspondingly, the bottom rows are induced from the inclusion $I_{1}%
^{+}\hookrightarrow A_{n+1}$ resp. $A_{n+1}\hookrightarrow A_{n}$, as
described in Lemma \ref{lemma_CrucialLAMBDADiagram}.

\begin{proposition}
\label{Prop_LAMBDAFromLocalization}As always in this section, assume
$\mathcal{C}$ is an idempotent complete split exact category with a generator
$S\in\mathcal{C}$. Then the following diagram commutes:%
\begin{equation}%
\xymatrix{
{HH(\mathsf{Tate}_{\aleph_{0}}(\mathcal{C}_{n}))} \ar[r]^-{quot} \ar[d]^{\sim}
&
{HH((\mathsf{Tate}_{\aleph_{0}}^{el}(\mathcal{C}_{n})/\mathsf{Ind}_{\aleph
_{0}}^{a}(\mathcal{C}_{n}))^{ic})}
\ar[r]^-{\sim} &
{HH((\mathsf{Pro}_{\aleph_{0}}^{a}(\mathcal{C}_{n})/\mathcal{C}_{n})^{ic})}
\ar[d] \\
HH(A_{n+1}) \ar[rr]_{\Lambda} & &
HH(A_{n+1}/A_{n})
}%
\label{lmai_25}%
\end{equation}

\end{proposition}

\begin{proof}
Prop. \ref{TMT_PropTateQuotientsToTateModInd} provides a commutative diagram
of exact categories and exact functors and applying Hochschild homology gives
us the commutative diagram%
\[%
\bfig\Square(0,0)[{ HH\mathsf{Tate}_{\aleph_{0}}(\mathcal{C}_{n}%
) }`{ HH((\mathsf{Tate}_{\aleph_{0}}^{el}(\mathcal{C}_{n})/\mathsf
{Ind}_{\aleph_{0}}^{a}(\mathcal{C}_{n}))^{ic}) }`{ HH(A_{n+1}) }%
`{ HH(A_{n+1}/I_{1}^{-}) };`{\sim}`{\sim}`]
\efig
\]
whose downward arrows are isomorphisms. Secondly, Prop.
\ref{TMT_PropTateModIndEqualsProModC} tells us that the inclusion%
\begin{equation}
\mathsf{Pro}_{\aleph_{0}}^{a}(\mathcal{C}_{n})\hookrightarrow\mathsf{Tate}%
_{\aleph_{0}}^{el}(\mathcal{C}_{n})\label{lmai_21}%
\end{equation}
induces the exact equivalence of exact categories%
\begin{equation}
\mathsf{Pro}_{\aleph_{0}}^{a}(\mathcal{C}_{n})/\mathcal{C}_{n}\overset{\sim
}{\longrightarrow}\mathsf{Tate}_{\aleph_{0}}^{el}(\mathcal{C}_{n}%
)/\mathsf{Ind}_{\aleph_{0}}^{a}(\mathcal{C}_{n})\text{.}\label{lmai_22}%
\end{equation}
If we apply Lemma \ref{TMT_LemmaInclusionAndMoritaGenerator} to the fully
exact sub-category of Equation \ref{lmai_21} we obtain the commutative square%
\[%
\bfig\Square(0,0)[{ \mathsf{Pro}_{\aleph_{0}}^{a}(\mathcal{C}_{n})^{ic}
}`{ \mathsf{Tate}_{\aleph_{0}}^{el}(\mathcal{C}_{n})^{ic} }`{ P_{f}%
(\operatorname*{End}\nolimits_{\mathcal{C}}(\tilde{S}_{n}[[t_{n+1}%
]])) }`{ P_{f}(\operatorname*{End}\nolimits_{\mathcal{D}}(\tilde{S}_{n+1}%
)) };```]
\efig
\]
($\tilde{S}_{n}[[t_{n+1}]]$ was previously also called $\hat{S}$; cf. Prop.
\ref{Prop_ModuleInterpretationProToProModC}) where the bottom rightward arrow
stems from the algebra homomorphism%
\begin{equation}
\operatorname*{End}(\tilde{S}_{n}[[t_{n+1}]])\longrightarrow
\operatorname*{End}(\tilde{S}_{n}[[t_{n+1}]]\oplus\tilde{S}_{n}[t_{n+1}%
^{-1}])\text{,}\qquad f\longmapsto%
\begin{pmatrix}
f & 0\\
0 & 0
\end{pmatrix}
\text{.}\label{lmai_23}%
\end{equation}
By Prop. \ref{TMT_PropTateQuotientsToTateModInd} the equivalence
$\mathsf{Tate}(\mathcal{C}_{n})\overset{\sim}{\longrightarrow}P_{f}(A_{n+1})$
identifies the quotient functor $\mathsf{Tate}_{\aleph_{0}}^{el}%
(\mathcal{C}_{n})\rightarrow\mathsf{Tate}_{\aleph_{0}}^{el}(\mathcal{C}%
_{n})/\mathsf{Ind}_{\aleph_{0}}^{a}(\mathcal{C}_{n})$ just with quotienting
out the ideal $I_{1}^{-}$. In view of Equation \ref{lmai_23} the inverse of
the equivalence in Equation \ref{lmai_22} corresponds to finding an
endomorphism of $\tilde{S}_{n}[[t_{n+1}]]$ which is mapped under the map of
Equation \ref{lmai_23} to a given element (alternatively this follows from the
description of the inverse functor on morphisms as given by Prop.
\ref{TMT_PropTateModIndEqualsProModC}). But this is easy to achieve
concretely: Given $f\in\operatorname*{End}\nolimits_{\mathcal{D}}(\tilde
{S}_{n+1})$ we compose it with a projector to $\tilde{S}_{n}[[t_{n+1}]]$, i.e.
the composition of both these steps is realized by%
\begin{align}
& A_{n+1}\longrightarrow A_{n+1}/I_{1}^{-}\longrightarrow\operatorname*{End}%
(\tilde{S}_{n}[[t_{n+1}]])\overset{(\ast)}{\longrightarrow}E/I_{1}^{0}(\hat
{S})\label{lmai_24}\\
& \text{(with }E/I_{1}^{0}(\hat{S})\text{ and the arrow }(\ast)\text{ as in
Diag. \ref{lmai_18})}\nonumber\\
& \qquad\qquad f\longmapsto P^{+}f\longmapsto\overline{P^{+}f}\text{,}%
\nonumber
\end{align}
where $P^{+}$ is the idempotent projecting $\tilde{S}_{n+1}$ to $\tilde{S}%
_{n}[[t_{n+1}]]$. Finally, note that we already know what the map%
\[
HH((\mathsf{Pro}_{\aleph_{0}}^{a}(\mathcal{C}_{n})/\mathcal{C}^{n}%
)^{ic})\longrightarrow HH(A_{n+1}/A_{n})
\]
does: It arises as the composition of a series of arrows in Diagrams
\ref{lmai_20}, \ref{lmai_20b}, namely%
\[
HH(\left.  \mathsf{Pro}_{\aleph_{0}}^{a}(\mathcal{C}_{n})/\mathcal{C}%
^{n}\right.  )\rightarrow HH(E/I_{1}^{0}(\hat{S}))\rightarrow HH(I_{1}%
^{+}/I_{1}^{0})\rightarrow HH(A_{n+1}/A_{n})\text{,}%
\]
where the last two arrows are just induced from non-unital inclusions of
associative algebras into each other. As our lift of Equation \ref{lmai_24}
already gives us a concrete representative for $HH(E/I_{1}^{0}(\hat{S}))$, we
see that%
\begin{equation}
f\longmapsto P^{+}f\label{lmai_26}%
\end{equation}
is a representative of the algebra homomorphism making Diagram \ref{lmai_25}
commutative (This is by the way indeed a ring homomorphism: The failure to
respect multiplication of $f,g\in A_{n+1}$ is%
\[
\mathfrak{d}:=P^{+}(f\cdot g)-(P^{+}f)\cdot(P^{+}g)=P^{+}f(1-P^{+})g\text{.}%
\]
Since the image of $P^{+}$ lies in the lattice $\tilde{S}_{n}[[t_{n+1}]]$ of
$\tilde{S}_{n+1}$, we have $\mathfrak{d}\in I_{1}^{+}$, and since the kernel
of $1-P^{+}$ contains the lattice, we also have $\mathfrak{d}\in I_{1}^{-}$.
Thus, $\mathfrak{d}\in I_{1}^{+}\cap I_{1}^{-}=I_{1}^{0}$ and therefore
$\mathfrak{d}\equiv0$ in $A_{n+1}/A_{n}=A_{n+1}/I_{1}^{0}$.) Finally, observe
that the map in Equation \ref{lmai_26} is a concrete representative of the map
$\Lambda$ (see Equation \ref{lEquationToeplitzLAMBDA}). This finishes the proof.
\end{proof}

\begin{theorem}
\label{Thm_DeloopingBoundaryAgreesWithDMap}As always in this section, assume
$\mathcal{C}$ is an idempotent complete split exact category with a generator
$S\in\mathcal{C}$. Then the natural diagram%
\[%
\bfig\Square(0,0)[{ HH(\mathsf{Tate}_{\aleph_{0}}(\mathcal{C}_{n})) }`{ \Sigma
HH(\mathcal{C}_{n}) }`{ HH(A_{n+1}) }`{ \Sigma HH(A_{n}) };`{\sim}`{\sim}`d]
\efig
\]
commutes. Here the downward arrows are the exact equivalences of Equation
\ref{lmai_28}, and $d$ is the homomorphism of degree $-1$ defined in Equation
\ref{lml_40} (or in the paper \cite[\S 6]{olliTateRes}).
\end{theorem}

\begin{proof}
The top row stems from the square of exact categories in line \ref{lmai_35}.
By the crucial idea of Sho Saito's paper \cite{MR3317759} (his proof of the
Kapranov-Previdi delooping conjecture), after taking algebraic $K$-theory,
this diagram becomes homotopy Cartesian. However, the same idea works with
Hochschild homology, and we get the homotopy commutative diagram%
\begin{equation}%
\xymatrix{ {HH(\mathcal{C}_{n})} \ar[rr] \ar[d] && {HH(\mathsf{Pro}%
_{\aleph_{0}}^{a}(\mathcal{C}_{n}))} \ar[rr] \ar[d] && {HH(\mathsf
{Pro}_{\aleph_{0}}^{a}(\mathcal{C}_{n})/\mathcal{C}_{n})} \ar[d] \\
{HH(\mathsf{Ind}_{\aleph_{0}}^{a}(\mathcal{C}_{n}))} \ar[rr] && {HH(\mathsf
{Tate}_{\aleph_{0}}^{el}(\mathcal{C}_{n}))} \ar[rr] && {HH(\mathsf
{Tate}_{\aleph_{0}}^{el}(\mathcal{C}_{n})/\mathsf{Ind}_{\aleph_{0}}%
^{a}(\mathcal{C}_{n}))}  }%
\label{lmai_30}%
\end{equation}
whose rows are fiber sequences, by Keller's localization Theorem, see Thm.
\ref{Thm_KellersLocalizationTheorem}. In more detail: The rows stem from the
fact that $\mathcal{C}_{n}\hookrightarrow\mathsf{Pro}_{\aleph_{0}}%
^{a}(\mathcal{C}_{n})$ is right $s$-filtering resp. $\mathsf{Pro}_{\aleph_{0}%
}^{a}(\mathcal{C}_{n})\hookrightarrow\mathsf{Tate}_{\aleph_{0}}^{el}%
(\mathcal{C}_{n})$ is left $s$-filtering. All of these constructions are
functorial on the level of exact functors of exact categories and this induces
the downward arrows. Following Saito's idea, since the right-hand side
downward map stems from an exact equivalence, Prop.
\ref{TMT_PropTateModIndEqualsProModC}, it is an equivalence, and thus the
square on the left-hand side is homotopy bi-Cartesian. We obtain the
equivalence $HH(\mathcal{C}_{n})\overset{\sim}{\rightarrow}\Sigma
HH(\mathsf{Tate}_{\aleph_{0}}^{el}(\mathcal{C}_{n}))$ and equivalences which
allow us to phrase this equivalence as the boundary map of the fiber sequence
induced from the top row in Diagram \ref{lmai_30}. Aside: Note that the
underlying equivalence%
\[
HH(\mathsf{Tate}_{\aleph_{0}}^{el}(\mathcal{C}_{n}))\sim HH(\mathsf{Pro}%
_{\aleph_{0}}^{a}(\mathcal{C}_{n})/\mathcal{C}_{n})\text{.}%
\]
of Hochschild spectra \textit{does not} come from an exact equivalence of
exact categories. By Prop. \ref{Prop_LAMBDAFromLocalization} we know that
under the identification of either side with the Hochschild homology of a
category of projective modules and following the middle downward arrow of
Diagram \ref{lmai_20b}, this is induced from the algebra homomorphism
$\Lambda:HH(A_{n+1})\rightarrow HH(A_{n+1}/A_{n})$. But from Diagram
\ref{lmai_20b} we also see that the boundary map of the localization sequence
for $\mathcal{C}_{n}\hookrightarrow\mathsf{Pro}^{a}(\mathcal{C}_{n})$ (in the
top row) commutes with the boundary map of the long exact sequence in
Hochschild homology of the algebra extension%
\[
A_{n}\hookrightarrow A_{n+1}\twoheadrightarrow A_{n+1}/A_{n}%
\]
(in the bottom row). We had denoted the latter boundary map by $\delta$ in
Equation \ref{lml_40}. Thus, in conjunction with the identification with the
boundary map of the Tate category variant (Equation \ref{lmai_30}), the map is
$\delta\circ\Lambda$, which is precisely the definition of the map $d$ in the
statement of the theorem. This finishes the proof.
\end{proof}

\section{\label{sect_TheBeilRes}The Beilinson residue}

\subsection{Ad\`{e}les of a scheme}

\subsubsection{Definition}

Let us recall as much material about ad\`{e}les of schemes as we need. The
original source is Beilinson's article \cite{MR565095}. Let $k$ be a field.
Suppose $X$ is a Noetherian $k$-scheme. Given a scheme point $\eta\in X$, we
shall write $\overline{\{\eta\}}$ for its Zariski closure, equipped with the
reduced closed sub-scheme structure. Moreover, we also abuse notation and
write $\eta$ for its defining ideal sheaf.

When given points $\eta_{0},\eta_{1}\in X$, we write \textquotedblleft%
$\eta_{0}\geq\eta_{1}$\textquotedblright\ if $\overline{\{\eta_{0}\}}\ni
\eta_{1}$. Write $S\left(  X\right)  _{n}:=\{(\eta_{0}>\cdots>\eta_{n}%
),\eta_{i}\in X\}$ for length $n+1$ sequences without repetitions. Suppose
$K_{n}\subseteq S\left(  X\right)  _{n}$ is a subset, for some chosen $n\geq
0$. Following \cite{MR565095}, define $\left.  _{\eta}(K_{n})\right.
:=\{(\eta_{1}>\cdots>\eta_{n})$ such that $(\eta>\eta_{1}>\cdots>\eta_{n})\in
K_{n}\}$, a subset of $S\left(  X\right)  _{n-1}$.

\begin{definition}
\label{def_AdelesFollowingBeilinson}Let $X$ be a Noetherian $k$-scheme.

\begin{enumerate}
\item Assume $\mathcal{F}$ is a coherent sheaf. Define inductively%
\begin{equation}
A(K_{0},\mathcal{F}):=%
{\textstyle\prod\nolimits_{\eta\in K_{0}}}
\underset{i}{\underleftarrow{\lim}}\,\mathcal{F}\otimes_{\mathcal{O}_{X}%
}\mathcal{O}_{X,\eta}/\eta^{i}\nonumber
\end{equation}
for $n=0$, and%
\begin{equation}
A(K_{n},\mathcal{F}):=%
{\textstyle\prod\nolimits_{\eta\in X}}
\underset{i}{\underleftarrow{\lim}}\,A(\left.  _{\eta}K_{n}\right.
,\mathcal{F}\otimes_{\mathcal{O}_{X}}\mathcal{O}_{X,\eta}/\eta^{i}%
)\label{lss2}%
\end{equation}
for $n\geq1$.

\item For a quasi-coherent sheaf $\mathcal{F}$, define $A(K_{n},\mathcal{F}%
):=\underrightarrow{\operatorname*{colim}}_{\mathcal{F}_{j}}A(K_{n}%
,\mathcal{F}_{j})$, where $\mathcal{F}_{j}$ runs through the coherent
sub-sheaves of $\mathcal{F}$.
\end{enumerate}
\end{definition}

Note that the arguments in Equation \ref{lss2} are usually only
quasi-coherent, so this additional definition is necessary to give
$A(-,\mathcal{F})$ a meaning, even if $\mathcal{F}$ happens to be coherent.

\begin{theorem}
[A. Beilinson]\label{lX_BeilinsonResolutionThm}Let $X$ be a Noetherian scheme.

\begin{enumerate}
\item For any $n\geq0$ and subset $K_{n}\subseteq S\left(  X\right)  _{n}$,
the above defines an exact functor%
\[
A(K_{n},-):\operatorname*{QCoh}(X)\longrightarrow\operatorname*{Mod}%
(\mathcal{O}_{X})\text{,}%
\]

\item and for every quasi-coherent sheaf $\mathcal{F}$, this gives rise to a
flasque resolution%
\begin{equation}
0\rightarrow\mathcal{F}\rightarrow\mathbf{A}_{\mathcal{F}}^{0}\rightarrow
\mathbf{A}_{\mathcal{F}}^{1}\rightarrow\mathbf{A}_{\mathcal{F}}^{2}%
\rightarrow\ldots\text{,}\label{lBeilRes}%
\end{equation}
where $\mathbf{A}_{\mathcal{F}}^{i}(U):=A(S\left(  U\right)  _{i}%
,\mathcal{F})$ for any Zariski open $U\subseteq X$.
\end{enumerate}
\end{theorem}

We will not go into further detail. This result is taken from Beilinson's
paper \cite[\S 2]{MR565095}. See Huber \cite{MR1138291} for the proof.

\subsubsection{Local structure for a single flag}

We fix a flag $\triangle=(\eta_{0}>\cdots>\eta_{r})$ with
$\operatorname*{codim}\nolimits_{X}\overline{\{\eta_{i}\}}=i$ throughout this
subsection. We may evaluate the ad\`{e}le group $A_{X}(\triangle
,\mathcal{O}_{X})$ of Definition \ref{def_AdelesFollowingBeilinson} for this
individual flag. Unravelling the definition, it consists alternatingly of
localizations at a multiplicative set, and completions at ideals. For the sake
of the following arguments, we will introduce a notation to keep these two
steps conceptually separated $-$ this notation will not appear anywhere else
again. Namely:

\begin{definition}
\label{Def_RingsAi}Set $L_{r}:=\mathcal{O}_{\eta_{r}}$ and $C_{r}%
:=\widehat{\mathcal{O}_{\eta_{r}}}$. Inductively for $j\leq r$ let

\begin{itemize}
\item $L_{j-1}:=C_{j}[(\mathcal{O}_{\eta_{j}}-\eta_{j-1})^{-1}]\qquad\qquad
$(\textquotedblleft localization\textquotedblright)

\item $C_{j-1}:=\underset{i_{j-1}}{\underleftarrow{\lim}}\,L_{j-1}/\eta
_{j-1}^{i_{j-1}}\qquad\qquad$(\textquotedblleft completion\textquotedblright)
\end{itemize}
\end{definition}

This proceeds downward along $j$ until we reach $A_{X}(\triangle
,\mathcal{O}_{X})=C_{0}$. So this is a step-by-step description of the
formation of an ad\`{e}le completion. A detailed verification of this is given
in \cite[\S 3]{bgwGeomAnalAdeles} (which uses notation largely compatible with
ours, except for our $C_{(-)}$ being called $A_{(-)}$ in loc. cit.). The
localizations and completions are ring maps which, as affine schemes, lead to
the following sequence of flat morphisms:%
\begin{equation}
\operatorname*{Spec}A_{X}(\triangle,\mathcal{F})\rightarrow\cdots
\rightarrow\operatorname*{Spec}C_{r-1}\rightarrow\operatorname*{Spec}%
L_{r-1}\rightarrow\operatorname*{Spec}C_{r}\rightarrow\operatorname*{Spec}%
L_{r}\rightarrow X\text{.}\label{lCOA_35}%
\end{equation}
The behaviour of the prime ideals under these maps is very carefully studied
in \cite[\S 3]{MR1213064} and \cite{bgwGeomAnalAdeles}, but we will not need
more than the following:

\begin{lemma}
[{\cite[Lemma 4.5]{bgwGeomAnalAdeles}}]%
\label{linkTATE_StructureOfLocalAdelesLem}For any $i=0,\ldots,r$ we have

\begin{enumerate}
\item $C_{j}$ is a faithfully flat Noetherian $\mathcal{O}_{\eta_{j}}$-algebra.

\item The maximal ideals of the ring $C_{j}$ are precisely the primes minimal
over $\eta_{j}C_{j}$.

\item The ring $C_{j}$ is a finite product of $j$-dimensional reduced local
rings, each complete with respect to its maximal ideal.
\end{enumerate}
\end{lemma}

\subsubsection{Coherent Cousin complex}

For the sake of legibility, let us allow ourselves (just for this section) the
shorthand%
\[
H_{x}^{r}(X):=H_{x}^{r}(X,\Omega^{n})\text{,}%
\]
where $x\in X$ is any scheme point, and $n$ any fixed integer. If $R$ is a
ring, we shall also write $H_{x}^{r}(R):=H_{x}^{r}(\operatorname*{Spec}R)$. We
may now consider the coherent Cousin complex of the scheme $X$ for the
coherent sheaf $\Omega^{n}$, i.e. with the above shorthand%
\begin{equation}
Cous^{\bullet}(X):\cdots\overset{d}{\longrightarrow}\coprod_{x_{r-2}\in
X^{r-2}}H_{x_{r-2}}^{r-2}(X)\overset{d}{\longrightarrow}\coprod_{x_{r-1}\in
X^{r-1}}H_{x_{r-1}}^{r-1}(X)\overset{d}{\longrightarrow}\coprod_{x_{r}\in
X^{r}}H_{x}^{r}(X)\longrightarrow0\text{.}\label{lCOA_38}%
\end{equation}
We write $d$ for its differential and $d_{\ast}^{\ast}$ for the components of
$d$ among the individual direct summands, as in%
\[
d=\sum_{x_{r},x_{r+1}}(d_{x_{r+1}}^{x_{r}}:H_{x_{r}}^{r}(X)\longrightarrow
H_{x_{r+1}}^{r+1}(X))\text{.}%
\]
We proceed as follows: For a flat morphism $f:X\rightarrow Y$ of schemes we
know by Proposition \ref{Prop_FlatPullbacks} that there is an induced pullback
of coherent Cousin complexes $f^{\ast}:\Gamma(Y,Cous^{\bullet}(Y))\rightarrow
\Gamma(X,Cous^{\bullet}(X))$ and even better, we understand in a very precise
way the induced morphisms between the individual direct summands appearing in
Equation \ref{lCOA_38} (see again Proposition \ref{Prop_FlatPullbacks} for
details). For the flag $\triangle$ with $\eta_{i}\in X^{i}$ that we had fixed,
we may consider the diagram consisting only of the summands of Equation
\ref{lCOA_38} for $x_{r}:=\eta_{r}$ (and the morphisms between them instead of
$d$ being just the respective component $d_{\eta_{r+1}}^{\eta_{r}}$). This
yields a diagram, call it $Q_{X}$,%
\[
Q_{X}:\cdots\longrightarrow H_{\eta_{r-2}}^{r-2}(X)\longrightarrow
H_{\eta_{r-1}}^{r-1}(X)\longrightarrow H_{\eta_{r}}^{r}(X)\longrightarrow
0\text{.}%
\]
(Of course this will \textit{not} be a complex anymore; there is no reason the
composition of individual $d_{\ast}^{\ast}$ should be zero). Since $f^{\ast}$
commutes with the differential $d$ of the Cousin complex, the components
$d_{\eta_{r+1}}^{\eta_{r}}$ individually also commute with $f^{\ast}$.
Therefore $f^{\ast}$ induces also a flat pullback between the diagrams of
shape very much like $Q_{X}$, namely%
\[%
\begin{array}
[c]{rrcccccc}%
Q^{\prime}: & \cdots\longrightarrow & \coprod_{x_{r-2}}H_{x_{r-2}}^{r-2}(Y) &
\longrightarrow & \coprod_{x_{r-1}}H_{x_{r-1}}^{r-1}(Y) & \longrightarrow &
\coprod_{x_{r}}H_{x_{r}}^{r}(Y) & \longrightarrow0\\
&  & \uparrow &  & \uparrow &  & \uparrow & \\
Q_{Y}: & \cdots\longrightarrow & H_{\eta_{r-2}}^{r-2}(Y) & \longrightarrow &
H_{\eta_{r-1}}^{r-1}(Y) & \longrightarrow & H_{\eta_{r}}^{r}(Y) &
\longrightarrow0
\end{array}
\]
where for each $i$ the points $x_{i}$ run through the finitely many
irreducible components of the scheme-theoretic fiber $f^{-1}(\eta_{i})$ (this
is because by Proposition \ref{Prop_FlatPullbacks} the pullback $f^{\ast}$ of
the direct summands appearing in the lower row has non-zero image at most in
these direct summands of the coherent Cousin complex of $X$). For example, if
each of the points $\eta_{i}$ has precisely one pre-image under $f$, the top
row complex would literally have the shape of $Q $, but with the $\eta_{i}$
each replaced by $f^{-1}(\eta_{i})$.\medskip

Now consider the following commutative diagram (whose construction we will
explain below):%
\begin{equation}%
\begin{array}
[c]{lcccccccc}
&  & \vdots &  & \vdots &  &  &  & \\
&  & \parallel &  & \parallel &  &  &  & \\
\text{4)} & \rightarrow & \coprod H_{\eta_{r-3}}^{r-3}(C_{r-1}) & \rightarrow
& \coprod H_{\eta_{r-2}}^{r-2}(C_{r-1}) & \rightarrow & \coprod H_{\eta_{r-1}%
}^{r-1}(C_{r-1}) &  & \\
&  & \uparrow &  & \uparrow &  & \parallel &  & \\
\text{3)} & \rightarrow & \coprod H_{\eta_{r-3}}^{r-3}(L_{r-1}) & \rightarrow
& \coprod H_{\eta_{r-2}}^{r-2}(L_{r-1}) & \rightarrow & \coprod H_{\eta_{r-1}%
}^{r-1}(L_{r-1}) &  & \\
&  & \parallel &  & \parallel &  & \parallel &  & \\
\text{2)} & \rightarrow & \coprod H_{\eta_{r-3}}^{r-3}(C_{r}) & \rightarrow &
\coprod H_{\eta_{r-2}}^{r-2}(C_{r}) & \rightarrow & \coprod H_{\eta_{r-1}%
}^{r-1}(C_{r}) & \rightarrow & H_{\eta_{r}}^{r}(C_{r})\\
&  & \uparrow &  & \uparrow &  & \uparrow &  & \parallel\\
\text{1)} & \rightarrow & H_{\eta_{r-3}}^{r-3}(L_{r}) & \rightarrow &
H_{\eta_{r-2}}^{r-2}(L_{r}) & \rightarrow & H_{\eta_{r-1}}^{r-1}(L_{r}) &
\rightarrow & H_{\eta_{r}}^{r}(L_{r})\\
&  & \parallel &  & \parallel &  & \parallel &  & \parallel\\
\text{0)} & \rightarrow & H_{\eta_{r-3}}^{r-3}(X) & \rightarrow &
H_{\eta_{r-2}}^{r-2}(X) & \rightarrow & H_{\eta_{r-1}}^{r-1}(X) & \rightarrow
& H_{\eta_{r}}^{r}(X)
\end{array}
\label{lD_1}%
\end{equation}
To construct this diagram, we begin with the bottom row and work upwards. The
bottom row is a sequence of direct summands in the Cousin complex of $X$. The
rows above now result inductively from applying the flat pullback along the
respective morphisms in the chain of Equation \ref{lCOA_35}. More
precisely:\medskip\newline\textsc{(Odd Rows)} To obtain odd-indexed rows: This
is the flat pullback of the row below along the localization%
\[
L_{j}:=C_{j+1}[(\mathcal{O}_{\eta_{j+1}}-\eta_{j})^{-1}]\text{.}%
\]
The primes in such a localization correspond bijectively to those primes $P$
of $C_{j+1}$ with $P\cap(\mathcal{O}_{\eta_{j+1}}-\eta_{j})=\varnothing$.
Hence, the entire flag $\eta_{0}>\cdots>\eta_{j}$ lies also in
$\operatorname*{Spec}L_{j}$. Since the local cohomology of the row below takes
supports in $\eta_{0},\ldots,\eta_{j}$ respectively, it follows that in each
case excision (Lemma \ref{COA_Lemma_LocalCohomExcision}) guarantees that the
flat pullback induces an isomorphism, explaining the equalities
\textquotedblleft$\parallel$\textquotedblright. Note that under an open
immersion a point has at most one pre-image, so direct summands do not fiber
up into further direct summands when going upward.\medskip\newline%
\textsc{(Even Rows)} To obtain the even-indexed rows: This is the flat
pullback of the row below along the completion%
\[
C_{j}:=\underset{i_{j}}{\underleftarrow{\lim}}\,L_{j}/\eta_{j}^{i_{j}}\text{.}%
\]
Firstly, we note that in a completion a point can have several (finitely many)
pre-images; therefore several summands may appear in the new row (as indicated
in the above diagram); see Proposition \ref{Prop_FlatPullbacks} for a precise
description of the map between these direct summands. Applying Lemma
\ref{lemma_CompletionPreservesLocalCohomology} to $I=I^{\prime}:=\eta_{j}$, we
obtain that the pullback induces an isomorphism%
\[
H_{\eta_{j}}^{p}(\operatorname*{Spec}L_{j},M)\overset{\sim}{\longrightarrow
}H_{\eta_{j}}^{p}(\operatorname*{Spec}C_{j},\widehat{M})\mid_{L_{j}}\text{,}%
\]
producing the isomorphism of the right-most non-zero term with the
corresonding term in the row below.\medskip

In the above diagram we find \textquotedblleft downward staircase
steps\textquotedblright\ on the right, of the shape%
\begin{equation}%
\begin{array}
[c]{ccc}%
\coprod H_{\eta_{r-1}}^{r-1}(C_{r-1}) &  & \\
\parallel &  & \\
\coprod H_{\eta_{r-1}}^{r-1}(L_{r-1}) &  & \\
\parallel &  & \\
\coprod H_{\eta_{r-1}}^{r-1}(C_{r}) & \rightarrow & H_{\eta_{r}}^{r}%
(C_{r})\text{,}%
\end{array}
\label{l_COT_1}%
\end{equation}
for varying $r$. The arrows \textquotedblleft$\parallel$\textquotedblright%
\ are actually upward arrows coming from the flat pullback and we had seen
above that these are isomorphisms in the situation at hand. So we may run them
backwards, giving something that could be called an ad\`{e}le enrichment of
the usual boundary maps of the coherent Cousin complex. Let us give them a name:

\begin{definition}
[Ad\`{e}le boundary maps, Cousin version]%
\label{Def_AdeleBoundaryMapCousinVersion}For a flag $\eta_{0}>\cdots>\eta_{n}$
and a quasi-coherent sheaf $\mathcal{F}$ we call the morphisms%
\[
\boldsymbol{\partial}_{\eta_{r+1}}^{\eta_{r}}:H_{\eta_{r}}^{r}(C_{r}%
)\longrightarrow H_{\eta_{r+1}}^{r+1}(C_{r+1})\text{,}%
\]
i.e. in self-contained notation,%
\begin{align*}
& H_{\eta_{r}}^{r}(A(\eta_{r}>\cdots>\eta_{n},\mathcal{O}_{X}),\,A(\eta
_{r}>\cdots>\eta_{n},\mathcal{F}))\\
& \qquad\qquad\longrightarrow H_{\eta_{r+1}}^{r+1}(A(\eta_{r+1}>\cdots
>\eta_{n},\mathcal{O}_{X}),\,A(\eta_{r+1}>\cdots>\eta_{n},\mathcal{F}%
))\text{,}%
\end{align*}
the \emph{(Cousin) ad\`{e}le boundary maps}.
\end{definition}

Using the HKR\ theorem with supports and its compatibility with boundary map
on the local cohomology vs. Hochschild side, Prop.
\ref{marker_Prop_HKRWithSupport}, there is also a Hochschild counterpart of
the same map for $\mathcal{F}:=\Omega^{n}$:

\begin{definition}
[Ad\`{e}le boundary maps, Hochschild version]%
\label{Def_AdeleBoundaryMapHochschildVersion}Suppose $X$ is a smooth scheme of
pure dimension $n$. For a flag $\eta_{0}>\cdots>\eta_{n}$ with
$\operatorname*{codim}\nolimits_{X}\eta_{i}=i$ we call the morphisms%
\[
\left.  ^{HH}\boldsymbol{\partial}\right.  _{\eta_{r+1}}^{\eta_{r}}%
:HH_{n-r}^{\eta_{r}}(C_{r})\longrightarrow HH_{n-(r+1)}^{\eta_{r+1}}(C_{r+1})
\]
the \emph{(Hochschild) ad\`{e}le boundary maps}.
\end{definition}

\subsubsection{Tate realization}

As explained above, we have the concatenation of flat morphisms%
\[
\operatorname*{Spec}A_{X}(\triangle,\mathcal{F})\rightarrow\cdots
\rightarrow\operatorname*{Spec}C_{r-1}\rightarrow\operatorname*{Spec}%
L_{r-1}\rightarrow\operatorname*{Spec}C_{r}\rightarrow\operatorname*{Spec}%
L_{r}\rightarrow X\text{.}%
\]
We will now construct exact functors originating from the module categories of
the individual rings appearing along this composition, i.e. functors
$\mathsf{Mod}_{f}(R)\rightarrow(\bigstar)$, where $\mathsf{Mod}_{f}(R)$
denotes the category of finitely generated $R$-modules and \textquotedblleft%
$\bigstar$\textquotedblright\ will be suitably chosen exact categories built
from Ind-, Pro- and\ Tate objects (as recalled in
\S \ref{subsect_TateCatsViaIndDiagrams}). The basic idea is that
$A_{X}(\triangle,\mathcal{F})$ is a finite product of $n$-local fields
\cite{MR1213064}, \cite{bgwGeomAnalAdeles} and can be presented as an $n$-Tate
object in finite-dimensional $k$-vector spaces, say%
\[
A_{X}(\triangle,\mathcal{F})=\underset{\longleftarrow\alpha\longrightarrow
}{\underrightarrow{\operatorname*{colim}}\underleftarrow{\lim}\cdots
\underrightarrow{\operatorname*{colim}}\underleftarrow{\lim}}\,A_{\alpha
}\qquad\text{with}\qquad A_{\alpha}\in\mathsf{Vect}_{f}(k)\text{,}%
\]
and then there is an exact functor
\begin{align}
\mathsf{Mod}_{f}(A_{X}(\triangle,\mathcal{F}))  & \longrightarrow\left.
n\text{-}\mathsf{Tate}_{\aleph_{0}}(\mathsf{Vect}_{f}(k))\right.
\label{lsimv1}\\
M  & \longmapsto\underset{\longleftarrow\alpha\longrightarrow}%
{\underrightarrow{\operatorname*{colim}}\underleftarrow{\lim}\cdots
\underrightarrow{\operatorname*{colim}}\underleftarrow{\lim}}\,\left(
M\otimes A_{\alpha}\right)  \text{.}\nonumber
\end{align}
See \cite[\S 7.2]{TateObjectsExactCats} for details. As the rings $L_{n-r}$
arise from $r$ alternating localizations and completions, and similarly for
$C_{n-r}$, there are analogous exact functors taking values in $r$-Tate
objects. At the risk of repeating ourselves, let us unravel a bit the
structure of these analogues:

Each completion of a ring can be interpreted as a Pro-limit, given by a
projective system (as depicted below on the left), and each localization as an
Ind-limit, given by the inductive limit of finitely generated sub-modules
inside the localization (as depicted below on the right):%
\[
\widehat{R}=\underset{i}{\underleftarrow{\lim}}R/I^{i}\qquad\text{and}%
\qquad\underset{t\notin S}{\underrightarrow{\operatorname*{colim}}}\frac{1}%
{t}R=R[S^{-1}]\text{.}%
\]
Concretely, let $X/k$ is an $n$-dimensional scheme. Then, by presenting
$C_{n-r}$ resp. $L_{n-r}$ by alternating localizations and completions (as
dictated by Definition \ref{Def_RingsAi}), the analogue of the functor in line
\ref{lsimv1} yields exact functors%
\begin{align}
\mathsf{Mod}_{f}(C_{n})  & \longrightarrow\mathsf{Pro}_{\aleph_{0}}%
^{a}(\mathsf{Vect}_{f}(k))\nonumber\\
\mathsf{Mod}_{f}(L_{n-1})  & \longrightarrow\mathsf{Ind}_{\aleph_{0}}%
^{a}\mathsf{Pro}_{\aleph_{0}}^{a}(\mathsf{Vect}_{f}(k))\label{lau1}\\
\mathsf{Mod}_{f}(C_{n-1})  & \longrightarrow\mathsf{Pro}_{\aleph_{0}}%
^{a}\mathsf{Ind}_{\aleph_{0}}^{a}\mathsf{Pro}_{\aleph_{0}}^{a}(\mathsf{Vect}%
_{f}(k))\nonumber\\
& \vdots\nonumber
\end{align}
and in fact all the pairs of Ind-Pro-limits lie in the sub-category of Tate
objects so that%
\begin{align}
& \vdots\nonumber\\
\mathsf{Mod}_{f}(C_{1})  & \longrightarrow\mathsf{Pro}_{\aleph_{0}}^{a}\left.
\left(  (n-1)\text{-}\mathsf{Tate}_{\aleph_{0}}\right)  \right.
(\mathsf{Vect}_{f}(k))\nonumber\\
\mathsf{Mod}_{f}(L_{0})  & \longrightarrow\left.  \left(  n\text{-}%
\mathsf{Tate}_{\aleph_{0}}\right)  \right.  (\mathsf{Vect}_{f}(k))\label{lau2}%
\\
\mathsf{Mod}_{f}(C_{0})  & \longrightarrow\mathsf{Pro}_{\aleph_{0}}^{a}\left.
\left(  n\text{-}\mathsf{Tate}_{\aleph_{0}}\right)  \right.  (\mathsf{Vect}%
_{f}(k))\nonumber
\end{align}
and $C_{0}=A(\triangle,\mathcal{O}_{X})$ still lies in $\left.  \left(
n\text{-}\mathsf{Tate}_{\aleph_{0}}\right)  \right.  (\mathsf{Vect}_{f}(k))$
since the outermost Pro-limit is just taken over nil-thickenings of the
irreducible components/minimal primes. These Pro-limits reduce to an
eventually stationary projective system and thus already exist in the $n$-Tate
category without having to take a further category of Pro-objects. As a
result, $A(\triangle,\mathcal{O}_{X})$-modules can naturally be sent to their
associated $n$-Tate object in the category of finite-dimensional $k$-vector spaces.

\begin{remark}
The exactness of these functors can be shown step-by-step: For the inductive
systems defining Ind-objects the exactness is immediately clear, and for the
projective systems defining the Pro-objects one uses the Artin--Rees lemma. We
refer to \cite[\S 7.2]{TateObjectsExactCats}. A more detailed investigation of
such functors $\mathbf{C}_{Z}:\operatorname*{Coh}(X)\rightarrow\mathsf{Pro}%
_{\aleph_{0}}^{a}(\operatorname*{Coh}_{Z}(X))$, $\mathcal{F}\mapsto\lbrack
i\mapsto\mathcal{F}/\mathcal{J}_{Z}^{i}]$, where $\mathcal{J}_{Z}$ denotes the
ideal sheaf of $Z$ and $i\in\mathbf{Z}_{\geq1}$, is given in
\cite{bgwRelativeTateObjects}. See \cite[Prop. 3.25]{bgwRelativeTateObjects}.
\end{remark}

\begin{proposition}
\label{prop_bigcomdiag1}We obtain a commutative diagram%
\[%
\bfig\node a(0,0)[H_{\eta_{r+1}}^{r+1}(C_{r+1})]
\node b(800,0)[H_{\eta_{r+1}}^{r+1}(L_{r+1},\Omega^{n})]
\node c(1700,0)[HH_{n-r-1}^{\eta_{r+1}}(L_{r+1})]
\node d(3100,0)[HH_{n-r-1}(\left.  (n-r-1)\text{-}\mathsf{Tate}_{{\aleph}%
_0}(\mathsf{Vect}_{f}(k))\right.  ),]
\node A(0,700)[H_{\eta_{r}}^{r}(C_{r})]
\node B(800,700)[H_{\eta_{r}}^{r}(L_{r},\Omega^{n})]
\node C(1700,700)[HH_{n-r}^{\eta_{r}}(L_{r})]
\node D(3100,700)[HH_{n-r}(\left.  (n-r)\text{-}\mathsf{Tate}_{{\aleph}%
_0}(\mathsf{Vect}_{f}(k))\right.  )]
\arrow[a`b;\sim]
\arrow[b`c;\sim]
\arrow[c`d;]
\arrow[A`B;\sim]
\arrow[B`C;\sim]
\arrow[C`D;]
\arrow[A`a;]
\arrow[B`b;]
\arrow[C`c;]
\arrow[D`d;]
\efig
\]
where

\begin{enumerate}
\item the first and second downward arrows the ad\`{e}le boundary maps
$\boldsymbol{\partial}_{\eta_{r+1}}^{\eta_{r}}$ of Definition
\ref{Def_AdeleBoundaryMapCousinVersion},

\item the third downward arrow is the analogous ad\`{e}le boundary map in
Hochschild homology (i.e. the maps of Definition
\ref{Def_AdeleBoundaryMapHochschildVersion} up to the canonical isomorphism
induced from swapping $L_{r}$ with $C_{r}$),

\item the fourth downward arrow is induced from the delooping map%
\[
HH(\left.  j\text{-}\mathsf{Tate}_{\aleph_{0}}(-)\right.  )\overset{\sim
}{\longrightarrow}\Sigma HH(\left.  (j-1)\text{-}\mathsf{Tate}_{\aleph_{0}%
}(-)\right.  )\text{.}%
\]

\end{enumerate}
\end{proposition}

\begin{proof}
\textit{(Left square)} In the left-most column we consider the ad\`{e}le
boundary map as constructed in Definition
\ref{Def_AdeleBoundaryMapCousinVersion}. The relevant local cohomology groups
are invariant under the last completion (so this is Lemma
\ref{lemma_CompletionPreservesLocalCohomology}, or see Diagram \ref{l_COT_1}).
This implies the commutativity of the left-most square.\newline\textit{(Middle
square)} We use the HKR\ isomorphism with supports both on the left and on the
right and the fact that this transforms the boundary map in local cohomology
into the boundary map of the Hochschild homology localization sequence, Prop.
\ref{marker_Prop_HKRWithSupport}. As these maps are also differentials on the
$E_{1}$-page of the coherent Cousin vs. coniveau spectral sequence, we may
also directly cite Theorem \ref{marker_Thm_ComparisonOfRowsOnEOnePage}, but
unravelling its proof both results reduce to the same core.\newline%
\textit{(Right square)} We use the realization functors with values in the
relevant higher Tate categories as in lines \ref{lau1}-\ref{lau2}. Thus, the
commutativity of this square is equivalent to the fact that these realization
functors transform the localization sequence boundary map into the delooping
map of Hochschild homology. We discuss this at length in
\cite{bgwRelativeTateObjects}, but see Appendix
\ref{section_Appendix_BoundaryAndRealization} for a quick overview.
\end{proof}

\begin{proposition}
\label{prop_bigcomdiag2}Pick $\mathcal{C}:=\mathsf{Vect}_{f}(k)$ and let
$A_{i}$ denote the Beilinson cubical algebra as provided by Theorem
\ref{Thm_NTateIsModuleCatWithGoodCubicalEndoAlgebra} for this choice of
$\mathcal{C}$. Then there is a canonical commutative square%
\[%
\bfig\Square(0,0)[{HH_{n-r}^{\eta_{r}}(L_{r})}`{HH_{n-r}(A_{n-r})}%
`{HH_{n-r-1}^{\eta_{r+1}}(L_{r+1})}`{ HH_{n-r-1}(A_{n-r-1})},;``d`]
\efig
\]
where the left downward arrow is as in Prop. \ref{prop_bigcomdiag1}, the right
downward arrow is the map $d$ of Definition \ref{marker_Def_map_d}.
\end{proposition}

\begin{proof}
We pick $\mathcal{C}:=\mathsf{Vect}_{f}(k)$, which is a split exact idempotent
complete abelian category with the generator $k$ (viewed as a one-dimensional
$k$-vector space). Hence, we may use our version of Morita theory and apply
Theorem \ref{Thm_NTateIsModuleCatWithGoodCubicalEndoAlgebra}. We obtain the
cubical algebras $A_{n-r}$ and $A_{n-r-1}$. By the cited theorem, there is an
exact equivalence%
\begin{equation}
\left.  n\text{-}\mathsf{Tate}_{\aleph_{0}}(\mathcal{C})\right.  \overset
{\sim}{\longrightarrow}P_{f}(A_{n})\text{,}\label{lwwe1}%
\end{equation}
inducing isomorphisms between the respective Hochschild homology groups. Next,
we use Prop. \ref{prop_bigcomdiag1} and using the isomorphisms of line
\ref{lwwe1}, we may replace the objects in the right-most column by the
Hochschild homology of the $A_{i}$. Thanks to Theorem
\ref{Thm_DeloopingBoundaryAgreesWithDMap} our diagram remains commutative if
we replace the downward arrow between these objects by the map $d$. This
results precisely in our claim.
\end{proof}

\subsection{\label{subsect_ComparisonTateBeilResidue}Comparison with the
Tate--Beilinson residue in\ Lie homology}

For every $n$-fold cubical algebra $A$ over a field $k$, Beilinson constructs
a canonical map%
\[
\phi_{Beil}:H_{n+1}^{\operatorname*{Lie}}(A_{Lie},k)\longrightarrow k\text{,}%
\]
where $A_{Lie}$ denotes the\ Lie algebra of the associative algebra $A$, i.e.
$[x,y]:=xy-yx$. This is \cite[\S 1, Lemma, (a)]{MR565095}. For $n=1$ this
functional describes a class in $H^{2}(A_{Lie},k)$, and thus a central
extension known as Tate's central extension. Although not spelled out
explicitly, it was originally constructed by Tate to have a
coordinate-independent defnition of the residue on curves in his paper
\cite{MR0227171}. See \cite{olliTateRes} for a detailed review. To connect Lie
homology with differential forms, use the square%
\begin{equation}%
\xymatrix{H_{n}({A_{Lie}},{A_{Lie}}) \ar[r]^-{\varepsilon} \ar[d]_{I^{\prime}}
& HH_{n}(A) \ar[d]^{\phi_{HH}} \\ H_{n+1}({A_{Lie}},k) \ar[r]_-{\phi_{Beil}}
& k,}%
\label{vva1}%
\end{equation}
of \cite{olliTateRes}. The map $I^{\prime}$ is a\ Lie analogue of the map $I $
in the SBI sequence of Hochschild homology, see loc. cit. Any element coming
from any commutative sub-algebra of $A$ can be lifted to the upper left
corner, showing that Beilinson's map and the abstract Hochschild symbol agree
on such elements. See loc. cit. One slogan of the present paper might
be:\medskip\newline\textsl{We show that both Tate's and Beilinson's
constructions essentially encode an iterated boundary map in Keller's
localization sequence for Hochschild homology, after iteratively cutting out
the divisors defining a saturated flag in the scheme.\medskip}

Let us turn this into a precise statement. The agreement of the Beilinson
residue with some other residue notions is already known in part, because one
can work with explicit formulae and compare them. See \cite{MR3207578},
\cite{olliTateRes}. The novelty here is the interpretation in terms,
essentially, of differentials in the ad\`{e}lic variant of the
Hochschild--Cousin complex or equivalently coherent Cousin complex:

\begin{theorem}
[Ad\`{e}le Cousin differentials via abstract Hochschild symbol]%
\label{thm_AdeleCousinDiffsViaAbstractHHSymbol}There is a commutative diagram%
\[%
\bfig\Square(0,0)[{H_{\eta_{0}}^{0}(C_{0})}`{HH_{n}(A_{n})}`{H_{\eta_{n}}%
^{n}(C_{n})}`{k,};\chi`\rho`{\phi}_{HH}`\xi]
\efig
\]
where

\begin{enumerate}
\item $\chi$ is the composition of all right-ward arrows in the top rows of
Prop. \ref{prop_bigcomdiag1} and then Prop. \ref{prop_bigcomdiag2},

\item $\rho$ is the composition of the downward arrows in the same
propositions, concatenated for $n,n-1,\ldots$ down to $0$,

\item $\xi$ is the trace of the local cohomology group of a closed point down
to $k$ (= literally the trace of an endomorphism of a finite-dimensional
$k$-vector space)

\item $\phi_{HH}$ is the abstract Hochschild symbol of Definition
\ref{Def_AbstractHochschildResidueSymbol}.
\end{enumerate}
\end{theorem}

\begin{proof}
The left downward arrow is a composition of ad\`{e}le boundary maps in the
Cousin version, Definition \ref{Def_AdeleBoundaryMapCousinVersion}. Thanks to
the HKR isomorphism with supports, in the concrete guise of Prop.
\ref{prop_bigcomdiag1}, we may isomorphically work with Hochschild homology
with supports instead, as on the left-hand side in the diagram in Prop.
\ref{prop_bigcomdiag2}. Moreover, thanks to this proposition, we might again
isomorphically replace these by maps $d$ between the Hochschild homology
groups of cubical algebras. Next, by the very definition of the abstract
Hochschild symbol (Definition \ref{Def_AbstractHochschildResidueSymbol}; or
see \cite{olliTateRes}) as a composition of all these maps $d$, we learn that
by composing the isomorphisms that we have just discussed, the left downward
arrow can be identified with the Hochschild symbol $HH_{n}(A_{n})\rightarrow
HH_{0}(A_{0})\overset{\tau}{\rightarrow}k$.
\end{proof}

\begin{theorem}
[Agreement with Tate--Beilinson Lie map]\label{thm_MainOfLastPart}Suppose
$X/k$ is a separated, finite type scheme of pure dimension $n$. Fix a flag
$\triangle=(\eta_{0}>\cdots>\eta_{n})$ with $\operatorname*{codim}%
\nolimits_{X}\overline{\{\eta_{i}\}}=i$. The Tate--Beilinson Lie homology
residue symbol%
\[
\Omega_{\operatorname*{Frac}L_{n}/k}^{n}\longrightarrow H_{n+1}((A_{n}%
)_{Lie},k)\overset{\phi_{Beil}}{\longrightarrow}k
\]
(as defined \cite[\S 1, Lemma, (b)]{MR565095}) also agrees with%
\[
\Omega_{\operatorname*{Frac}L_{n}/k}^{n}\longrightarrow HH_{n}^{\eta_{0}%
}(L_{n})\longrightarrow HH_{n}^{\eta_{0}}(C_{0})\longrightarrow HH_{n}%
(A_{n})\overset{\phi_{HH}}{\longrightarrow}k\text{.}%
\]
Here $L_{(-)},C_{(-)}$ are as in Definition \ref{Def_RingsAi}, in particular
they depend on $\triangle$.
\end{theorem}

\begin{proof}
This is very easy now. As $R:=\operatorname*{Frac}L_{n}$ is commutative, any
differential form $f_{0}\,\mathrm{d}f_{1}\wedge\cdots\wedge\mathrm{d}f_{n} $
lifts to its symmetrization $\sum(-1)^{\pi}f_{\pi(0)}\otimes f_{\pi(1)}%
\wedge\cdots\wedge f_{\pi(n)}$ in the Chevalley--Eilenberg complex describing
the Lie homology group $H_{n}(R_{Lie},R_{Lie})$. One checks that the
Chevalley--Eilenberg differential vanishes on commuting elements (this is
trivial since the latter is a linear combination of terms each of which
contains at least one commutator). Thus, by functoriality, even after mapping
$R$ into the non-commutative algebra $A_{n}$, we still have a Lie homology
cycle. Thus, we have a lift to the upper left corner in Diagram \ref{vva1} and
along with Theorem \ref{thm_AdeleCousinDiffsViaAbstractHHSymbol} this implies
the claim.
\end{proof}

\appendix

\section{\label{section_Appendix_BoundaryAndRealization}Boundary map under
localization}

We recall the following basic construction:

\begin{proposition}
[Pro Realization]\label{TMT_PropProRealization}Let $X$ be a Noetherian scheme,
$Z$ a closed subset and $U:=X\setminus Z$ the open complement. Define%
\begin{equation}
\mathbf{C}_{Z}:\operatorname*{Coh}(X)\longrightarrow\mathsf{Pro}_{\aleph_{0}%
}^{a}(\operatorname*{Coh}\nolimits_{Z}(X))\text{,}\qquad\mathcal{F}%
\longmapsto\underset{r}{\underleftarrow{\lim}}\,j_{r,\ast}j_{r}^{\ast
}\mathcal{F}\text{,}\label{lmai_38}%
\end{equation}
where

\begin{itemize}
\item $\mathcal{F}$ is an arbitrary coherent sheaf on $X$,

\item $j_{r}:Z^{(r)}\hookrightarrow X$ the closed immersion of the $r$-th
infinitesimal neighbourhood\footnote{If $\mathcal{I}_{Z}$ is the radical ideal
sheaf such that $\mathcal{O}_{X}/\mathcal{I}_{Z}$ is reduced and has support
$Z$, then $Z^{(r)}$ is the closed subscheme determined by the ideal sheaf
$\mathcal{I}_{Z}^{r}$. These sheaves also have support in $Z$.} of $Z$ as a
closed subscheme with the reduced subscheme structure, and

\item the limit $\underleftarrow{\lim}_{r}$ is understood as the admissible
Pro-diagram $\mathbf{N}\rightarrow j_{r,\ast}j_{r}^{\ast}\mathcal{F}$.
\end{itemize}

Then this defines an exact functor and it sits in the commutative diagram of
exact categories and exact functors:
\begin{equation}%
\xymatrix{ {\operatorname{Coh}\nolimits_{Z}(X)} \ar[rr] \ar
[d]^1 && {\operatorname{Coh}(X)} \ar[rr] \ar[d]^{\mathbf{C}_Z}
&& {\operatorname{Coh}(U)} \ar[d] \\ {\operatorname{Coh}\nolimits_{Z}(X)}
\ar[rr] && {{\mathsf{Pro}_{\aleph_{0}}^{a}}( \operatorname{Coh}\nolimits
_{Z}(X))} \ar[rr] && {{\mathsf{Pro}_{\aleph_{0}}^{a}} ( \operatorname
{Coh}\nolimits_{Z}(X) )}/{\operatorname{Coh}\nolimits_{Z}(X)}  }%
\label{lVia1}%
\end{equation}
The right downward arrow is induced from $\mathbf{C}_{Z}$ to the quotient
categories, in view of the natural exact equivalence $\operatorname*{Coh}%
(U)\cong\operatorname*{Coh}(X)/\operatorname*{Coh}\nolimits_{Z}(X)$.
\end{proposition}

\begin{proof}
See for example \cite[Prop. 3.23]{bgwRelativeTateObjects} and its discussion.
Alternatively, a number of statements of this type are discussed in
\cite{bgwCCSymbol}.
\end{proof}

The rows in Diagram \ref{lVia1} are exact sequences of exact categories, i.e.
on the left-hand side we have fully exact sub-categories that are left- resp.
right $s$-filtering in the middle exact categories and the right-hand side
arrows are the quotient functors to the quotient exact categories. Thus, we
obtain the commutative square%
\[%
\bfig\Square(0,0)[{HH\operatorname*{Coh}(U)}`{\Sigma HH\operatorname
*{Coh}\nolimits_{Z}(X)}`{HH\left(  \mathsf{Pro}_{\aleph_{0}}^{a}%
\left(  \operatorname*{Coh}\nolimits_{Z}(X)\right)  /\operatorname
*{Coh}\nolimits_{Z}(X)\right)}`{\Sigma HH\operatorname*{Coh}\nolimits_{Z}%
(X)};``1`]
\efig
\]
from Keller's localization sequence. Using the equivalence%
\[
HH\left(  \mathsf{Pro}_{\aleph_{0}}^{a}\left(  \operatorname*{Coh}%
\nolimits_{Z}(X)\right)  /\operatorname*{Coh}\nolimits_{Z}(X)\right)
\overset{\sim}{\longrightarrow}HH\mathsf{Tate}_{\aleph_{0}}\left(
\operatorname*{Coh}\nolimits_{Z}(X)\right)
\]
this may be rephrased as%
\[%
\bfig\Square(0,0)[{HH\operatorname*{Coh}(U)}`{\Sigma HH\operatorname
*{Coh}\nolimits_{Z}(X)}`{HH\mathsf{Tate}_{\aleph_{0}}\left(  \operatorname
*{Coh}\nolimits_{Z}X\right)}`{\Sigma HH\operatorname*{Coh}\nolimits_{Z}%
(X)}.;``1`]
\efig
\]
As a result, the boundary map of the localization sequence for Hochschild
homology of a closed-open complement $Z\hookrightarrow X\hookleftarrow U$ is
compatible with a delooping boundary coming from the delooping property of the
Tate category. From this fact, one also obtains that the differentials on the
$E_{1}$-page of the Hochschild coniveau spectral sequence are compatible with
the functor to a Tate category. This is the same argument as in the proof of
Theorem \ref{marker_Thm_ComparisonOfRowsOnEOnePage}, and simply based on the
fact that these $E_{1}$-differentials can be realized as colimits of boundary
maps of the ordinary localization sequence, see loc. cit.

\bibliographystyle{amsalpha}
\bibliography{ollinewbib}

\end{document}